\RequirePackage{amsthm}

\documentclass[pdflatex,sn-mathphys-ay]{sn-jnl}

\usepackage[utf8]{inputenc}
\usepackage[T1]{fontenc}

\usepackage{microtype}
\usepackage[all]{nowidow}
\usepackage[lastparline]{impnattypo}

\usepackage{enumitem}

\usepackage{mathtools, thmtools}

\usepackage{float}

\usepackage{multicol}

\usepackage{centernot}

\usepackage{tikz}

\usetikzlibrary{cd, arrows}
\usetikzlibrary{decorations.pathreplacing}

\usepackage[title]{appendix}

\usepackage{hyperref}
\hypersetup{
    colorlinks=true,       %
    pdfborder={0 0 0}      %
}

\usepackage{babel}

\theoremstyle{theorem}
\newtheorem{theorem}{Theorem}[section]
\newtheorem*{theorem*}{Theorem}
\newtheorem{lemma}[theorem]{Lemma}

\newtheorem{proposition}[theorem]{Proposition}
\newtheorem{corollary}[theorem]{Corollary}

\theoremstyle{definition}
\newtheorem{definition}[theorem]{Definition}
\newtheorem*{definition*}{Definition}

\newtheorem{construction}[theorem]{Construction}

\newtheorem{terminology}[theorem]{Terminology}
\newtheorem{observation}[theorem]{Observation}

\newtheorem{example}[theorem]{Example}

\newtheorem{remark}[theorem]{Remark}
\newtheorem{convention}[theorem]{Convention}

\newtheorem{warning}[theorem]{Warning}

\renewcommand{\proofname}{Proof.}

\numberwithin{equation}{section}

\usepackage{fonttable}
\usepackage{wasysym}

\usepackage{stackrel, stackengine, old-arrows}

\usepackage{mathrsfs, amsfonts, stmaryrd}
\usepackage{quiver-1.6.0} %

\usepackage{todonotes}

\newcommand{\N}{\mathbb{N}}

\DeclareFontFamily{OT1}{pzc}{}
\DeclareFontShape{OT1}{pzc}{m}{it}{<->s*[1.19]pzcmi7t}{}
\DeclareMathAlphabet{\mathpzc}{OT1}{pzc}{m}{it}
\newcommand{\catfont}[1]{\mathpzc{#1}}

\newcommand{\class}{\mathrm{d}}
\newcommand{\id}{\mathrm{id}}
\newcommand{\op}{^\mathrm{op}}

\newcommand{\psunit}{\upsilon}
\newcommand{\psmult}{\mu}

\newcommand{\profto}{\stackMath\mathrel{\stackinset{c}{-0.25ex}{c}{0.25ex}{\shortmid}{\to}}}
\newcommand{\twoto}{\Rightarrow}

\newcommand{\looseto}{\rightsquigarrow}
\newcommand{\markto}{\nrightarrow}
\newcommand{\marklooseto}{\not\looseto}
\newcommand{\discto}{\twoheadrightarrow}
\newcommand{\discdown}{{\rotatebox[origin=c]{-90}{$\discto$}}}
\newcommand{\loosedown}{{\rotatebox[origin=c]{-90}{$\looseto$}}}

\newcommand{\slice}[1]{\mathbin{\!\rotatebox[origin=c]{-90}{\({#1}\)}}}

\newcommand{\tdoarrow}[1]{#1^{\discdown}}
\newcommand{\stdoarrow}[1]{#1^{\discdown_{\!s}}}
\newcommand{\loosearrow}[1]{#1^\loosedown}

\newcommand{\lax}{\mathrm{lx}}
\newcommand{\colax}{\mathrm{cx}}
\newcommand{\pseudo}{\mathrm{ps}}
\newcommand{\strict}{\mathrm{st}}

\newcommand{\F}{\mathscr{F}}
\newcommand{\K}{\mathcal{K}}
\newcommand{\E}{\mathcal{E}}
\newcommand{\Alg}{\catfont{Alg}}
\newcommand{\PsAlg}{\catfont{Alg}^\pseudo}
\newcommand{\LaxAlg}{\catfont{Alg}^\lax}
\newcommand{\Cat}{\catfont{Cat}}
\newcommand{\Set}{\mathrm{Set}}
\newcommand{\Finset}{\mathrm{Fin}}
\newcommand{\lxel}{\mathcal{E}\ell_{\lax}}
\newcommand{\DiscTwoFib}{\catfont{Disc}2\catfont{Fib}}
\newcommand{\twocat}{\catfont{2Cat}}

\newcommand{\dblAlg}{{\mathbb{A}\mathrm{lg}}}
\newcommand{\Mod}{{\mathbb{M}\mathrm{od}}}
\newcommand{\sMod}{{\mathrm{s}\Mod}}

\usepackage[capitalise,nameinlink,noabbrev]{cleveref}

\crefname{assumption}{Assumption}{Assumption}
\crefname{diagram}{Diagram}{Diagrams}
\crefname{construction}{Construction}{Constructions}
\crefname{notation}{Notation}{Notations}

\begin{document}
\title[Classifying strict discrete opfibrations with lax morphisms]{Classifying strict discrete opfibrations\linebreak{}with lax morphisms\linebreak{
\large
 {\it In honour of Robert Par\'e on the occasion of his 80th birthday.}}}

\author[1,2]{\fnm{Matteo} \sur{Capucci}}\email{matteo.capucci@gmail.com}
\author[3]{\fnm{David Jaz} \sur{Myers}}\email{davidjaz@topos.institute}

\affil[1]{\orgdiv{Department of Computer and Information Sciences}, \orgname{University of Strathclyde}, \orgaddress{\city{Glasgow}, \state{Scotland}, \country{United Kingdom}}}

\affil[2]{\orgname{Independent Researcher}, \orgaddress{\city{Modena}, \country{Italy}}}

\affil[3]{\orgname{Topos Institute UK}, \orgaddress{\city{Oxford}, \state{England}, \country{United Kingdom}}}

\abstract{
	We study discrete opfibration classifiers in enhanced 2-categories and show how, under suitable hypotheses, such classifiers can be endowed with the structure of a (lax or pseudo-)$T$-algebra and classify strict discrete opfibrations in 2-categories of (lax or pseudo-)$T$-algebras and lax morphisms.
	This leads to a notion of discrete opfibration classifier in the enhanced setting, in which `small' (e.g.~strict) discrete opfibrations are classified by `loose' (e.g.~lax) maps.

	We identify conditions on an enhanced 2-monad $T$ and on a discrete opfibration classifier ensuring that this lifting to algebras is possible.
	These conditions hold in a broad range of examples, including double categories, monoidal and symmetric monoidal categories, orthogonal factorization systems, and, more generally, structures encoded by opfamilial 2-monads.
	In particular, this recovers and explains the role of $\mathbb{S}\mathrm{pan}(\Set)$ as a classifier for strict double discrete opfibrations via lax double functors.

	We also characterize when representable copresheaves are pseudo rather than merely lax in terms of `cartesianness at the representing object', for an abstract notion of cartesianness we introduce.
}



\maketitle

\section{Introduction}
\label[section]{sec:intro}

Bob Paré's work on double category theory has been hugely influential: together with Grandis, he kick-started the entire field, moving it from Ehresmann's curiosity to a central and powerful discipline.
In \emph{Yoneda theory for double categories}, \citeauthor{pareYonedaTheoryDouble2011} argues that \(\mathbb {S}\mathrm{pan}(\Set)\) is the correct codomain for double presheaves as the natural recipient of the representables, noting these are in general only \emph{lax} double functors.

Lambert would later bring further evidence of the adequacy of such definitions by showing that presheaves over a double category \(\mathbb {A}\) are in fact in correspondence with `double discrete opfibrations' over \(\mathbb {A}\): this is (the dual of) Theorem 2.27 in \cite{lambertDiscreteDoubleFibrations2021}, which, formally, exhibits an equivalence of categories
\begin{equation}
\label{mc-003A}
	\catfont{DiscOpfib}(\mathbb {A}) \simeq \catfont{Dbl}_{\lax}\!(\mathbb {A},\, \mathbb {S}\mathrm{pan}(\Set))
\end{equation}
where on the left we have the category of discrete double opfibrations in the sense of Lambert (Definition 2.12 in \cite{lambertDiscreteDoubleFibrations2021}) and on the right the hom-category of lax functors and tight natural transformations of (weak) double categories.
We note that Lambert's double discrete opfibrations are \emph{strict} double functors; they are, in fact, \emph{strict} discrete opfibrations in the 2-category of double categories, lax functors, and tight transformations.

Paré \cite[p.437]{pareYonedaTheoryDouble2011} already notices that such a result means \(\mathbb {S}\mathrm{pan}(\Set)\) is almost a \emph{discrete opfibration classifier}, or \emph{2-classifiers}, in the sense of \cite{weberYonedaStructures2toposes2007}, in the 2-category of double categories and lax functors.

Discrete opfibrations and their 2-classifiers are categorifications of subobjects and their classifiers: where a subobject classifier is a subobject which is universal in the sense that every other subobject appears as pullback of it, a discrete opfibration classifier is a discrete opfibration $u : \Omega_{\bullet} \to \Omega$ which is universal in the same sense.
The ur-example of a discrete opfibration classifier is the projection $u : \Set_{\bullet} \to \Set$ from the category of pointed sets to the category of sets, forgetting the chosen element.%
\footnote{
	Indeed, subobjects have fibers with category level $-1$ (propositional) while discrete opfibrations have fibers with category level $0$ (discrete).
	Similarly, discrete opfibrations are the decategorification of indexed categories, by the notorious Grothendieck construction.
}

In this paper we address the following question: \textbf{why does \(\mathbb {S}\mathrm{pan}(\Set)\) \emph{specifically} act like a 2-classifier?}
Our answer stems from a few elementary observations.
As Paré already mentions \cite[437]{pareYonedaTheoryDouble2011}, double categories (of a particular unbiased flavor) are pseudoalgebras for the free category monad \(T\) on the 2-category \(\mathrm{Graph}(\Cat)= \left[\left\{\begin{tikzcd}[sep=small,cramped]
		e & v
		\arrow["t"', shift right, from=1-1, to=1-2]
		\arrow["s", shift left, from=1-1, to=1-2]
	\end{tikzcd}\right\}, \Cat\right]\) of graphs of categories.
Mesiti \cite{mesiti2classifiersDenseGenerators2025} gives a construction of discrete opfibration classifiers $\Omega$ in presheaf 2-categories; applied to graphs of categories, Mesiti's recipe gives:
\begin{equation}
	\Omega (e) = \left [ \left \{
		\begin{tikzcd}[cramped,sep=tiny]
			& {\id_e} \\
			s && t
			\arrow[from=1-2, to=2-1]
			\arrow[from=1-2, to=2-3]
		\end{tikzcd}
		\right \}, \Set
		\right ],
	\qquad
	\Omega (v) = \left [\left \{ \id_v \right \}, \Set \right ]
\end{equation}
with \(\Omega (s)\) and \(\Omega (t)\) given by projecting the feet of the spans. In other words, \(\Omega \) is the (large) graph of spans of sets, the underlying graph of the double category \(\mathbb{S}\mathrm{pan}(\Set)\).
Moreover, the universal discrete opfibration is the projection $u : \Omega_{\bullet} \to \Omega$ from spans of \emph{pointed} sets down to spans of sets.

This suspicious coincidence already suggests that some abstract nonsense is afoot, but the connection runs deeper. The double category structure on $\mathbb{S}\mathsf{pan}(\Set)$ may be deduced from the universal property of the discrete opfibration classifier $\Omega$.

The double category structure $\omega : T \Omega \to \Omega$ defining $\mathbb{S}\mathsf{pan}(\Set)$ is given by sending a sequence of spans
\begin{equation}\label{eqn:sequence.spans}
	\begin{tikzcd}[cramped, sep = small]
		& {S_1} && \cdots && {S_n} & \\
		{A_0} && {A_1} && {A_{n-1}} && {A_n}
		\arrow[from=1-2, to=2-1]
		\arrow[from=1-2, to=2-3]
		\arrow[from=1-4, to=2-3]
		\arrow[from=1-4, to=2-5]
		\arrow[from=1-6, to=2-5]
		\arrow[from=1-6, to=2-7]
	\end{tikzcd}
\end{equation}
to its limit, which is the set of elements that map to each other along the legs of the spans in the following manner:

\[\begin{tikzcd}[cramped, sep = small]
		& {s_1} && \cdots && {s_n} & \\
		{a_0} && {a_1} && {a_{n-1}} && {a_n}
		\arrow[maps to, from=1-2, to=2-1]
		\arrow[maps to, from=1-2, to=2-3]
		\arrow[maps to, from=1-4, to=2-3]
		\arrow[maps to, from=1-4, to=2-5]
		\arrow[maps to, from=1-6, to=2-5]
		\arrow[maps to, from=1-6, to=2-7]
	\end{tikzcd}\]

Notice that we may equivalently describe this set $\omega(S_1, \cdots, S_n)$ as the set of sequences of spans of \emph{pointed sets}
\[
	\begin{tikzcd}[cramped, sep = small]
		& {(S_1, s_1)} && \cdots && {(S_n, s_n)} & \\
		{(A_0, a_0)} && {(A_1, a_1)} && {(A_{n-1}, a_{n-1})} && {(A_n, a_n)}
		\arrow[from=1-2, to=2-1]
		\arrow[from=1-2, to=2-3]
		\arrow[from=1-4, to=2-3]
		\arrow[from=1-4, to=2-5]
		\arrow[from=1-6, to=2-5]
		\arrow[from=1-6, to=2-7]
	\end{tikzcd}
\]
whose underlying sequence of spans of sets is \cref{eqn:sequence.spans}. In other words, the double category structure $\omega : T\Omega \to \Omega$ on spans of sets is, in the 2-category of graphs of categories, the map classifying the discrete opfibration $Tu : T\Omega_{\bullet} \to T\Omega$ which projects out the underlying sequence of spans of sets of a sequence of spans of pointed sets.

This leads to a natural conjecture that the double category structure of \(\mathbb {S}\mathrm{pan}(\Set)\) arises naturally as but one of many instances of a general phenomenon whereby a 2-classifier \emph{lifts} to algebras of a 2-monad. In this work, we give conditions under which such a lift is possible, and find that the case of double categories is indeed typical.
In a wide variety of 2-algebraic theories, spans or something like them classify \emph{strict} discrete opfibrations via lax morphisms.
To give some examples (anticipated from \cref{sec:examples}), this variety includes double categories, orthogonal factorization systems, any structure induced by an opfamilial 2-monad (such as limit completions), and monoidal versions of all of the above.

Because our theorem concerns the interplay between \emph{strict}, \emph{pseudo}, and \emph{lax} algebra morphisms, we work with enhanced 2-categories (also $\F$-categories) as introduced by \cite{lackEnhanced2categoriesLimits2012}---the theory of enhanced 2-categories was developed for just this purpose.

We set our main result in \emph{plumbuses} (\cref{def:plumbus}), enhanced 2-categories $\K$ with a distinguished choice of \emph{tight discrete opfibrations} (called \emph{small}). Concretely, we express this structure as a \emph{split enhanced discrete 2-fibration} $\partial_0 : \stdoarrow{\K} \to \K$ from an enhanced 2-category of small discrete opfibrations and commuting squares.
For $\partial_0$ to be an enhanced discrete 2-fibration means small discrete opfibrations admit so-called `left-tight' pullbacks (\cref{def:left-tight-pbs}), and that we have a strictly functorial choice of these pullbacks.

A plumbus is \emph{representable} when it admits a \emph{universal} small discrete opfibration. Explicitly, this is a \emph{marked-lax terminal object} of the enhanced 2-category of small discrete opfibrations (\cref{def:enhanced.quasi}, though see \cref{lem:marked-lax-defn} for a concrete characterization). In \cref{thm:dotted.representability}, we show a universal small discrete opfibration $u : \Omega_\bullet \to \Omega$ determines an equivalence of the enhanced 2-category of small discrete opfibrations with the lax slice $\K\slice{\looseto}_{\lax}\Omega$. In fact, \cref{thm:dotted.representability} proves the equivalence between the representability of an enhanced discrete 2-fibration and the existence of a marked-lax terminal object in general; this may be of independent interest.

We then show that enhanced 2-categories of lax and pseudo-$T$-algebras of an enhanced 2-monad $T$ on a plumbus $\K$ are again plumbuses.
The class of small discrete opfibrations we consider therein are the \emph{$T$-strict} discrete opfibrations whose underlying map is small in $\K$.

We can finally ask: when $\K$ is representable, what conditions on $T$ makes $\LaxAlg_{\lax,t\pseudo}(T)$ (or $\PsAlg_{\lax,t\pseudo}(T)$) representable as well?
In \cref{mc-0003} we prove the following conditions are sufficient:

\begin{center}
	\begin{minipage}[c]{.65\textwidth}
		\begin{enumerate}[label=\Alph*.]
			\item $Tu$ is a small discrete opfibration,
			\item $T$ preserves pullbacks of $u$,
			\item $Tu$ is perfect (\cref{mc-002Z}), i.e. $\omega$ is tight,
			\item $i$ and $m$ are cartesian\footnotemark~at $u$,
		\end{enumerate}
	\end{minipage}%
	\begin{minipage}[c]{.33\textwidth}
		\begin{equation}
			\begin{tikzcd}[ampersand replacement=\&]
				{T\Omega_\bullet} \& {\Omega_\bullet} \\
				{T\Omega} \& \Omega
				\arrow["{\omega_\bullet}", from=1-1, to=1-2]
				\arrow["Tu"', two heads, from=1-1, to=2-1]
				\arrow["\lrcorner"{anchor=center, pos=0.125}, draw=none, from=1-1, to=2-2]
				\arrow["u", two heads, from=1-2, to=2-2]
				\arrow["\omega"', dashed, from=2-1, to=2-2]
			\end{tikzcd}
		\end{equation}
	\end{minipage}
\end{center}

In particular, enhanced 2-monads satisfying the first three conditions are called \textbf{opfibrantly cartesian} (\cref{def:opfib-cart}) and they admit lifts of 2-classifiers to $\LaxAlg_{\lax,t\pseudo}(T)$, while the last condition makes them \textbf{fully} opfibrantly cartesian and induces a lift to $\PsAlg_{\lax,t\pseudo}(T)$ as well.
As observed earlier, the algebra structure on $\Omega$ is induced by the map $\omega : T\Omega \to \Omega$ classifying $Tu : T\Omega_{\bullet} \to T\Omega$.

In fact, to close the circle, we can note that the free category monad \(T\) on \(\mathrm{Graph}(\Cat)\) is \emph{cartesian} in the sense of Definition 4.1.1 of \citep{leinsterHigherOperadsHigher2004} and, since it preserves arrow objects, it preserves discrete opfibrations too; therefore it is an instance of fully opfibrantly cartesian 2-monad.
\cref{mc-0003} therefore applies to double categories, showing that the special status of $\mathbb{S}\mathsf{pan}(\Set)$ follows from a general principle of 2-categorical algebra.

\subsubsection*{Outline}
The paper begins by recalling some vocabulary from enhanced 2-category theory in \cref{sec:enhanced-2-cats}, and establishing background concepts.
In \cref{sec:disc-opfibs} we present the first original results on tight discrete opfibrations in enhanced 2-categories of algebras, specifically:
\cref{sec:enhanced-discrete-2-fibs} is then a technical section in which we develop some 2-fibrational vocabulary, specifically we identify the structure on an enhanced 2-fibration that makes it `representable', meaning suitably equivalent to a slice 2-fibration---the main result there is \cref{thm:dotted.representability}.
The main usefulness of such machinery is to enable a very slick proof of the main theorem, but we first use it in \cref{sec:plumbuses} to give a version of Weber and Mesiti's notions of (good) discrete opfibration classifier in the enhanced setting: it is a discrete opfibration which represents a chosen enhanced 2-fibration of discrete opfibrations.
Then, in \cref{sec:lifting-theorem} we prove first an `abstract' version of the main theorem, about general representors of enhanced 2-fibrations (\cref{th:lift.of.enh.quasi-terminal.to.laxalg}) and then the `concrete' one (\cref{mc-0003}) about lifting 2-classifiers to enhanced 2-categories of lax and pseudo $T$-algebras.
In \cref{sec:cartesianness}, we pause to ponder the notion of \emph{$T$-cartesian \(T\)-morphisms} for an enhanced 2-monad \(T\), a notion that generalizes \emph{descent} conditions such as extensivity (see \cref{mc-001E}) or cartesian monoidal categories.
It is of significance in this work because (small) discrete opfibrations in a representable plumbus are $T$-cartesian if and only if they are classified by tight pseudo-$T$-morphisms.

We end the paper in \cref{sec:examples} with a smattering of examples.
Our main theorem is \emph{iterable}, in the sense that, except for the 2-monad, the kind of data going in it is the same as the data coming out of it.
We may therefore apply the main theorem again to any opfibrantly cartesian 2-monad on \(\LaxAlg_{\lax, t\pseudo}(T)\).
We use this in \cref{djm-00L3} to deduce that the cartesian monoidal structure on \(\mathbb {S}\mathrm{pan}(\Set)\) classifies strict monoidal, strict double discrete opfibrations via lax monoidal, lax double functors.

\subsubsection*{Related work}
Late during the preparation of this work, we became aware of \cite{koudenburgDoubledimensionalApproachFormal2022}, where Koudenburg proves a general lifting theorem for Yoneda structures in the setting of (augmented virtual) equipments.
The result likely subsumes ours, which should be an instance of his for certain (augmented virtual) equipments of tight two-sided discrete opfibrations in a suitable 2-category---we give a more detailed account of the relationship of the two in \cref{app:koudenburg}, once all the necessary terminology has been introduced.

However, this does not mean \emph{ibid.} makes the present work useless: ours is a more elementary, if less general, approach; and since we work 2-categorically rather than double-categorically, we provide an alternative point of view, closer to the way most `2-algebra' has been developed so far.

Finally, and somewhat aspirationally, we expect a connection with \emph{oriented category theory} \cite{gepner2025orientedcategorytheory}. Our theorem has an evidently \emph{oriented} flavour: we consider lax morphisms of lax algebras and even a lax classifier. Our arguments are fairly abstract, and it is reasonable to expect that they might extend to $\omega$-categories in a oriented way---lax all the way up. We'll leave the formulation and proof of this generalization up to an enterprising higher category theorist.

\subsection*{Acknowledgments}
\label[section]{mc-0044}

The authors wish to thank Fosco Loregian, Vincent Moreau, and the participants of CT2025 for helpful conversations and comments.
We are especially indebted to Seerp Roald Koudenburg for patiently helping us with \cref{app:koudenburg}.
We are also grateful to the anonymous reviewer for spotting some gaps in the original version of this paper, prompting a significant revamping of the main argument. We appreciate their careful and detailed comments.

\subsection*{Notation and conventions}
\label[section]{mc-003M}

In the following, \(\K\) will denote a 2-category, often enhanced (see below for a definition).
Therein, objects are usually denoted by uppercase Latin letters \(A,B,C,...\), 1-cells by lowercase Latin letters \(f,g,h,...\) and 2-cells by lowercase Greek letter \(\alpha ,\beta ,...\).
Notable exception are algebra maps which are denoted by the lowercase Greek equivalent of their carrier, for ease of tracking (so the algebra structure associated to \(A\) is usually denoted \(\alpha \)), as well as co/laxator associated to maps of algebras \(f\) which are denoted as \(\overline {f}\).

We use the convention of calling any object of a 2-category \emph{category}, any 1-cell \(a:X \to A\) a \emph{(generalized) object} \(a\) of \(A\), and any 2-cell \(\phi : a \Rightarrow b : X \to A\) a \emph{(generalized) morphism} \(\phi : a \to b\) of \(A\).

We assume the reader is familiar with basic 2-category theory.
In the rest of this section, we establish some notation and terminology, as well as introduce some specialized notions.
Most of the background required to read this work can be found in \cite{lack2CategoriesCompanion2010}. %
We recall just one result:

\begin{theorem}[{Pasting lemma for commas}]\label[theorem]{mc-000F}
	Given a diagram in a 2-category \(\K\) as below, where the right square is a comma, then the whole diagram is a comma if and only if the left square is a pullback:
	\begin{equation}
		\begin{tikzcd}
		\cdot  & \cdot  & \cdot  \\
		\cdot  & \cdot  & \cdot
		\arrow[from=1-1, to=1-2]
		\arrow[from=1-1, to=2-1]
		\arrow[from=1-2, to=1-3]
		\arrow[from=1-2, to=2-2]
		\arrow["\lrcorner "{anchor=center, pos=0.125}, draw=none, from=1-2, to=2-3]
		\arrow[shorten <=4pt, shorten >=4pt, Rightarrow, from=1-3, to=2-2]
		\arrow[from=1-3, to=2-3]
		\arrow[from=2-1, to=2-2]
		\arrow[from=2-2, to=2-3]
		\end{tikzcd}
	\end{equation}
\end{theorem}

\section{Enhanced 2-categories}
\label[section]{sec:enhanced-2-cats}

\emph{Enhanced 2-categories}, introduced by \cite{lackEnhanced2categoriesLimits2012}, are 2-categories equipped with a distinguished wide sub-2-category whose 1-cells are called \emph{tight}.

We employ such 2-categories for much of the same reasons Lack and Shulman do, that is, to single out the well-behaved \emph{strict} morphisms of 2-algebras amongst (co)lax ones.

\begin{definition}
\label[definition]{mc-000S}
	An \textbf{enhanced 2-category}, or \textbf{\(\F\)-category}, is a 2-category \(\K \equiv : \K_l\) whose 1-cells are called \textbf{loose} and a wide and locally full subcategory \(J_{\K} : \K_t \hookrightarrow \K\) whose 1-cells are called \textbf{tight}.
\end{definition}

In other words, being tight is a mere property of 1-cells of an enhanced 2-category, and 2-cells between tight 1-cells are the same as the 2-cells between the same 1-cells considered as loose.
In particular, \(\K_t \hookrightarrow \K_l\) is identity-on-objects, faithful and locally fully faithful.
This is equivalent to the definition as a category enriched in the category \(\F\) of full embeddings of categories (i.e. `full subcategories'), as already noted by \cite{lackEnhanced2categoriesLimits2012}.

\begin{terminology}
\label[terminology]{term:chordate}
	An enhanced 2-category is called \textbf{chordate} when all its 1-cells are tight.
	Indeed, in what follows, when we say \emph{2-category} we always mean \emph{chordate enhanced 2-category}.
	Dually, an enhanced 2-category with only identity tight morphisms is called \textbf{inchordate}.
\end{terminology}

\begin{terminology}
\label[terminology]{term:repl-enhanced}
	When the subcategory of tights is \emph{replete}, meaning that $f$ tight and $f \cong g$ implies $g$ is tight to, we talk of \textbf{repletely enhanced} 2-category.
\end{terminology}

\begin{example}
\label[example]{mc-000Z}
	A trivial but useful example is taking \(\Cat\) with all its morphisms considered tight.
\end{example}

From the definition as enriched categories, we also get that \emph{enhanced 2-functors} are 2-functors that preserve tightness, while \emph{enhanced 2-natural transformations}, or \emph{tight natural transformations}, are 2-natural transformations whose components are all tight.

\begin{definition}[{Enhanced 2-monad}]\label[definition]{mc-0010}
	An \textbf{enhanced 2-monad} is an \(\F\)-monad, thus a 2-monad \((T, i, m)\) such that \(T\) preserves tightness and where \(i\) and \(m\) have tight components.
\end{definition}

Given an enhanced 2-monad \(T\), we can define various enhanced 2-categories of $T$-algebras, depending on the weakness of the morphisms we would like to have.
By taking suitable Eilenberg--Moore objects in the 2-category of enhanced 2-categories, one obtains enhanced 2-categories of $T$-algebras where the tights are tight (in the underlying 2-category) and strict (as \(T\)-morphisms), while the loose have the desired weakness.
See \cite{blackwellTwodimensionalMonadTheory1989,lack2CategoriesCompanion2010,lackEnhanced2categoriesLimits2012} for references on categorical 2-algebra and enhanced techniques therein.
Here we will work with more general \emph{lax} $T$-algebras, whose definition we recall now.

\begin{definition}[{Enhanced 2-category of lax \(T\)-algebras and lax morphisms}]\label[definition]{def:lax-alg}
	The \textbf{enhanced 2-category of \(T\)-algebras and lax \(T\)-morphism} \(\LaxAlg_{\lax, t\pseudo}(T)\) for an enhanced 2-monad \(T\) on the enhanced 2-category \(\K\) is the enhanced 2-category so comprised:
	\begin{enumerate}
		\item its objects are lax \(T\)-algebras, which are tight maps $\alpha:TA \to A$ equipped with unitor $\psunit$ and multiplicator $\psmult$ 2-cells:
		\begin{equation}
			\begin{tikzcd}[column sep=scriptsize]
				A && A \\
				& TA
				\arrow[""{name=0, anchor=center, inner sep=0}, equals, from=1-1, to=1-3]
				\arrow["i"', from=1-1, to=2-2]
				\arrow["\alpha"', from=2-2, to=1-3]
				\arrow["\psunit", between={0.3}{0.7}, Rightarrow, from=0, to=2-2]
			\end{tikzcd}
			\hspace*{10ex}
			\begin{tikzcd}[ampersand replacement=\&]
				{T^2A} \& TA \\
				TA \& A
				\arrow[""{name=0, anchor=center, inner sep=0}, "{{T\alpha}}", from=1-1, to=1-2]
				\arrow[""{name=0p, anchor=center, inner sep=0}, phantom, from=1-1, to=1-2, start anchor=center, end anchor=center]
				\arrow["m"', from=1-1, to=2-1]
				\arrow["\alpha", from=1-2, to=2-2]
				\arrow[""{name=1, anchor=center, inner sep=0}, "\alpha"', from=2-1, to=2-2]
				\arrow[""{name=1p, anchor=center, inner sep=0}, phantom, from=2-1, to=2-2, start anchor=center, end anchor=center]
				\arrow["\psmult"', shift left, between={0.2}{0.8}, Rightarrow, from=0p, to=1p]
			\end{tikzcd}
		\end{equation}
		satisfying the following equations:
		\begin{align}
			\begin{tikzcd}[ampersand replacement=\&, sep=scriptsize]
				\& A \\
				TA \&\& TA \& A \\
				\& {T^2A} \& TA
				\arrow["i"', from=1-2, to=2-3]
				\arrow[""{name=0, anchor=center, inner sep=0}, curve={height=-16pt}, equals, from=1-2, to=2-4]
				\arrow["\alpha", from=2-1, to=1-2]
				\arrow["i"', from=2-1, to=3-2]
				\arrow["\alpha", from=2-3, to=2-4]
				\arrow["\psmult", between={0.2}{0.8}, Rightarrow, from=2-3, to=3-3]
				\arrow["{T\alpha}", from=3-2, to=2-3]
				\arrow["m"', from=3-2, to=3-3]
				\arrow["\alpha"', from=3-3, to=2-4]
				\arrow["\psunit", shift right, between={0.2}{1}, Rightarrow, from=0, to=2-3]
			\end{tikzcd}
			\ &= \ %
			\begin{tikzcd}[ampersand replacement=\&, sep=scriptsize]
				\& A \\
				TA \&\&\& A \\
				\& {T^2A} \& TA
				\arrow[curve={height=-16pt}, equals, from=1-2, to=2-4]
				\arrow["\alpha", from=2-1, to=1-2]
				\arrow["i"', from=2-1, to=3-2]
				\arrow[curve={height=-16pt}, equals, from=2-1, to=3-3]
				\arrow["m"', from=3-2, to=3-3]
				\arrow["\alpha"', from=3-3, to=2-4]
			\end{tikzcd}
			\label{eq:lax-alg-triangle-1} \\[4ex]
			\begin{tikzcd}[ampersand replacement=\&, sep=scriptsize]
				\& {} \& TA \\
				TA \& {T^2A} \&\& A \\
				\&\& TA
				\arrow["\psunit", shift right, between={0.2}{0.9}, Rightarrow, from=1-2, to=2-2]
				\arrow["\alpha", from=1-3, to=2-4]
				\arrow["\psmult", between={0.3}{0.7}, Rightarrow, from=1-3, to=3-3]
				\arrow[curve={height=-16pt}, equals, from=2-1, to=1-3]
				\arrow["i"', from=2-1, to=2-2]
				\arrow[curve={height=16pt}, equals, from=2-1, to=3-3]
				\arrow["{T\alpha}"', from=2-2, to=1-3]
				\arrow["m"', from=2-2, to=3-3]
				\arrow["\alpha"', from=3-3, to=2-4]
			\end{tikzcd}
			\ &= \ %
			\begin{tikzcd}[ampersand replacement=\&, sep=scriptsize]
				\&\& TA \\
				TA \&\&\& A \\
				\&\& TA
				\arrow["\alpha", from=1-3, to=2-4]
				\arrow[curve={height=-16pt}, equals, from=2-1, to=1-3]
				\arrow[curve={height=16pt}, equals, from=2-1, to=3-3]
				\arrow["\alpha"', from=3-3, to=2-4]
			\end{tikzcd}
			\label{eq:lax-alg-triangle-2} \\[4ex]
			\begin{tikzcd}[ampersand replacement=\&, sep=scriptsize]
				\& {T^2A} \& TA \\
				{T^3A} \&\& TA \& A \\
				\& {T^2A} \& TA
				\arrow["T\alpha", from=1-2, to=1-3]
				\arrow["m"', from=1-2, to=2-3]
				\arrow["\psmult", between={0.1}{0.9}, Rightarrow, from=1-3, to=2-3]
				\arrow["\alpha", from=1-3, to=2-4]
				\arrow["{T^2a}", from=2-1, to=1-2]
				\arrow["m"', from=2-1, to=3-2]
				\arrow["\alpha", from=2-3, to=2-4]
				\arrow["\psmult", between={0.1}{0.9}, Rightarrow, from=2-3, to=3-3]
				\arrow["Ta", from=3-2, to=2-3]
				\arrow["m"', from=3-2, to=3-3]
				\arrow["\alpha"', from=3-3, to=2-4]
			\end{tikzcd}
			\ &= \ %
			\begin{tikzcd}[ampersand replacement=\&, sep=scriptsize]
				\& {T^2A} \& TA \\
				{T^3A} \& {T^2A} \&\& A \\
				\& {T^2A} \& TA
				\arrow["{T\alpha}", from=1-2, to=1-3]
				\arrow["{T\psmult}", shift right, between={0.2}{0.8}, Rightarrow, from=1-2, to=2-2]
				\arrow["\alpha", from=1-3, to=2-4]
				\arrow["\psmult", between={0.2}{0.8}, Rightarrow, from=1-3, to=3-3]
				\arrow["{{T^2a}}", from=2-1, to=1-2]
				\arrow["Tm", from=2-1, to=2-2]
				\arrow["m"', from=2-1, to=3-2]
				\arrow["Ta"', from=2-2, to=1-3]
				\arrow["m", from=2-2, to=3-3]
				\arrow["m"', from=3-2, to=3-3]
				\arrow["\alpha"', from=3-3, to=2-4]
			\end{tikzcd}
			\label{eq:lax-alg-pentagon}
		\end{align}
		\item its loose maps are lax \(T\)-morphisms whose underlying map is loose in \(\K\) and with structure 2-cell:
		\begin{equation}
			\begin{tikzcd}[ampersand replacement=\&]
				TA \& TB \\
				A \& B
				\arrow[""{name=0, anchor=center, inner sep=0}, "Tf", squiggly, from=1-1, to=1-2]
				\arrow["{\alpha }"', from=1-1, to=2-1]
				\arrow["{\beta }", from=1-2, to=2-2]
				\arrow[""{name=1, anchor=center, inner sep=0}, "f"', squiggly, from=2-1, to=2-2]
				\arrow["{\overline{f}}"', between={0.2}{0.8}, Rightarrow, from=0, to=1]
			\end{tikzcd}
		\end{equation}
		satisfying the following equations:
		\begin{align}
			\label{eq:lax-mor-unit}
			\begin{tikzcd}[ampersand replacement=\&,sep=scriptsize]
				\& A \&\& B \\
				TA \\[4ex]
				A \&\& B
				\arrow["f", squiggly, from=1-2, to=1-4]
				\arrow["i"', from=1-2, to=2-1]
				\arrow[""{name=0, anchor=center, inner sep=0}, curve={height=-18pt}, equals, from=1-2, to=3-1]
				\arrow[curve={height=-18pt}, equals, from=1-4, to=3-3]
				\arrow["\alpha"', from=2-1, to=3-1]
				\arrow["f"', squiggly, from=3-1, to=3-3]
				\arrow["\psunit"', between={0.2}{0.9}, Rightarrow, from=0, to=2-1]
			\end{tikzcd}
			\quad &= \quad
			\begin{tikzcd}[ampersand replacement=\&,sep=scriptsize]
				\& A \&\& B \\
				TA \&\& TB \\[4ex]
				A \&\& B
				\arrow["f", squiggly, from=1-2, to=1-4]
				\arrow["i"', from=1-2, to=2-1]
				\arrow["i"', from=1-4, to=2-3]
				\arrow[""{name=0, anchor=center, inner sep=0}, curve={height=-18pt}, equals, from=1-4, to=3-3]
				\arrow[""{name=1, anchor=center, inner sep=0}, "Tf", squiggly, from=2-1, to=2-3]
				\arrow["\alpha"', from=2-1, to=3-1]
				\arrow["\beta"', from=2-3, to=3-3]
				\arrow[""{name=2, anchor=center, inner sep=0}, "f"', squiggly, from=3-1, to=3-3]
				\arrow["\psunit"', between={0.2}{0.9}, Rightarrow, from=0, to=2-3]
				\arrow["{\overline{f}}", between={0.2}{0.8}, Rightarrow, from=1, to=2]
			\end{tikzcd}
			\\
			\label{eq:lax-mor-mult}
			\begin{tikzcd}[ampersand replacement=\&, sep=scriptsize, column sep=small]
				\& {T^2A} \&\& TA \\
				{T^2B} \\
				\& TA \&\& A \\
				TB \&\& B
				\arrow["{m}"{description}, from=1-2, to=1-4]
				\arrow[""{name=0, anchor=center, inner sep=0}, "{T^2f}"', squiggly, from=1-2, to=2-1]
				\arrow[""{name=0p, anchor=center, inner sep=0}, phantom, from=1-2, to=2-1, start anchor=center, end anchor=center]
				\arrow[""{name=1, anchor=center, inner sep=0}, "{T\alpha}"{description}, from=1-2, to=3-2]
				\arrow[""{name=1p, anchor=center, inner sep=0}, phantom, from=1-2, to=3-2, start anchor=center, end anchor=center]
				\arrow[""{name=2, anchor=center, inner sep=0}, "\alpha"{description}, from=1-4, to=3-4]
				\arrow[""{name=2p, anchor=center, inner sep=0}, phantom, from=1-4, to=3-4, start anchor=center, end anchor=center]
				\arrow["T\beta"', from=2-1, to=4-1]
				\arrow["\alpha"{description}, from=3-2, to=3-4]
				\arrow[""{name=3, anchor=center, inner sep=0}, "Tf"{description}, squiggly, from=3-2, to=4-1]
				\arrow[""{name=3p, anchor=center, inner sep=0}, phantom, from=3-2, to=4-1, start anchor=center, end anchor=center]
				\arrow[""{name=4, anchor=center, inner sep=0}, "f", squiggly, from=3-4, to=4-3]
				\arrow["\beta"{description}, from=4-1, to=4-3]
				\arrow["\psmult", between={0.2}{0.8}, Rightarrow, from=1p, to=2p]
				\arrow["{T\overline{f}}"', between={0.2}{0.8}, Rightarrow, from=0p, to=3p]
				\arrow["{\overline{f}}"', between={0.2}{0.8}, Rightarrow, from=3, to=4]
			\end{tikzcd}
			\quad &= \quad
			\begin{tikzcd}[ampersand replacement=\&, sep=scriptsize, column sep=small]
				\& {T^2A} \&\& TA \\
				{T^2B} \&\& TB \\
				\&\&\& A \\
				TB \&\& B
				\arrow["m"{description}, from=1-2, to=1-4]
				\arrow["{T^2f}"', squiggly, from=1-2, to=2-1]
				\arrow[""{name=0, anchor=center, inner sep=0}, "Tf"', squiggly, from=1-4, to=2-3]
				\arrow[""{name=0p, anchor=center, inner sep=0}, phantom, from=1-4, to=2-3, start anchor=center, end anchor=center]
				\arrow["\alpha"{description}, from=1-4, to=3-4]
				\arrow["m"{description}, from=2-1, to=2-3]
				\arrow[""{name=1, anchor=center, inner sep=0}, "{T\beta}"', from=2-1, to=4-1]
				\arrow[""{name=1p, anchor=center, inner sep=0}, phantom, from=2-1, to=4-1, start anchor=center, end anchor=center]
				\arrow[""{name=2, anchor=center, inner sep=0}, "\beta"{description}, from=2-3, to=4-3]
				\arrow[""{name=2p, anchor=center, inner sep=0}, phantom, from=2-3, to=4-3, start anchor=center, end anchor=center]
				\arrow[""{name=3, anchor=center, inner sep=0}, "f", squiggly, from=3-4, to=4-3]
				\arrow[""{name=3p, anchor=center, inner sep=0}, phantom, from=3-4, to=4-3, start anchor=center, end anchor=center]
				\arrow["\beta"{description}, from=4-1, to=4-3]
				\arrow["{\overline{f}}"', between={0.2}{0.8}, Rightarrow, from=0p, to=3p]
				\arrow["\psmult", between={0.2}{0.8}, Rightarrow, from=1p, to=2p]
			\end{tikzcd}
		\end{align}
		\item its tight maps are pseudo-\(T\)-morphisms whose underlying map is tight in \(\K\), where pseudo means the laxator is invertible:
		\begin{equation}
			\begin{tikzcd}[ampersand replacement=\&]
				TA \& TB \\
				A \& B
				\arrow[""{name=0, anchor=center, inner sep=0}, "Tf", squiggly, from=1-1, to=1-2]
				\arrow["{\alpha }"', from=1-1, to=2-1]
				\arrow["{\beta }", from=1-2, to=2-2]
				\arrow[""{name=1, anchor=center, inner sep=0}, "f"', squiggly, from=2-1, to=2-2]
				\arrow["{\overline{f}}"', "\wr", between={0.2}{0.8}, Rightarrow, from=0, to=1]
			\end{tikzcd}
		\end{equation}
		\item its 2-morphisms are \(T\)-2-morphisms, i.e. 2-cells satisfying
		\begin{equation}
			\begin{tikzcd}[ampersand replacement=\&]
				TA \&\& \\
				\&\& TB \\
				A \\
				\&\& B
				\arrow[""{name=1, anchor=center, inner sep=0}, "Tf"{pos=0.7}, curve={height=-18pt}, squiggly, from=1-1, to=2-3]
				\arrow[""{name=3, anchor=center, inner sep=0}, "f"{pos=0.7}, curve={height=-18pt}, squiggly, from=3-1, to=4-3]
				\arrow["{\overline{f}}"{pos=0.35}, between={0.2}{0.8}, Rightarrow, from=1, to=3]
				\arrow[""{name=0, anchor=center, inner sep=0}, "{Tf'}"'{pos=0.3}, curve={height=18pt}, squiggly, from=1-1, to=2-3]
				\arrow["\alpha"', from=1-1, to=3-1]
				\arrow["\beta", from=2-3, to=4-3]
				\arrow[""{name=2, anchor=center, inner sep=0}, "{f'}"'{pos=0.3}, curve={height=18pt}, squiggly, from=3-1, to=4-3]
				\arrow["{T\varphi}"', between={0.2}{0.8}, Rightarrow, from=1, to=0]
				\arrow["{\overline{f'}}"'{pos=0.65}, between={0.2}{0.8}, Rightarrow, from=0, to=2]
				\arrow["\varphi", between={0.2}{0.8}, Rightarrow, from=3, to=2]
			\end{tikzcd}
		\end{equation}
	\end{enumerate}
	Similarly, we define
	\begin{itemize}
		\item $\PsAlg_{\lax, t\pseudo}(T)$, the full subcategory of $\LaxAlg_{\lax, t\pseudo}(T)$ spanned by algebras with invertible unitor and multiplicator,
		\item $\Alg_{\lax, t\pseudo}(T)$, which is definitely likewise but restricting further to algebras with identity unitor and multiplicator.
	\end{itemize}
\end{definition}

We now define the three kinds of enhanced 2-limits we will use in this paper, chiefly featuring in the definition of \emph{plumbus} (\cref{def:plumbus}), which is the setting in which \cref{sec:lifting-theorem} takes place.

We start from a notion of pullback of tight maps along loose maps, which moreover marks the left projection as tight (this is how the \(\F\)-enrichment intervenes in the definition of enhanced 2-limit, see again \cite{lackEnhanced2categoriesLimits2012} for more details on the theory of enhanced 2-limits).

\begin{definition}[{Left-tight pullback}]\label[definition]{def:left-tight-pbs}
	A \textbf{left-tight pullback} is an enhanced 2-limit of shape
	\begin{equation}
	\label[diagram]{mc-002M}
		V^\ell  = \left \{
		\begin{tikzcd}
		& b \\
		a & c
		\arrow[from=1-2, to=2-2]
		\arrow[squiggly, from=2-1, to=2-2]
		\end{tikzcd}
		\right \}
	\end{equation}
	and weighted as defined below, i.e. where \(\Phi _l\) is the standard (conical) weight for a pullback,  \(\Phi _t\) marks the left projection as tight, and where \(\varphi \) is the unique natural transformation of its type.

	\begin{equation}
		\begin{tikzcd}[row sep=small]
			{V^\ell _l} \\
			& {\Cat} \\
			{V^\ell _t}
			\arrow[""{name=0, anchor=center, inner sep=0}, "{\Phi _l}", curve={height=-6pt}, from=1-1, to=2-2]
			\arrow["{J_{V^\ell }}"', from=1-1, to=3-1]
			\arrow[""{name=1, anchor=center, inner sep=0}, "{\Phi _t}"', curve={height=6pt}, from=3-1, to=2-2]
			\arrow["\varphi "', shift right=4, shorten <=5pt, shorten >=5pt, Rightarrow, from=0, to=1]
			\end{tikzcd}
			\qquad
			\begin{tikzcd}
			& {\mathrm{0}} \\
			{\mathrm{1}} & {\mathrm{0}} && {\mathrm{1}} \\
			&& {\mathrm{1}} & {\mathrm{1}}
			\arrow[equals, from=1-2, to=2-2]
			\arrow["{\varphi _a}"{description}, dotted, from=2-1, to=3-3]
			\arrow["{\varphi _b}"{description}, dotted, from=1-2, to=2-4]
			\arrow["{\varphi _c}"{description}, dotted, from=2-2, to=3-4]
			\arrow[equals, from=2-4, to=3-4]
			\arrow[equals, from=3-3, to=3-4]
		\end{tikzcd}
	\end{equation}
	Thus a diagram of shape \(V^\ell \) in an enhanced 2-category \(\K\) is a cospan whose right leg is tight; and its \(\Phi \)-weighted limit is given by a pullback square whose left leg is tight and \emph{detects tightness}, meaning when \(p_A h\) as below is tight, so is \(h\).
	\begin{equation}
		\begin{tikzcd}
			X & {f \times _C g} & B \\
			& A & C
			\arrow["h", squiggly, from=1-1, to=1-2]
			\arrow["{p_A h}"', curve={height=12pt}, from=1-1, to=2-2]
			\arrow["{p_B}", squiggly, from=1-2, to=1-3]
			\arrow["{p_A}"', from=1-2, to=2-2]
			\arrow["\lrcorner "{anchor=center, pos=0.125}, draw=none, from=1-2, to=2-3]
			\arrow["{f}", from=1-3, to=2-3]
			\arrow["g"', squiggly, from=2-2, to=2-3]
		\end{tikzcd}
	\end{equation}
	Notice, in particular, that any left-tight pullback is also a pullback.
\end{definition}

Next, we have the following `almost tight' flavour of comma object:

\begin{definition}[{l-rigged comma}]\label[definition]{mc-002L}
	An \textbf{l-rigged comma} is an enhanced 2-limit of the same shape as in \cref{mc-002M} and weighted as defined below, i.e. where \(\Psi _l\) is the standard weight for a comma, \(\Psi _t\) marks the left projection as tight, and where \(\psi \) is the identity:
	\begin{equation}
		\begin{tikzcd}
			{V^\ell _l} \\
			& {\Cat} \\
			{V^\ell _t}
			\arrow[""{name=0, anchor=center, inner sep=0}, "{\Psi _l}", curve={height=-6pt}, from=1-1, to=2-2]
			\arrow["{J_{V^\ell }}"', from=1-1, to=3-1]
			\arrow[""{name=1, anchor=center, inner sep=0}, "{\Psi _t}"', curve={height=6pt}, from=3-1, to=2-2]
			\arrow["\psi "', shift right=5, between={0.2}{0.8}, equals, from=0, to=1]
		\end{tikzcd}
		\qquad
		\begin{tikzcd}
			& {\mathrm{1}} \\
			{\mathrm{1}} & {\mathrm{2}} && {\mathrm{1}} \\
			&& {\mathrm{1}} & {\mathrm{2}}
			\arrow["{d_0}", from=1-2, to=2-2]
			\arrow[equals, dotted, from=1-2, to=2-4]
			\arrow["{d_1}", from=2-1, to=2-2]
			\arrow[equals, dotted, from=2-1, to=3-3]
			\arrow[equals, dotted, from=2-2, to=3-4]
			\arrow["{d_0}", from=2-4, to=3-4]
			\arrow["{d_1}"', from=3-3, to=3-4]
		\end{tikzcd}
	\end{equation}
	Thus a diagram of shape \(V^\ell \) in an enhanced 2-category \(\K\) is a cospan whose right leg is tight; and its \(\Psi \)-weighted limit is given by a comma square whose legs are tight and detect tightness, meaning that when \(\partial_1 h\) and \(\partial_0 h\) as below are tight, so is \(h\).
	\begin{equation}
		\begin{tikzcd}
			X \\
			& {f/g} & B \\
			& A & C
			\arrow["h", squiggly, from=1-1, to=2-2]
			\arrow["{\partial_0 h}", curve={height=-12pt}, from=1-1, to=2-3]
			\arrow["{\partial_1 h}"', curve={height=12pt}, from=1-1, to=3-2]
			\arrow["{\partial_0}", from=2-2, to=2-3]
			\arrow["{\partial_1}"', from=2-2, to=3-2]
			\arrow["\lrcorner "{anchor=center, pos=0.125}, draw=none, from=2-2, to=3-3]
			\arrow[between={0.2}{0.8}, Rightarrow, from=2-3, to=3-2]
			\arrow["f", from=2-3, to=3-3]
			\arrow["g"', squiggly, from=3-2, to=3-3]
		\end{tikzcd}
	\end{equation}
\end{definition}

The epithet `\emph{l-rigged}' comes from \cite{lackEnhanced2categoriesLimits2012}---dually, an \emph{r-rigged} comma is one where the tightness assumption on the legs of \cref{mc-002M} is swapped.

Finally, recall that an enhanced 2-limit is \emph{tight} just when it is weighted `in a chordate way', i.e. its shape is chordate and its weight has equal tight and loose parts---see §3.5.1 in \cite{lackEnhanced2categoriesLimits2012}.
Thus:

\begin{definition}[{Tight terminal object}]\label[definition]{mc-000U}
	A \textbf{tight terminal object} in an enhanced 2-category \(\K\) is a terminal object \(1\) in \(\K_l\) such that the unique maps \(!:X \to 1\) are tight.
\end{definition}

\section{Discrete opfibrations of $T$-algebras}
\label[section]{sec:disc-opfibs}

In this section, we establish some basic facts about discrete opfibrations in enhanced 2-categories of $T$-algebras.
We start by recalling their definition, due to \cite{arkorEnhanced2categoricalStructures2024}.

\begin{definition}[{(Tight) discrete opfibration}]\label[definition]{def:opfib}
	A \textbf{(tight) discrete opfibration} in a(n enhanced) 2-category \(\K\) is a (tight) map \(p:E \to  B\) that admits unique (\emph{cocartesian}) lifts:
	\begin{equation}
	\label{eq:lifting-problem}
		\begin{tikzcd}[sep=scriptsize]
			X & {} & E \\
			\\
			& {} & B
			\arrow["e", squiggly, from=1-1, to=1-3]
			\arrow["b"', curve={height=18pt}, squiggly, from=1-1, to=3-3]
			\arrow["\varphi "'{pos=0.35}, shift left=3, shorten <=4pt, shorten >=18pt, Rightarrow, from=1-2, to=3-2]
			\arrow["p", two heads, from=1-3, to=3-3]
		\end{tikzcd}
		\qquad = \quad
		\begin{tikzcd}[sep=scriptsize]
			X &[1ex]&[1ex] E \\
			\\
			&& B
			\arrow[squiggly, ""{name=0, anchor=center, inner sep=0}, "e", from=1-1, to=1-3]
			\arrow[squiggly, ""{name=1, anchor=center, inner sep=0}, "{\varphi_* e}"', curve={height=18pt}, from=1-1, to=1-3]
			\arrow[squiggly, "b"', curve={height=12pt}, from=1-1, to=3-3]
			\arrow["p", two heads, from=1-3, to=3-3]
			\arrow["{\exists  !}", shorten <=2pt, shorten >=2pt, Rightarrow, from=0, to=1]
		\end{tikzcd}
	\end{equation}
\end{definition}

The chief example of discrete opfibration is given by projections out of a comma:

\begin{definition}[{Representable discrete opfibration}]\label[definition]{mc-001T}
	A (tight) discrete opfibration is \textbf{representable} if it is equivalent to the projection out of a(n r-rigged) comma object, as dashed below:
	\begin{equation}
		\begin{tikzcd}
			{b/B} & X \\
			B & B
			\arrow[from=1-1, to=1-2]
			\arrow["{\partial_1}"', dashed, two heads, from=1-1, to=2-1]
			\arrow["\lrcorner "{anchor=center, pos=0.125}, draw=none, from=1-1, to=2-2]
			\arrow[shorten <=4pt, shorten >=4pt, Rightarrow, from=1-2, to=2-1]
			\arrow["b", squiggly, from=1-2, to=2-2]
			\arrow[equals, from=2-1, to=2-2]
		\end{tikzcd}
	\end{equation}
	We say \(b/B \xrightarrow {\partial_1} B\) is \textbf{represented by the object \(b:X \to B\)}.
	When \(b=\id_B\), we get the \textbf{codomain opfibration} associated to \(B\).
\end{definition}

\begin{warning}
\label[warning]{warn:representability-opfib}
	Here we do not use the terminology `representable' in the sense of e.g. \cite{riehlElements2022}, where the term just means `being a discrete opfibration', since
	\cref{def:opfib} can be repackaged extrinsically to say that $p$ is a discrete opfibration if and only if, for each $X \in \K$, the induced representable functor below is, i.e. when it is so \emph{representably}:
	\begin{equation}
		\begin{tikzcd}[ampersand replacement=\&]
			{\K(X,E)} \& {\K(X,B)}
			\arrow["{p_*}", two heads, from=1-1, to=1-2]
		\end{tikzcd}
	\end{equation}
\end{warning}

Being concerned, as we are in this work, with enhanced 2-categories of $T$-algebras and lifting 2-classifiers there, it is important to characterize discrete opfibrations therein.

\begin{proposition}
\label[proposition]{prop:tdos-in-alg-lax}
	Let \((T,i,m)\) be an enhanced 2-monad on \(\K\), let \(p:(E,\eta ) \to (B,\beta )\) be a strict \(T\)-morphism.
	Suppose \(p\) is a tight discrete opfibration in $\K$, then it is so in \(\LaxAlg_{\lax, t\pseudo}(T)\) too.
\end{proposition}
\begin{proof}
\label[proof]{mc-0032}
	Consider a lifting problem for $p$ in \(\LaxAlg_{\lax, t\pseudo}(T)\):
	\begin{equation}
		\begin{tikzcd}
		&&& X &&& E \\
		\\
		&&&&&& B \\
		TX &&& TE \\
		\\
		&&& TB
		\arrow[""{name=0, anchor=center, inner sep=0}, "e", squiggly, from=1-4, to=1-7]
		\arrow[""{name=0p, anchor=center, inner sep=0}, phantom, from=1-4, to=1-7, start anchor=center, end anchor=center]
		\arrow[""{name=1, anchor=center, inner sep=0}, "b"{description, pos=0.3}, curve={height=24pt}, squiggly, from=1-4, to=3-7]
		\arrow[""{name=1p, anchor=center, inner sep=0, pos=0.55}, phantom, from=1-4, to=3-7, start anchor=center, end anchor=center, curve={height=24pt}]
		\arrow["p"{description}, two heads, from=1-7, to=3-7]
		\arrow[""{name=2, anchor=center, inner sep=0}, "\xi "{description}, from=4-1, to=1-4]
		\arrow[""{name=3, anchor=center, inner sep=0}, "Te", squiggly, from=4-1, to=4-4]
		\arrow[""{name=3p, anchor=center, inner sep=0}, phantom, from=4-1, to=4-4, start anchor=center, end anchor=center]
		\arrow[""{name=4, anchor=center, inner sep=0}, "Tb"{description, pos=0.3}, curve={height=24pt}, squiggly, from=4-1, to=6-4]
		\arrow[""{name=4p, anchor=center, inner sep=0, pos=0.55}, phantom, from=4-1, to=6-4, start anchor=center, end anchor=center, curve={height=24pt}]
		\arrow[""{name=6, anchor=center, inner sep=0}, "\beta "{description}, from=6-4, to=3-7]
		\arrow["{\overline {b}}"{description}, shift right=5, curve={height=-18pt}, shorten <=21pt, shorten >=21pt, Rightarrow, from=6, to=2]
		\arrow["Tp"{description}, two heads, from=4-4, to=6-4]
		\arrow["\varphi "{description}, shorten <=7pt, shorten >=7pt, Rightarrow, from=0p, to=1p]
		\arrow["{T\varphi }"{description}, shorten <=7pt, shorten >=7pt, Rightarrow, from=3p, to=4p]
		\arrow[""{name=5, anchor=center, inner sep=0}, "\eta "{description}, from=4-4, to=1-7]
		\arrow["{\overline {e}}"{description}, shift left=2, shorten <=19pt, shorten >=19pt, Rightarrow, from=5, to=2]
		\end{tikzcd}
	\end{equation}
	We solve it in the back face, getting \(\lambda  : e \Rightarrow  \varphi_* e\):
	\begin{equation}
		\begin{tikzcd}
			&&& X &&& E \\
			\\
			&&&&&& B \\
			TX &&& TE \\
			\\
			&&& TB
			\arrow[""{name=0, anchor=center, inner sep=0}, "e", squiggly, from=1-4, to=1-7]
			\arrow[""{name=0p, anchor=center, inner sep=0}, phantom, from=1-4, to=1-7, start anchor=center, end anchor=center]
			\arrow[""{name=1, anchor=center, inner sep=0}, "{\varphi_* e}"', curve={height=24pt}, squiggly, from=1-4, to=1-7]
			\arrow[""{name=1p, anchor=center, inner sep=0}, phantom, from=1-4, to=1-7, start anchor=center, end anchor=center, curve={height=24pt}]
			\arrow["b"{description, pos=0.3}, curve={height=24pt}, squiggly, from=1-4, to=3-7]
			\arrow["p"{description}, two heads, from=1-7, to=3-7]
			\arrow[""{name=2, anchor=center, inner sep=0}, "\xi "{description}, from=4-1, to=1-4]
			\arrow[""{name=3, anchor=center, inner sep=0}, "Te", squiggly, from=4-1, to=4-4]
			\arrow[""{name=3p, anchor=center, inner sep=0}, phantom, from=4-1, to=4-4, start anchor=center, end anchor=center]
			\arrow[""{name=4, anchor=center, inner sep=0}, "{T(\varphi_*e)}"', curve={height=24pt}, squiggly, from=4-1, to=4-4]
			\arrow[""{name=4p, anchor=center, inner sep=0}, phantom, from=4-1, to=4-4, start anchor=center, end anchor=center, curve={height=24pt}]
			\arrow["Tb"{description, pos=0.3}, curve={height=24pt}, squiggly, from=4-1, to=6-4]
			\arrow[""{name=6, anchor=center, inner sep=0}, "\beta "{description}, from=6-4, to=3-7]
			\arrow["{\overline {b}}"{description}, shift right=5, curve={height=-18pt}, shorten <=21pt, shorten >=21pt, Rightarrow, from=6, to=2]
			\arrow["Tp"{description}, two heads, from=4-4, to=6-4]
			\arrow["{\exists!\,\lambda \,}"', shorten <=3pt, shorten >=3pt, Rightarrow, from=0p, to=1p]
			\arrow["{T\lambda \,}"', shorten <=3pt, shorten >=3pt, Rightarrow, from=3p, to=4p]
			\arrow[""{name=5, anchor=center, inner sep=0}, "\eta "{description, pos=0.445}, from=4-4, to=1-7]
			\arrow["{\overline {e}}"{description}, shift left=2, shorten <=19pt, shorten >=19pt, Rightarrow, from=5, to=2]
		\end{tikzcd}
	\end{equation}

	We must now equip \(\varphi_* e\) with the structure of a lax \(T\)-morphism and show that \(\lambda \) is a well defined \(T\)-2-morphism.
	We define \(\overline {\varphi_* e} : \eta  T(\varphi_* e) \Rightarrow  (\varphi_* e)\xi \) by appealing to the universal property of \(\eta  T\lambda \) as the lift of $\overline{b} : p(\eta T\varphi_* e) \beta Tb \to b \xi$:
	\begin{equation*}
		\begin{tikzcd}[ampersand replacement=\&]
			TX \&\& TE \& E \\
			\&\& TB \\
			\& X \&\& B
			\arrow["{T\varphi_*e}", squiggly, from=1-1, to=1-3]
			\arrow["Tb"{description}, squiggly, from=1-1, to=2-3]
			\arrow["\xi"', curve={height=12pt}, from=1-1, to=3-2]
			\arrow["\eta", from=1-3, to=1-4]
			\arrow["Tp", two heads, from=1-3, to=2-3]
			\arrow["p", two heads, from=1-4, to=3-4]
			\arrow["\beta"{description}, from=2-3, to=3-4]
			\arrow[""{name=0, anchor=center, inner sep=0}, "b"', squiggly, from=3-2, to=3-4]
			\arrow["{\overline{b}}"'{pos=0.3}, shift right=5, between={0}{0.7}, Rightarrow, from=2-3, to=0]
		\end{tikzcd}
		\qquad = \quad
		\begin{tikzcd}[ampersand replacement=\&]
			TX \&\& TE \& E \\
			\\
			\& X \&\& B
			\arrow[""{name=0, anchor=center, inner sep=0}, "{T\varphi_*e}", squiggly, from=1-1, to=1-3]
			\arrow["\xi"', curve={height=12pt}, from=1-1, to=3-2]
			\arrow["\eta", from=1-3, to=1-4]
			\arrow["p", two heads, from=1-4, to=3-4]
			\arrow[""{name=1, anchor=center, inner sep=0}, "{\varphi_* e}"', squiggly, from=3-2, to=1-4]
			\arrow["b"', squiggly, from=3-2, to=3-4]
			\arrow["{\exists!\, \overline{\varphi_*e}}"{pos=0.35}, shift right=5, between={0}{0.6}, Rightarrow, from=0, to=1]
		\end{tikzcd}
	\end{equation*}

	We now embark on a diagram chase to show that \(\lambda \) is a well-defined \(T\)-2-morphism, i.e. that \((\overline {\varphi_* e})\eta T\lambda = (\lambda \xi ) \overline {e}\).
	By uniqueness of lifts along \(p\) it suffices to show that these become equal after composing with \(p\):
	\begin{equation}
		\begin{aligned}
			p (\overline {\varphi_* e})\eta  T\lambda  &= \overline {b} (p \eta  T\lambda ) \\
			&= \overline {b} (\beta  Tp T\lambda ) \\
			&= \overline {b} (\beta  T\varphi ) \\
			&= (\varphi  \xi ) (p \overline {e}) \\
			&= (p \lambda ) \xi  (p \overline {e}) \\
			&= p((\lambda  \xi )\overline {e}).
		\end{aligned}
	\end{equation}
	To conclude, we observe $\overline{\varphi_* e}$ is a lawful laxator: the equations it must satisfy are proven by uniqueness of the lift.
	We do \eqref{eq:lax-mor-unit}, the another is analogous.
	We are thus look at the equation $\varphi_* e(\psunit_X) = (\overline{\varphi_*e} i)(\psunit_E \varphi_*e)$.
	This equation is over the corresponding one for $b$ (observe that $\psunit_E \varphi_* e$ is the lift of $\psunit_B b$ since $p$ is $T$-strict), which holds, and thus both sides are lifts of the same 2-cell, and therefore equal.
\end{proof}

Later in \cref{prop:lax-alg-plumbus} we show that these discrete opfibrations (that is, the `$T$-strict discrete opfibrations') in $\LaxAlg_{\lax, t\pseudo}(T)$ are the one amenable to classification according to our main result.

For \emph{strict} $T$-algebras, we can also tighten this last proposition:

\begin{proposition}
	\label[type]{prop:tdos-in-alg-strict}
	Let \((T,i,m)\) be an enhanced 2-monad on \(\K\), let \(p:(E,\eta ) \to (B,\beta )\) be a strict \(T\)-morphism.
	Then \(p\) is a tight discrete opfibration in $\K$ if and only if it is so in \(\Alg_{\lax, t\pseudo}(T)\) too.
\end{proposition}
\begin{proof}
	Let $U:\Alg_{\lax, t\pseudo}(T) \to \K$ be the forgetful enhanced 2-functor.
	Let $p$ be a discrete opfibration in $\Alg_{\lax, t\pseudo}(T)$, then we have the following diagram of adjunctions:
	\begin{equation}
		\begin{tikzcd}[ampersand replacement=\&]
			{\Alg_\lax(X,E)} \& {\Alg_\lax(X, B)} \\
			{\K(X,E)} \& {\K(X,B)}
			\arrow["{p_*}", two heads, from=1-1, to=1-2]
			\arrow[""{name=0, anchor=center, inner sep=0}, "{i^* U}", shift left, from=1-1, to=2-1]
			\arrow[""{name=0p, anchor=center, inner sep=0}, phantom, from=1-1, to=2-1, start anchor=center, end anchor=center, shift left]
			\arrow[""{name=1, anchor=center, inner sep=0}, "{i^* U}", shift left, from=1-2, to=2-2]
			\arrow[""{name=1p, anchor=center, inner sep=0}, phantom, from=1-2, to=2-2, start anchor=center, end anchor=center, shift left]
			\arrow[""{name=2, anchor=center, inner sep=0}, "{\eta_* T}", shift left, curve={height=-6pt}, from=2-1, to=1-1]
			\arrow[""{name=2p, anchor=center, inner sep=0}, phantom, from=2-1, to=1-1, start anchor=center, end anchor=center, shift left, curve={height=-6pt}]
			\arrow["{Up_*}", from=2-1, to=2-2]
			\arrow[""{name=3, anchor=center, inner sep=0}, "{\beta_* T}", shift left, curve={height=-6pt}, from=2-2, to=1-2]
			\arrow[""{name=3p, anchor=center, inner sep=0}, phantom, from=2-2, to=1-2, start anchor=center, end anchor=center, shift left, curve={height=-6pt}]
			\arrow["\dashv"{anchor=center}, draw=none, from=2p, to=0p]
			\arrow["\dashv"{anchor=center}, draw=none, from=3p, to=1p]
		\end{tikzcd}
	\end{equation}
	By strictness of $E$ and $B$, the units of both adjunctions are identities, and therefore exhibit $Up_*$ as a retract of $p_*$.
	By an easy argument, discrete opfibrations are closed under retracts, and by representability (Warning~\ref{warn:representability-opfib}), we can thus conclude $p$ is a discrete opfibration in $\K$.

	Conversely, it is easy to observe \cref{prop:pbs-of-tdos-in-alg-lax} above restricts to lax algebras.
\end{proof}

\begin{remark}
	The converse direction above is significant, since it shows that defining a tight discrete opfibration to be a strict \(T\)-morphism in \(\Alg_{\lax, t\pseudo}(T)\) is in fact a natural requirement.
	When \(T\) is sketchable\footnotemark, then this is also a consequence of Theorem 7.6 in \cite{arkorEnhanced2categoricalStructures2024}.
	\footnotetext{\(T\) is sketchable when \(\Alg_{\lax, t\pseudo}(T) \simeq \catfont{Mod}_{\lax}(\mathbb {T}, \K)\) for \(\mathbb {T}\) an enhanced 2-sketch (see \cite{arkorEnhanced2categoricalStructures2024} for a definition).}
\end{remark}

In conclusion, we record the following theorem, which is the first substantial result of the paper:

\begin{proposition}
\label[proposition]{prop:pbs-of-tdos-in-alg-lax}
	Let $\K$ be an enhanced 2-category and $(T,i,m)$ and enhanced 2-monad on it.
	Suppose $p:E \discto B$ is a tight discrete opfibration in \(\LaxAlg_{\lax, t\pseudo}(T)\) and that $f:A \to B$ is a lax $T$-morphism.
	Then the left-tight pullback of $p$ along $f$ exists in \(\LaxAlg_{\lax, t\pseudo}(T)\) if and only if it exists in $\K$.
\end{proposition}
\begin{proof}
	The `only if' direction is obvious, so we prove the other one.
	We start from the following diagram, summarizing the situation:
	\begin{equation}
		\begin{tikzcd}
			{Tf^*E} \\
			& {f^*TE} && TE \\
			&& {f^*E} && E \\
			& TA && TB \\
			&& A && B
			\arrow["{{\gamma_1}}"', from=1-1, to=2-2]
			\arrow["{{Tp_2}}", shift left, curve={height=-12pt}, squiggly, from=1-1, to=2-4]
			\arrow["{{Tf^*p}}"', curve={height=12pt}, from=1-1, to=4-2]
			\arrow[squiggly, from=2-2, to=2-4]
			\arrow["{(Tf)^*Tp}"{description}, from=2-2, to=4-2]
			\arrow["\lrcorner"{anchor=center, pos=0.125}, draw=none, from=2-2, to=4-4]
			\arrow["\eta", from=2-4, to=3-5]
			\arrow["Tp"{description, pos=0.7}, from=2-4, to=4-4]
			\arrow[squiggly, from=3-3, to=3-5]
			\arrow["{f^*p}"'{pos=0.7}, two heads, from=3-3, to=5-3]
			\arrow["\lrcorner"{anchor=center, pos=0.125}, draw=none, from=3-3, to=5-5]
			\arrow["p", two heads, from=3-5, to=5-5]
			\arrow["Tf"{pos=0.3}, squiggly, from=4-2, to=4-4]
			\arrow["\alpha"', from=4-2, to=5-3]
			\arrow["{{\overline{f}}}", between={0.2}{0.8}, Rightarrow, from=4-4, to=5-3]
			\arrow["\beta"{description}, from=4-4, to=5-5]
			\arrow["f"', squiggly, from=5-3, to=5-5]
		\end{tikzcd}
	\end{equation}
	The solid map $\gamma_1:Tf^*E \to f^*TE$ is the comparison map induced by the curved cone with apex $Tf^*E$ insisting on the cospan $f^*TE$ is a pullback of.
	Since $Tf^*p$ is tight by enhancedness of $T$ and the left projection of left-tight pullbacks detects tightness (\cref{def:left-tight-pbs}) and $f^*Tp \circ \gamma_1$, we conclude $\gamma_1$ is tight.

	Our first goal is to construct an map $f^*TE \to f^*E$, so that by composition with $\gamma_1$ we get an algebra structure on $f^*E$.
	We start by lifting $\overline{f}$ as follows, using the opfibrancy of $p$:
	\begin{equation}
	\label{eq:lift-of-laxator-f}
		\begin{tikzcd}[ampersand replacement=\&]
			{f^*TE} \&\& TE \\
			\&\&\& E \\
			TA \&\& TB \\
			\& A \&\& B
			\arrow[squiggly, from=1-1, to=1-3]
			\arrow[""{name=0, anchor=center, inner sep=0}, "l"{description, pos=0.4}, curve={height=12pt}, squiggly, from=1-1, to=2-4]
			\arrow["(Tf)^*p"', from=1-1, to=3-1]
			\arrow["\lrcorner"{anchor=center, pos=0.125}, draw=none, from=1-1, to=3-3]
			\arrow["\eta", from=1-3, to=2-4]
			\arrow["Tp"{description, pos=0.7}, from=1-3, to=3-3]
			\arrow["p", two heads, from=2-4, to=4-4]
			\arrow["Tf", squiggly, from=3-1, to=3-3]
			\arrow["\alpha"', from=3-1, to=4-2]
			\arrow["{\overline{f}}", between={0.2}{0.8}, Rightarrow, from=3-3, to=4-2]
			\arrow["\beta"{description}, from=3-3, to=4-4]
			\arrow["f"', squiggly, from=4-2, to=4-4]
			\arrow["{\exists!\ \lambda}"', between={0.2}{0.8}, Rightarrow, from=1-3, to=0]
		\end{tikzcd}
	\end{equation}
	This induces the desired map $\gamma_2$, which by the same considerations we made for $\gamma_1$, is tight:
	\begin{equation}
	\label{eq:gamma-two}
		\begin{tikzcd}
			{Tf^*E} \\
			& {f^*TE} && TE \\
			&& {f^*E} && E \\
			& TA && TB \\
			&& A && B
			\arrow["{{\gamma_1}}"', from=1-1, to=2-2]
			\arrow["{{Tp_2}}", shift left, curve={height=-12pt}, squiggly, from=1-1, to=2-4]
			\arrow["{{Tf^*p}}"', curve={height=12pt}, from=1-1, to=4-2]
			\arrow[squiggly, from=2-2, to=2-4]
			\arrow["{{\gamma_2}}"', from=2-2, to=3-3]
			\arrow[from=2-2, to=4-2]
			\arrow["\lrcorner"{anchor=center, pos=0.125}, draw=none, from=2-2, to=4-4]
			\arrow["\lambda"', between={0}{0.8}, Rightarrow, from=2-4, to=3-3]
			\arrow["\eta", from=2-4, to=3-5]
			\arrow["Tp"{description, pos=0.7}, from=2-4, to=4-4]
			\arrow[squiggly, from=3-3, to=3-5]
			\arrow["{f^*p}"'{pos=0.7}, two heads, from=3-3, to=5-3]
			\arrow["\lrcorner"{anchor=center, pos=0.125}, draw=none, from=3-3, to=5-5]
			\arrow["p", two heads, from=3-5, to=5-5]
			\arrow["Tf"{pos=0.3}, squiggly, from=4-2, to=4-4]
			\arrow["\alpha"', from=4-2, to=5-3]
			\arrow["{{\overline{f}}}", between={0.2}{0.8}, Rightarrow, from=4-4, to=5-3]
			\arrow["\beta"{description}, from=4-4, to=5-5]
			\arrow["f"', squiggly, from=5-3, to=5-5]
		\end{tikzcd}
	\end{equation}
	Then let $\eta^+ := \gamma_2\gamma_1$.
	Note that, with respect to this algebra structure, $\overline{p_2} = \lambda$ and that $f^*p$ is strict.

	The unitor and multiplicator for this algebra are obtained by lifting the ones of $A$, for example:
	\begin{equation}
		\begin{tikzcd}[ampersand replacement=\&]
			{f^*E} \&\&\& {f^*E} \\
			\& {Tf^*E} \\
			A \&\&\& A \\
			\& TA
			\arrow[""{name=0, anchor=center, inner sep=0}, shift left, equals, from=1-1, to=1-4]
			\arrow["i"', from=1-1, to=2-2]
			\arrow["{f^*p}"', two heads, from=1-1, to=3-1]
			\arrow["{f^*p}", two heads, from=1-4, to=3-4]
			\arrow["{\eta^+}"', from=2-2, to=1-4]
			\arrow[two heads, from=2-2, to=4-2]
			\arrow[""{name=1, anchor=center, inner sep=0}, equals, from=3-1, to=3-4]
			\arrow["i"', from=3-1, to=4-2]
			\arrow["\alpha"', from=4-2, to=3-4]
			\arrow["{\exists!\,\psunit_{f^*E}}"{pos=0.2}, shift right, between={0.2}{0.8}, Rightarrow, dashed, from=0, to=2-2]
			\arrow["{\psunit_A}"{pos=0.3}, shift right, between={0.2}{0.8}, Rightarrow, from=1, to=4-2]
		\end{tikzcd}
	\end{equation}
	The lax algebras laws (\cref{eq:lax-alg-triangle-1,eq:lax-alg-triangle-2,eq:lax-alg-pentagon}), being equations pertaining 2-cells all built by unique lifts, immediately follow as a consequence of them holding for $A$.

	Finally, we show $\eta^+ : Tf^*E \to f^*E$ is indeed a left-tight pullback in \(\Alg_{\lax, t\pseudo}(T)\).
	Most of the universal property follows from the same one in $\K$, so that for instance for each cone with apex $X$ as below we get a comparison map $\varphi := (a, e)$:
	\begin{equation}
		\begin{tikzcd}
			X \\
			& {f^*E} & E \\
			& A & B
			\arrow["\varphi", dashed, from=1-1, to=2-2]
			\arrow["e", shift right, curve={height=-12pt}, squiggly, from=1-1, to=2-3]
			\arrow["a"', curve={height=12pt}, squiggly, from=1-1, to=3-2]
			\arrow["{p_2}", squiggly, from=2-2, to=2-3]
			\arrow["{f^*p}"', two heads, from=2-2, to=3-2]
			\arrow["\lrcorner"{anchor=center, pos=0.125}, draw=none, from=2-2, to=3-3]
			\arrow["p", two heads, from=2-3, to=3-3]
			\arrow["f"', squiggly, from=3-2, to=3-3]
		\end{tikzcd}
	\end{equation}
	This comparison map is uniquely equipped with a laxator $\overline{\varphi}$, obtained by lifting $\overline{a}$ as follows:
	\begin{equation}
		\begin{tikzcd}[ampersand replacement=\&]
			TX \&\& {Tf^*E} \\
			\& X \&\& {f^*E} \\
			TX \&\& TA \\
			\& X \&\& A
			\arrow["{T\varphi}", squiggly, from=1-1, to=1-3]
			\arrow["\xi"{description}, from=1-1, to=2-2]
			\arrow[""{name=0, anchor=center, inner sep=0}, dashed, from=1-1, to=2-4]
			\arrow[equals, from=1-1, to=3-1]
			\arrow["{\eta^+}", from=1-3, to=2-4]
			\arrow["{Tf^*p}"{description, pos=0.7}, from=1-3, to=3-3]
			\arrow["\varphi"{description, pos=0.3}, squiggly, from=2-2, to=2-4]
			\arrow[equals, from=2-2, to=4-2]
			\arrow["{f^*p}", two heads, from=2-4, to=4-4]
			\arrow["{Ta}"{description, pos=0.3}, squiggly, from=3-1, to=3-3]
			\arrow["\xi"{description}, from=3-1, to=4-2]
			\arrow["{\overline{a}}"', between={0.2}{0.8}, Rightarrow, from=3-3, to=4-2]
			\arrow["\alpha"{description}, from=3-3, to=4-4]
			\arrow["{a}"{description}, squiggly, from=4-2, to=4-4]
			\arrow["{\exists!\ \overline{\varphi}}"', between={0}{0.8}, Rightarrow, from=1-3, to=0]
		\end{tikzcd}
	\end{equation}
	The fact that this laxator does satisfy the requirement that $f^*p(\overline{\varphi}) = \overline{a}$ follows by construction.
	As for the other, recall that $\overline{p_2} = \lambda$ from \eqref{eq:lift-of-laxator-f}, and that, since $p(e)=fa$, $p(\overline{e})=\overline{fa}$, and thus by uniqueness of the lift $\overline{e} = \overline{p_2\varphi}$.
	\begin{equation}
		\begin{tikzcd}
			TX && {Tf^*E} && TE \\
			& X && {f^*E} && E \\
			TX && TA && TB \\
			& X && A && B
			\arrow["{T\varphi}", squiggly, from=1-1, to=1-3]
			\arrow["\xi"{description}, from=1-1, to=2-2]
			\arrow[equals, from=1-1, to=3-1]
			\arrow["{Tp_2}", squiggly, from=1-3, to=1-5]
			\arrow["{\overline{\varphi}}"', between={0}{0.8}, Rightarrow, from=1-3, to=2-2]
			\arrow["{\eta^+}"{description}, from=1-3, to=2-4]
			\arrow["{Tf^*p}"{description, pos=0.7}, from=1-3, to=3-3]
			\arrow["\lambda"', between={0}{0.8}, Rightarrow, from=1-5, to=2-4]
			\arrow["\eta"{description}, from=1-5, to=2-6]
			\arrow["Tp"{description, pos=0.7}, from=1-5, to=3-5]
			\arrow["\varphi"{description, pos=0.3}, squiggly, from=2-2, to=2-4]
			\arrow[equals, from=2-2, to=4-2]
			\arrow["{p_2}"{description, pos=0.3}, squiggly, from=2-4, to=2-6]
			\arrow["{f^*p}"{description, pos=0.7}, two heads, from=2-4, to=4-4]
			\arrow["\lrcorner"{anchor=center, pos=0.125}, draw=none, from=2-4, to=4-6]
			\arrow["p"{description}, two heads, from=2-6, to=4-6]
			\arrow["Ta"{description, pos=0.3}, squiggly, from=3-1, to=3-3]
			\arrow["\xi"{description}, from=3-1, to=4-2]
			\arrow["Tf"{description, pos=0.3}, squiggly, from=3-3, to=3-5]
			\arrow["{\overline{a}}"', between={0.2}{0.8}, Rightarrow, from=3-3, to=4-2]
			\arrow["\alpha"{description}, from=3-3, to=4-4]
			\arrow["{\overline{f}}"', between={0.2}{0.8}, Rightarrow, from=3-5, to=4-4]
			\arrow["\beta"{description}, from=3-5, to=4-6]
			\arrow["a"{description}, squiggly, from=4-2, to=4-4]
			\arrow["f"{description}, squiggly, from=4-4, to=4-6]
		\end{tikzcd}
	\end{equation}
	Lawfulness can be shown by lifting the relevant equations (cf. \cref{def:lax-alg}), analogously to how we argued above---we leave this to the reader.

	It remains to show that $f^*p$ detects tightness, specifically that when \(f^*p(\varphi)\) is strict then so is \(\varphi\).
	Indeed \(f^*p(\varphi)\) strict means \(p(\overline {\varphi}) = \id_{p(\varphi)}\), and since \(p\) is discrete we must have had \(\overline {\varphi} = \id_\varphi\) to begin with.

	As for the 2-dimensional universal property, we must show that given 2-cells $u:a \twoto a':X \to A$ and $v:e \twoto e':X \to E$ such that $p(v) = f(u)$, the induced 2-cell $k:\varphi \twoto \varphi':X \to f^*E$ suitably commutes with the laxators of $\varphi$ and $\varphi'$.
	By our construction above, these maps are cocartesian and so the commutativity follows by showing $k$ is cocartesian---but that's by definition since $f^*p(k) = u$.
\end{proof}

\begin{remark}
	A similar result appears as Theorem 5.10 in \cite{capucciContextadsWreathsKleisli2024}: there it is proven that, for $(T,i,m)$ a 2-monad on a sufficiently complete 2-category, \(\Alg_{\colax}(T)\) must admit pullbacks of strict normal fibrations.
	Aside from the different in the enhancement structure, this latter result requires the base $\K$ to have commas while \cref{prop:pbs-of-tdos-in-alg-lax} does not.
	We maintain the two results are essentially equivalent, in some sense to be clarified in a different place.
\end{remark}

\section{Enhanced discrete 2-fibrations}
\label[section]{sec:enhanced-discrete-2-fibs}

We will define an enhanced discrete opfibration classifier as a representor of the 2-functor sending an object to the category of (small) discrete opfibrations over it. Accordingly, we spend a bit of time in this section adapting Lambert's theory of discrete 2-fibrations \cite{lambertDiscrete2Fibrations} to the discrete setting. We will also prove a representation theorem which is, to our knowledge, novel: a(n enhanced) discrete 2-fibration is representable if and only if it has a marked (or codotted) lax terminal object.

\subsection{Adapting Lambert's theory to the enhanced setting}
We begin by adapting the notion of discrete 2-fibration (Definition 2.4 of \cite{lambertDiscrete2Fibrations}) to the enhanced setting. While Lambert works with locally contravariant discrete 2-fibrations, we will need locally \emph{covariant} discrete 2-fibrations; for this reason, all our discrete 2-fibrations will be locally covariant.

\begin{definition}[Enhanced discrete 2-fibration]\label{defn:enhanced.disc.2.fib}
	A (split) enhanced discrete 2-fibration $P : \E \to \K$ is an enhanced 2-functor which is a (split, locally covariant) discrete 2-fibration for which the cartesian lifts of tights are tights and moreover \emph{detect tightness}. Explicitly:
	\begin{enumerate}
		\item The underlying functor of 1-categories is a split fibration, and the chosen lift of a tight map is tight.
		\item For all $X$ and $Y$, $P : \E(X, Y) \to \K(P X, P Y)$ is a discrete \emph{op}fibration.
		\item For every $f : A \to P B$ in $\K$ with lift $\lambda f : f^{\ast} B \to B$ and maps $x : X \to B$ and $y : P X \to A$ with $P x = f \circ y$, if $y$ is tight then so is the universal map $\hat{y} : X \to f^{\ast} B$ with $P \hat{y} = y$.
		\begin{equation}
			\begin{tikzcd}
				X && \\
				{P X} & {f^{\ast}B} & B \\
				& A & {P B}
				\arrow[maps to, from=1-1, to=2-1]
				\arrow[dashed, from=1-1, to=2-2]
				\arrow["x", squiggly, from=1-1, to=2-3]
				\arrow["y"', from=2-1, to=3-2]
				\arrow[squiggly, from=2-1, to=3-3]
				\arrow[squiggly, from=2-2, to=2-3]
				\arrow[maps to, from=2-2, to=3-2]
				\arrow[maps to, from=2-3, to=3-3]
				\arrow["f"', squiggly, from=3-2, to=3-3]
			\end{tikzcd}
		\end{equation}
	\end{enumerate}
\end{definition}

We adapt Construction 2.1 of \cite{lambertDiscrete2Fibrations} to the enhanced setting.

\begin{definition}[Lax 2-category of elements]\label{defn:lax.2.cat.of.elements}
	Let $\K$ be an enhanced 2-category and $F : \K\op \to \Cat$ a 2-functor. The \textbf{lax 2-category of elements} $\lxel(F)$ is the enhanced 2-category whose
	\begin{enumerate}
		\item objects are pairs $(B, X)$ with $B \in \K$ and $X \in F(B)$;
		\item arrows $(f, u) : (B, X) \to (C, Y)$ consist of a morphism $f : B \to C$ in $\K$ and a morphism $u : X \to f^{\ast}Y$ in $F(B)$, and are \textbf{tight} when $f$ is tight;
		\item 2-cells $\alpha : (f, u) \twoto (g, v)$ are 2-cells $\alpha : f \twoto g$ in $\K$ such that $\alpha^{\ast}_Y \circ u = v$ in $F(B)$.
		\begin{equation}
			\begin{tikzcd}
				& {f^{\ast}Y} \\
				X && {g^{\ast}Y}
				\arrow["u", from=2-1, to=1-2]
				\arrow["{\alpha^{\ast}_Y}", from=1-2, to=2-3]
				\arrow["v"', from=2-1, to=2-3]
			\end{tikzcd}
		\end{equation}
	\end{enumerate}
\end{definition}

The projection $P : \lxel(F) \to \K$ sending $(B, X) \mapsto B$ and $(f, u) \mapsto f$ is an enhanced discrete 2-fibration (\cref{defn:enhanced.disc.2.fib}), with splitting given by the functoriality of $F$ (as in Proposition 2.2 of \cite{lambertDiscrete2Fibrations}).

Note that in the lax slice, a map is tight if and only if its underlying map is tight.
This is in fact a general feature of enhanced discrete 2-fibrations, and it implies that an enhanced discrete 2-fibration is determined by its underlying discrete 2-fibration.

\begin{lemma}
\label{lem:characterization.of.enhanced.disc.2.fib}
	An enhanced 2-functor $P : \E \to \K$ is an enhanced discrete 2-fibration if and only if it is a discrete 2-fibration and it reflects tightness, i.e. a map $f$ in $\E$ is tight if and only if $P f$ is tight.

	As a corollary, enhanced discrete 2-fibrations over $\K$ are equivalently discrete 2-fibrations over the underlying 2-category of $\K$, enhanced by reflecting tights.
\end{lemma}
\begin{proof}
	If $f: X \to Y$ is tight if and only if $P f$ is tight, then clearly chosen cartesian lifts of tights are tight and these chosen lifts detect tightness. On the other hand, suppose that $P$ is an enhanced discrete 2-fibration and $P f$ is tight. Consider the vertical-cartesian factorization of $f$:
	\begin{equation}
		\begin{tikzcd}
			X & {f^{\ast}Y} & Y \\
			{P X} & {P X} & {P Y}
			\arrow["z", dashed, from=1-1, to=1-2]
			\arrow["f", curve={height=-18pt}, from=1-1, to=1-3]
			\arrow[maps to, from=1-1, to=2-1]
			\arrow[from=1-2, to=1-3]
			\arrow[maps to, from=1-2, to=2-2]
			\arrow["\lrcorner"{anchor=center, pos=0.125}, draw=none, from=1-2, to=2-3]
			\arrow[maps to, from=1-3, to=2-3]
			\arrow[equals, from=2-1, to=2-2]
			\arrow["{P f}"', from=2-2, to=2-3]
		\end{tikzcd}
	\end{equation}
	Since cartesian lifts of tights are tight, the lift $f^{\ast} Y \to Y$ is tight; and since cartesians lifts detect tightness, the induced map $z : X \to f^{\ast} Y$ is tight. Therefore, their composite $f$ is tight.
\end{proof}

Theorem 3.7 of \cite{lambertDiscrete2Fibrations} is easily adapted to the enhanced setting to give an equivalence between 2-functors $\K \to \Cat$ and enhanced discrete 2-fibrations over $\K$.

\begin{theorem}[{Representation theorem for enhanced discrete 2-fibrations}]
	Let $\K$ be an enhanced 2-category. Then there is an equivalence of 2-categories between the 2-category of enhanced discrete 2-fibrations over $\K$ and the 2-category of 2-functors $\K \to \Cat$:
	$$\DiscTwoFib(\K) \simeq \twocat(\K, \Cat).$$
\end{theorem}
\begin{proof}
	Using \cref{lem:characterization.of.enhanced.disc.2.fib}, this reduces to Theorem 3.7 of \cite{lambertDiscrete2Fibrations}.
\end{proof}

Note the following special case: if $X \in \K$, then the lax slice $\K \slice{\looseto}_{\lax} X \cong \lxel(\K(-, X))$ is the 2-category of elements of the 2-functor represented by $X$.

We will need the Yoneda lemma for enhanced discrete 2-fibrations. This follows quickly from Lambert's monadicity theorem (Theorem 4.13 of \cite{lambertDiscrete2Fibrations}).

\begin{lemma}[Yoneda lemma]\label{lem:2yoneda4me}
	Let $P : \E \to \K$ be an enhanced discrete 2-fibration and let $X \in \K$. Then precomposing by $\id_X : \ast \to \K \slice{\looseto}_{\lax} X$ gives an equivalence
	\begin{equation}
		(\F\Cat \slice{\to} \K)(X, P) \simeq \DiscTwoFib(\K)(\K \slice{\looseto}_{\lax} X, P).
	\end{equation}
	Moreover, this equivalence is natural in both variables.
\end{lemma}
\begin{proof}
	The slice $\K \slice{\looseto}_{\lax} X$ is the free algebra on $X : \ast \to \K$ for the monad $T$ of Theorem 4.13 of \cite{lambertDiscrete2Fibrations}; this then follows by the free-forgetful adjunction for $T$.
\end{proof}

\subsection{Marked-lax terminal objects}
The main goal of this section is to prove \cref{thm:dotted.representability}: an enhanced discrete 2-fibration is representable if and only if it has a marked-lax terminal object.
We begin to lay the groundwork for this theorem.

\begin{definition}
	A \textbf{marking} on an enhanced 2-category $\K$ is the data of a wide and locally full subcategory, whose 1-cells are called \textbf{marked} and denoted with a cut, like $\marklooseto$, $\markto$.
	A \textbf{codotting} on a marked enhanced 2-category is the data of a full subcategory of a full subcategory, whose objects are called \textbf{codotted}, such that if $B$ is codotted and $t:A \markto B$ is marked and tight, then $A$ is codotted too.
\end{definition}

There are evident (2-)categories of marked and codotted enhanced 2-categories, whose 1-cells are enhanced 2-functors preserving the marking and codotting respectively.

These definitions may seem to come a bit out of the blue; after all, we already have enhanced 2-categories and now we're adding more special 1-cells and even special objects? To this we can only answer that both decorations are very natural notions to consider when working with enhanced discrete 2-fibrations.

\begin{construction}[{Marking and codotting on lax 2-categories of elements}] \label{const:2fib.mark}
	Let $\K$ be an enhanced 2-category, and left $F : \K\op \to \F$ be an enhanced 2-functor. The enhanced lax 2-category of elements $\lxel(F)$ admits a natural marking and codotting as follows:
	\begin{equation}
		\begin{array}{ccc}
			\begin{tikzcd}
				X & {Y} \\
				B & C
				\arrow[from=1-1, to=1-2]
				\arrow[maps to, dotted, from=1-1, to=2-1]
				\arrow[maps to, dotted, from=1-2, to=2-2]
				\arrow["f"', from=2-1, to=2-2]
			\end{tikzcd}
			\qquad & \qquad
			\begin{tikzcd}
				X & Y \\
				B & C
				\arrow[from=1-1, to=1-2, squiggly]
				\arrow[maps to, dotted, from=1-1, to=2-1]
				\arrow["\lrcorner"{anchor=center, pos=0.125}, draw=none, from=1-1, to=2-2]
				\arrow[maps to, dotted, from=1-2, to=2-2]
				\arrow["f"', from=2-1, to=2-2, squiggly]
			\end{tikzcd}
			\qquad & \quad
			\begin{tikzcd}
				{\dot{Y}} \\
				C
				\arrow[maps to, dotted, from=1-1, to=2-1]
			\end{tikzcd}
			\\
			\begin{gathered}
				\textbf{tight}\\\text{tight on base}
			\end{gathered}
			\qquad & \qquad
			\begin{gathered}
				\textbf{marked}\\\text{cartesian}
			\end{gathered}
			\qquad & \quad
			\begin{gathered}
				\textbf{codotted}\\\text{tight in $F(C)$}
			\end{gathered}
		\end{array}
	\end{equation}
	We note that the codotting condition is satisfied since it is exactly the fact that $F(f)$ preserves tight objects when $f$ is tight.
\end{construction}

\begin{convention}
	If $F : \K\op \to \Cat$, then $\lxel(F)$ is marked in the same way as above, but obviously not codotted.
\end{convention}

As a special case, we have the following marking and codotting on the lax slice.

\begin{construction}[{Marking and codotting on lax slice enhanced 2-categories}] \label{const:slice.mark}
	Let $\K$ be a repletely enhanced 2-category, and let $X \in \K$ be an object. The lax slice enhanced 2-category $\K \slice{\looseto}_{\lax} X$ admits a natural marking and codotting as follows:
	\begin{equation}
		\begin{array}{ccc}
			\begin{tikzcd}
				A && B \\
				& X
				\arrow[from=1-1, to=1-3]
				\arrow[""{name=0, anchor=center, inner sep=0}, squiggly, from=1-1, to=2-2]
				\arrow[""{name=1, anchor=center, inner sep=0}, squiggly, from=1-3, to=2-2]
				\arrow[between={0.2}{0.8}, Rightarrow, shift left, from=0, to=1]
			\end{tikzcd}
			\qquad & \qquad
			\begin{tikzcd}
				A && B \\
				& X
				\arrow[squiggly, from=1-1, to=1-3]
				\arrow[""{name=0, anchor=center, inner sep=0}, squiggly, from=1-1, to=2-2]
				\arrow[""{name=1, anchor=center, inner sep=0}, squiggly, from=1-3, to=2-2]
				\arrow["\sim", between={0.2}{0.8}, Rightarrow, shift left, from=0, to=1]
			\end{tikzcd}
			\qquad & \quad
			\begin{tikzcd}
				A \\
				X
				\arrow[from=1-1, to=2-1]
			\end{tikzcd}
			\\
			\begin{gathered}
				\textbf{tight}\\\text{tight morphism on domains}
			\end{gathered}
			\qquad & \qquad
			\begin{gathered}
				\textbf{marked}\\\text{pseudo-commuting}
			\end{gathered}
			\qquad & \quad
			\begin{gathered}
				\textbf{codotted}\\\text{tight}
			\end{gathered}
		\end{array}
	\end{equation}
	We note that the codotting condition follows by the repleteness of tights.
\end{construction}

\begin{remark}
	We are being slightly weaker than what is natural in the split setting above by marking \emph{all} cartesian maps, and not just the lifts; equivalently, we mark the \emph{pseudo-}commuting triangles in the lax slice, not just the commuting ones. This subtle distinction will appear in the statement of \cref{thm:dotted.representability}, where it will correspond to the fact that the inverse to $u : \K \slice{\looseto}_{\lax} U \to \E$ is not assumed to preserve chosen lifts. We make this choice so that our definition of enhanced 2-classifier is correct with respect to Weber or Mesiti's.
\end{remark}

We now turn to the notion of marked (and codotted)-lax terminal object, which we take from \citep{gagna2022bilimitsbifinalobjects} where it is called a $M$-contraction.

Suppose that $\E$ is a marked enhanced 2-category.
Note that the projection $\E \slice{\looseto}_{\lax} X \to \E$ does not, in general, preserve the marking when $\E$ is somehow marked.
This is because every object $f : A \to X$ comes with a marked map $!_f : f \marklooseto \id_X$ in $\E \slice{\looseto}_{\lax} X$; if the domain object is not marked in $\E$, then the projection cannot preserve marked morphisms.
Rather, we must involve the marking of $\E$ in the construction of the slice:

\begin{definition}[{Marked lax slice enhanced 2-category}] \label{const:mark.mark.slice}
	Let $\E$ be a marked enhanced 2-category and let $X \in \E$. We define its \emph{marked lax slice} $\E \slice{\marklooseto}_{\lax} X$ to be the full enhanced sub-2-category of the lax slice $\E \slice{\looseto}_{\lax} X$ spanned by the morphisms $A \marklooseto X$ which are marked in $\E$ and with the following marking and codotting:
	\begin{equation}
		\begin{array}{ccc}
			\begin{tikzcd}
				A && B \\
				& X
				\arrow[from=1-1, to=1-3]
				\arrow[""{name=0, anchor=center, inner sep=0}, "/"{marking}, squiggly, from=1-1, to=2-2]
				\arrow[""{name=1, anchor=center, inner sep=0}, "/"{marking}, squiggly, from=1-3, to=2-2]
				\arrow[between={0.2}{0.8}, Rightarrow, from=0, to=1]
			\end{tikzcd}
			\quad & \quad
			\begin{tikzcd}
				A && B \\
				& X
				\arrow[squiggly, "/"{marking}, from=1-1, to=1-3]
				\arrow[""{name=0, anchor=center, inner sep=0}, "/"{marking}, squiggly, from=1-1, to=2-2]
				\arrow[""{name=1, anchor=center, inner sep=0}, "/"{marking}, squiggly, from=1-3, to=2-2]
				\arrow["\sim", between={0.2}{0.8}, Rightarrow, from=0, to=1]
			\end{tikzcd}
			\quad & \quad
			\begin{tikzcd}
				\dot{A} \\
				X
				\arrow["/"{marking}, from=1-1, to=2-1]
			\end{tikzcd}
			\\
			\begin{gathered}
				\textbf{tight}\\\text{tight}\\{}
			\end{gathered}
			\quad & \quad
			\begin{gathered}
				\textbf{marked}\\\text{marked and}\\\text{pseudo-commutative}
			\end{gathered}
			\quad & \quad
			\begin{gathered}
				\textbf{codotted}\\\text{tight and}\\\text{codotted domain}
			\end{gathered}
		\end{array}
	\end{equation}
	If $\E$ is codotted, we also require that the domain of a codotted morphism be codotted.
\end{definition}

\begin{definition}
\label{def:enhanced.quasi}
	Let $\E$ be a marked enhanced 2-category, and let $U \in \E$ be an object. We say that $U$ is \emph{marked-lax terminal} when the following commuting square of marked enhanced 2-categories admits a filler $\class$:
	\begin{equation}
		\begin{tikzcd}
			\ast & {\E \slice{\marklooseto}_{\lax} U} \\
			\E & \E
			\arrow[""{name=0, anchor=center, inner sep=0}, "{\id_U}", from=1-1, to=1-2]
			\arrow[""{name=0p, anchor=center, inner sep=0}, phantom, from=1-1, to=1-2, start anchor=center, end anchor=center]
			\arrow["U"', from=1-1, to=2-1]
			\arrow[from=1-2, to=2-2]
			\arrow[""{name=1, anchor=center, inner sep=0}, "\class"', dashed, from=2-1, to=1-2]
			\arrow[""{name=1p, anchor=center, inner sep=0}, phantom, from=2-1, to=1-2, start anchor=center, end anchor=center]
			\arrow[equals, from=2-1, to=2-2]
			\arrow["{!_U}", shift left=3, between={0.2}{0.8}, Rightarrow, from=1p, to=0p]
		\end{tikzcd}
		\quad = \quad
		\begin{tikzcd}
			\ast & {\E \slice{\marklooseto}_{\lax} U} \\
			\E & \E
			\arrow["{\id_U}", from=1-1, to=1-2]
			\arrow["U"', from=1-1, to=2-1]
			\arrow[from=1-2, to=2-2]
			\arrow[equals, from=2-1, to=2-2]
		\end{tikzcd}
	\end{equation}
	where $!_U : \class U \cong \id_U$ is a (vertical) isomorphism.

	If $\E$ is furthermore equipped with a codotting and $U$ is codotted, then we say that $U$ is \emph{codotted-lax terminal} when the above filler is a square of codotted enhanced 2-categories.
\end{definition}

There is a more down-to-earth characterization of marked-lax terminal objects.

\begin{lemma}
\label{lem:marked-lax-defn}
	Let $\E$ be a marked enhanced 2-category and let $U \in \E$. Then $U$ is a marked-lax terminal if and only if:
	\begin{enumerate}
		\item for every object $X \in \E$, there is a marked map
		\begin{equation}
			\class X : X \marklooseto U
		\end{equation}
		\item for every map $f: X \looseto Y$, there is a unique 2-cell filling the triangle below:
		\begin{equation}
			\begin{tikzcd}[ampersand replacement=\&,row sep=scriptsize]
				X \\[-1ex]
				\&\& U \\[-1ex]
				Y
				\arrow[""{name=0, anchor=center, inner sep=0}, "{\class X}", curve={height=-6pt}, squiggly, from=1-1, to=2-3]
				\arrow["f"', squiggly, from=1-1, to=3-1]
				\arrow[""{name=1, anchor=center, inner sep=0}, "{\class Y}"',squiggly, curve={height=6pt}, from=3-1, to=2-3]
				\arrow["{\exists!\ \class f}"', shift left=.5, between={0.2}{0.8}, Rightarrow, from=0, to=1]
			\end{tikzcd}
		\end{equation}
		\item for every marked map $f:X \to Y$, $\class f$ as above is invertible,
		\item We have an isomorphism $!_U : \class U \cong \id_U$.
	\end{enumerate}
	Furthermore, if $\E$ is codotted, then $U$ is codotted-lax terminal when $U$ is codotted and
	\begin{enumerate}[resume]
		\item for every $X \in \E$ codotted, $\class X$ is tight.
	\end{enumerate}
\end{lemma}
\begin{proof}
	This is a straightforward unrolling of the definitions, except for the uniqueness of condition (2). This uniqueness stands in for the 2-functoriality; it clearly implies 2-functoriality, but it is also implied by it. Suppose that $\class$ is 2-functorial $\E \to \E \slice{\marklooseto}_{\lax} U$, filling the appropriate square, and suppose $z : \class X \twoto \class Y \circ f$, seeking to show that $z = \class f$. By 2-functoriality of $\class$ applied to $z$, we have an equation of pasting diagrams (with an extra $!_U$ pasted on in the bottom right):
		\begin{equation}
		\begin{tikzcd}
			X & X & U && X & U \\
			Y &&& {=} & Y & U \\
			U & U & U && U & U
			\arrow[""{name=0, anchor=center, inner sep=0}, equals, from=1-1, to=1-2]
			\arrow["f"', from=1-1, to=2-1]
			\arrow[""{name=1, anchor=center, inner sep=0}, "{\class X}", from=1-2, to=1-3]
			\arrow["{\class X}"{description}, from=1-2, to=3-2]
			\arrow[equals, from=1-3, to=3-3]
			\arrow[""{name=2, anchor=center, inner sep=0}, "{\class X}", from=1-5, to=1-6]
			\arrow["f"', from=1-5, to=2-5]
			\arrow[equals, from=1-6, to=2-6]
			\arrow["{\class Y}"', from=2-1, to=3-1]
			\arrow[""{name=3, anchor=center, inner sep=0}, curve={height=-12pt}, from=2-5, to=2-6]
			\arrow["{\class Y}"', from=2-5, to=3-5]
			\arrow[equals, from=2-6, to=3-6]
			\arrow[""{name=4, anchor=center, inner sep=0}, equals, from=3-1, to=3-2]
			\arrow[""{name=5, anchor=center, inner sep=0}, "{\class U}", curve={height=-12pt}, from=3-2, to=3-3]
			\arrow[""{name=6, anchor=center, inner sep=0}, curve={height=12pt}, equals, from=3-2, to=3-3]
			\arrow[""{name=7, anchor=center, inner sep=0}, "{\class U}", curve={height=-12pt}, from=3-5, to=3-6]
			\arrow[""{name=8, anchor=center, inner sep=0}, curve={height=12pt}, equals, from=3-5, to=3-6]
			\arrow["z"', between={0.2}{0.8}, Rightarrow, from=0, to=4]
			\arrow["{\class(\class X)}"{description}, between={0.2}{0.8}, Rightarrow, from=1, to=5]
			\arrow["{\class f}"', between={0.2}{0.8}, Rightarrow, from=2, to=3]
			\arrow["{\class(\class Y)}"{description}, between={0.2}{0.8}, Rightarrow, from=3, to=7]
			\arrow["{!_U}", between={0.2}{0.8}, Rightarrow, from=5, to=6]
			\arrow["{!_U}", between={0.2}{0.8}, Rightarrow, from=7, to=8]
		\end{tikzcd}
	\end{equation}
	It therefore will follow that $z = \class f$ so long as $!_U \circ \class(\class X)$ and $!_U \circ \class(\class Y)$ are identity 2-cells; we'll show it for $!_U \circ \class(\class X)$, since the other case follows identically. We note that since $\class X$ is marked, $\class (\class X)$ is iso and by assumption so is $!_U$; it will therefore suffice to show that $!_U \circ \class(\class X)$ is idempotent. This follows from 2-functoriality applied to $!_U \circ \class(\class X) : \class X \twoto \class X \circ \id_X$ (with an extra $!_U$ pasted on in the bottom right):
	\begin{equation}
		\begin{tikzcd}
			X & X & U && X & U \\
			X &&& {=} & X & U \\
			U & U & U && U & U
			\arrow[""{name=0, anchor=center, inner sep=0}, equals, from=1-1, to=1-2]
			\arrow[equals, from=1-1, to=2-1]
			\arrow[""{name=1, anchor=center, inner sep=0}, "{\class X}", from=1-2, to=1-3]
			\arrow["{\class X}"{description}, from=1-2, to=3-2]
			\arrow[equals, from=1-3, to=3-3]
			\arrow["{\class X}", from=1-5, to=1-6]
			\arrow[equals, from=1-5, to=2-5]
			\arrow[equals, from=1-6, to=2-6]
			\arrow[from=2-1, to=3-1]
			\arrow[""{name=2, anchor=center, inner sep=0}, "{\class X}", curve={height=-12pt}, from=2-5, to=2-6]
			\arrow["{\class X}"', from=2-5, to=3-5]
			\arrow[equals, from=2-6, to=3-6]
			\arrow[""{name=3, anchor=center, inner sep=0}, "{\class U}", curve={height=-12pt}, from=3-1, to=3-2]
			\arrow[""{name=4, anchor=center, inner sep=0}, curve={height=12pt}, equals, from=3-1, to=3-2]
			\arrow[""{name=5, anchor=center, inner sep=0}, "{\class U}", curve={height=-12pt}, from=3-2, to=3-3]
			\arrow[""{name=6, anchor=center, inner sep=0}, curve={height=12pt}, equals, from=3-2, to=3-3]
			\arrow[""{name=7, anchor=center, inner sep=0}, "{\class U}", curve={height=-12pt}, from=3-5, to=3-6]
			\arrow[""{name=8, anchor=center, inner sep=0}, curve={height=12pt}, equals, from=3-5, to=3-6]
			\arrow["{\class(\class X)}"{description}, between={0.2}{0.8}, Rightarrow, from=0, to=3]
			\arrow["{\class (\class X)}"{description}, between={0.2}{0.8}, Rightarrow, from=1, to=5]
			\arrow["{\class(\class X)}"{description}, between={0.2}{0.8}, Rightarrow, from=2, to=7]
			\arrow["{!_U}", between={0.2}{0.8}, Rightarrow, from=3, to=4]
			\arrow["{!_U}", between={0.2}{0.8}, Rightarrow, from=5, to=6]
			\arrow["{!_U}", between={0.2}{0.8}, Rightarrow, from=7, to=8]
		\end{tikzcd}
	\end{equation}
\end{proof}

This characterization gives a straightforward corollary:

\begin{corollary}
\label{cor:slice.dotted.terminal}
	Let $\K$ be an enhanced 2-category and let $X \in \K$.
	Then $\id_X$ is a codotted-lax terminal object of the lax slice $\K \slice{\looseto}_{\lax} X$.
\end{corollary}
\begin{proof}
	For every $a : A \looseto X$, we clearly have a marked morphism $\class a: a \to \id_X$ given by $a$ itself. For every lax triangle $(f, \overline{f}) : a \to b$, we have $\class(f, \overline{f}) = \overline{f}$, and uniquely so. The rest is straightforward.
\end{proof}

While a marked-lax terminal object is not terminal, or even bi-terminal, it is terminal with respect to marked morphisms. In fact, it is a bit more, as we record in the following lemma.

\begin{lemma}
\label{lem:quasi.enhanced.etc}
	Let $\K$ be a marked enhanced 2-category, and let $U$ be a marked-lax terminal object in it. Then:
	\begin{enumerate}[label=\Alph*.]
		\item any marked 1-cell $X \marklooseto U$ is uniquely isomorphic to $\class X$ and,
		\item for any 1-cells $f : X \looseto Y$, $j : X \marklooseto U$ and $k : Y \marklooseto U$ there is a unique 2-cell $j \twoto kf$.
	\end{enumerate}
\end{lemma}
\begin{proof}
	For (A), apply condition (2) of \cref{lem:marked-lax-defn} to $j : X \marklooseto U$ to find a unique 2-cell $\class X \twoto \class U \circ j$; by condition (4) $!_U : \class U \cong \id_U$ and post-composing by an isomorphism is a bijection, so there is a unique 2-cell $\class X \twoto j$. Finally, by condition (3), this 2-cell is an isomorphism.
	For (B), we may use (A) above to unambiguously replace $j$ and $k$ by $\class X$ and $\class Y$ respectively, and then appeal to condition (2) of \cref{lem:marked-lax-defn}.
\end{proof}

As expected for a notion of terminal object, marked (and codotted)-lax terminal objects are preserved by right adjoints.

\begin{lemma}
\label{lemma:right.adjoint.preserve.codotted.terminal}
	Let $\E$ be a marked enhanced 2-category. Suppose that $r : \E \to \E'$ is a marked enhanced 2-functor with left 2-adjoint $\ell : \E' \to \E$ witnessed by 2-natural transformations unit $\eta : \id_{\E'} \Rightarrow r\ell$ and counit $\varepsilon : \ell r \Rightarrow \id_{\E}$. If $U$ is a marked-lax terminal object of $\E$, then $rU$ is a marked-lax terminal object of $\E'$. If $r$ and $\ell$ are codotted, then $r$ preserves codotted terminal objects as well.
	\end{lemma}
	\begin{proof}
	We conjugate by the adjunction as follows:
	\begin{equation}
	\begin{tikzcd}
		\ast & \ast & {\E \slice{\marklooseto}_{\lax} U} & {\E'\slice{\marklooseto}_{\lax} rU} \\
		{\E'} & \E & \E & {\E'}
		\arrow[equals, from=1-1, to=1-2]
		\arrow["rU"', from=1-1, to=2-1]
		\arrow[""{name=0, anchor=center, inner sep=0}, "{\id_U}", from=1-2, to=1-3]
		\arrow["U"', from=1-2, to=2-2]
		\arrow["{r_{\ast}}", from=1-3, to=1-4]
		\arrow[from=1-3, to=2-3]
		\arrow[from=1-4, to=2-4]
		\arrow["\varepsilon"', between={0.2}{0.8}, Rightarrow, from=2-1, to=1-2]
		\arrow["\ell"', from=2-1, to=2-2]
		\arrow[""{name=1, anchor=center, inner sep=0}, curve={height=30pt}, equals, from=2-1, to=2-4]
		\arrow[""{name=2, anchor=center, inner sep=0}, "\class"', dashed, from=2-2, to=1-3]
		\arrow[equals, from=2-2, to=2-3]
		\arrow["r"', from=2-3, to=2-4]
		\arrow["\eta", between={0.2}{1}, Rightarrow, from=1, to=2-3]
		\arrow["{!_U}", between={0.2}{0.8}, Rightarrow, from=2, to=0]
	\end{tikzcd}
	\end{equation}
	By the zig-zag laws, this pasting diagram equals the identity filling the outer square.

	We almost have a filler of the desired form, except for that pesky $\eta$. However, by the commutativity of the right half of the diagram, $\eta$ points into a split discrete 2-fibration; we may therefore lift it out of the way (by precomposing the maps into $eU$ with $\eta$) to get our desired filler:
	\begin{equation}
		\class'X := r(\class(\ell X)) \circ \eta_{X}.
	\end{equation}
\end{proof}

\begin{remark}
	In \cite{gagna2022bilimitsbifinalobjects}, the authors define a notion of $M$-bifinal object---which we would prefer to call a marked-lax terminal object, since it is to a quasi-terminal object what a marked-lax transformation is to a lax transformation---and they show that marked bilimits may be characterized by the existence of appropriate $M$-bifinal objects in an appropriately defined bicategory of marked-lax cones.
	The notion of marked-lax terminal object we are using here is very similar to that of a $M$-contraction given in Definition 2.2.4 of \cite{gagna2022bilimitsbifinalobjects}, except that we only ask that $\class U \cong \id_U$ and not that they are equal.
	Though they ask for the condition that $\class(\class X)$ be an identity, they point out in Remark 2.1.5 this follows from the assumption that $\class (\class X)$ is an isomorphism together with 2-functoriality of $\class$; this holds because $\class X$ is marked.
	For us, this appears instead as the fact that $!_U \circ \class(\class X)$ is an identity, and for the same reason.

	The main theorem of this section, \cref{thm:dotted.representability}, has much in common with the main results of \cite{gagna2022bilimitsbifinalobjects} concerning bilimits. We will not make this comparison explicit here; we could not find the exact result we needed in their paper.

	Because we are in the enhanced setting, we need a version of these results for enhanced 2-categories. The analogue of a marked-lax limit for enhanced 2-categories is the notion of \emph{dotted limits}, developed by \cite{koDottedLimits2023}.
	The first four conditions of \cref{def:enhanced.quasi} say that $U$ is a codotted-lax cocone of the identity enhanced 2-functor $\E \to \E$, in the (dual) sense of Definition 5.1.9 of \citep{koDottedLimits2023}.

	The marking condition (4) doesn't even make sense for a general (co)dotted-lax colimit, simply because usually the cocone and the diagram range over different collections of 1-cells.
	It is, however, crucial for the notion of marked-lax terminal object.
\end{remark}

We can now prove the main theorem of this section.

\begin{theorem}
\label{thm:dotted.representability}
	Let $\K$ be an enhanced 2-category and let $P : \E \to \K$ be an enhanced discrete 2-fibration over it.
	For $U \in \E$, the following are equivalent:
	\begin{enumerate}
		\item $U$ is marked-lax terminal in $\E$.
		\item The functor $(-)^{\ast}U : \K \slice{\looseto} P U \to \E$ given by the Yoneda lemma \cref{lem:2yoneda4me} is a (right adjoint) equivalence over $\K$. (The inverse need not preserve chosen lifts.)
	\end{enumerate}
\end{theorem}
\begin{proof}
	The second condition implies the first by \cref{lemma:right.adjoint.preserve.codotted.terminal} and \cref{cor:slice.dotted.terminal}. It remains to prove the converse.

	Suppose that $U$ is marked-lax terminal in $\E$. We need to show that $(-)^{\ast}U : \K \slice{\looseto} P U \to \E$ is an (adjoint) equivalence over $\K$. We define the inverse as the composite $P_{\ast} \circ \class$ as in the following diagram:
	\begin{equation}
		\begin{tikzcd}
			\ast & {\E \slice{\marklooseto}_{\lax} U} & {\K \slice{\looseto}_{\lax} P U} & \E \\
			\E & \E & \K & \K
			\arrow[""{name=0, anchor=center, inner sep=0}, "{\id_U}", from=1-1, to=1-2]
			\arrow[""{name=0p, anchor=center, inner sep=0}, phantom, from=1-1, to=1-2, start anchor=center, end anchor=center]
			\arrow["U"', from=1-1, to=2-1]
			\arrow["{P_{\ast}}", from=1-2, to=1-3]
			\arrow[from=1-2, to=2-2]
			\arrow["{(-)^{\ast}U}", from=1-3, to=1-4]
			\arrow[from=1-3, to=2-3]
			\arrow[from=1-4, to=2-4]
			\arrow[""{name=1, anchor=center, inner sep=0}, "\class"', dashed, from=2-1, to=1-2]
			\arrow[""{name=1p, anchor=center, inner sep=0}, phantom, from=2-1, to=1-2, start anchor=center, end anchor=center]
			\arrow[equals, from=2-1, to=2-2]
			\arrow["P"', from=2-2, to=2-3]
			\arrow[equals, from=2-3, to=2-4]
			\arrow["{!_U}"{pos=0.4}, shift left=4, between={0.2}{0.8}, Rightarrow, from=1p, to=0p]
		\end{tikzcd}
	\end{equation}
	where $P_{\ast}$ is given by applying $P$ to all components of the lax slice. We note that all maps in this diagram are marked (or codotted), with $\K$ considered to have all maps marked, $\K \slice{\looseto}_{\lax} PU$ marked as in \cref{const:slice.mark}, $\E$ marked as in \cref{const:2fib.mark}, $\E \slice{\marklooseto}_{\lax} U$ marked as in \cref{const:mark.mark.slice}.

	It remains to show that this is the inverse of $(-)^{\ast}U$, which we will accomplish by inspecting the whole of the above diagram.

	Looking at the left half of the above diagram, we see that $P_{\ast} !_U : P_{\ast} \circ \class \circ U \cong \id_{P U}$; by the Yoneda lemma \cref{lem:2yoneda4me}, $\id_{P U} : \ast \to \K \slice{\looseto}_{\lax} P U$ corresponds to the identity $\K \slice{\looseto}_{\lax} U \to \K \slice{\looseto}_{\lax} U$, and $U : \ast \to \E$ corresponds to $(-)^{\ast}U$. By the naturality of Yoneda, we conclude that $P_{\ast} \circ \class \circ (-)^{\ast}U \cong \id_{\K \slice{\looseto}_{\lax} U}$.

	On other hand, for every $f : X \to U$, we have a unique isomorphism $f \cong P(f)^{\ast} U$ by the universal property of cartesian lifts. This gives us the other side of the equivalence; we note that any equivalence may be rectified into an adjoint equivalence.
\end{proof}

\section{Plumbuses and enhanced 2-classifiers}
\label[section]{sec:plumbuses}

We now give the central definitions of the paper concerning classifiers for discrete opfibrations.
We closely follow \cite{weberYonedaStructures2toposes2007} and \cite{mesiti2classifiersDenseGenerators2025} in doing so, although we adapt the theory to the setting of suitably complete enhanced 2-categories we call \emph{plumbuses}.

\subsection{Plumbuses}
Consider the enhanced 2-category $\loosearrow{\K}$, where objects are loose arrows, morphisms are commuting squares, tights are squares with tight sides as below left, and 2-cells are given by barrels as below right:
\begin{equation}
\label{eq:barrels}
	\begin{tikzcd}[ampersand replacement=\&,sep=scriptsize]
		\cdot \&\& \cdot \\
		\\
		\cdot \&\& \cdot
		\arrow[""{name=0, anchor=center, inner sep=0}, curve={height=12pt}, squiggly, from=1-1, to=1-3]
		\arrow[""{name=1, anchor=center, inner sep=0}, curve={height=-12pt}, squiggly, from=1-1, to=1-3]
		\arrow[squiggly, from=1-1, to=3-1]
		\arrow[squiggly, from=1-3, to=3-3]
		\arrow[""{name=2, anchor=center, inner sep=0}, curve={height=12pt}, squiggly, from=3-1, to=3-3]
		\arrow[""{name=3, anchor=center, inner sep=0}, curve={height=-12pt}, squiggly, from=3-1, to=3-3]
		\arrow[between={0.2}{0.8}, Rightarrow, from=1, to=0]
		\arrow[between={0.2}{0.8}, Rightarrow, from=3, to=2]
	\end{tikzcd}
\end{equation}

\begin{definition}[{Plumbus}]\label[definition]{def:plumbus}
	A \textbf{smallness structure} over an enhanced 2-category $\K$ is the data of a split enhanced discrete 2-fibration which strictly includes in $\partial_1:\loosearrow{\K} \to \K$:
	\begin{equation}
		\label[diagram]{mc-0042}
		\begin{tikzcd}
			{\stdoarrow{\K}} && {\loosearrow{\K}} \\
			& \K
			\arrow[dashed, hook, from=1-1, to=1-3]
			\arrow["{{\partial_1}}"', two heads, from=1-1, to=2-2]
			\arrow["{\partial_1}", from=1-3, to=2-2]
		\end{tikzcd}
	\end{equation}
	This means $\K$ is equipped with a pullback-stable class of tight discrete opfibrations we call \textbf{small} (these are the objects of $\stdoarrow{\K}$), with \emph{chosen} left-tight pullbacks (\cref{def:left-tight-pbs}) along arbitrary loose 1-cells (these give the splitting).

	A \textbf{plumbus} is a repletely enhanced 2-category, with a tight terminal object $1$, and a smallness structure.
\end{definition}

\begin{terminology}
	The name `plumbus' is a bad pun on `plumb bob' and `cosmos', based on the idea that plumbuses are a good setting for `(un)straightening'---that is, the classification of discrete opfibrations.
\end{terminology}

\begin{remark}
	Since the 2-cells of $\loosearrow{\K}$ are barrels (cf.~\eqref{eq:barrels}), the 2-cartesianity \cite{hermidaPropertiesFibFibred1999} of \eqref{mc-0042} is due to the fact pullback of discrete opfibrations also induces a functorial action on 2-cells, thanks to their lifting property---see p.296 in \cite{weberYonedaStructures2toposes2007}.
\end{remark}

Note, also, that $\partial_1 : \stdoarrow{\K} \to \K$ is locally fully faithful, since the data of a barrel with discrete opfibrant sides is completely determined by its bottom side.

We also note an interesting consequence of being a plumbus:

\begin{proposition}
\label{prop:disc-opfibs-detect-tightness}
	In a plumbus, small discrete opfibrations \emph{detect tightness}, that is, when $p:E \discto B$ is a small discrete opfibration and $pe$ is tight for some $e:X \to E$, then so is $e$.
\end{proposition}
\begin{proof}
	Apply the tightness detection property of left-tight pullbacks to the pullback of $p$ along $\id_B$.
\end{proof}

In particular, all maps into small discrete objects (i.e. those for which the terminal morphism $!$ is a small discrete opfibration) are tight, and thus the full subcategory spanned by them is completely contained in $\K_t$.

\subsection{Enhanced 2-classifiers}
In this section, we will adapt Definition 4.1 from \cite{weberYonedaStructures2toposes2007} to the context of enhanced 2-categories.

\begin{definition}[{Enhanced 2-classifier}]
\label[definition]{mc-0045}
	An \textbf{enhanced 2-classifier} in a plumbus $\K$ is a small discrete opfibration:
	\begin{equation}
		\begin{tikzcd}[ampersand replacement=\&,column sep=scriptsize]
			{\Omega_\bullet} \&\& \Omega
			\arrow["u", two heads, from=1-1, to=1-3]
		\end{tikzcd}
	\end{equation}
	which is marked-lax terminal in the enhanced 2-category $\stdoarrow{\K}$. By \cref{thm:dotted.representability}, this equivalently means that the 2-functor \((-)^*u\) induced by taking pullbacks is an equivalence of enhanced discrete 2-fibrations over $\K$:
	\begin{equation}
		\begin{tikzcd}[ampersand replacement=\&]
			{\K\slice{\looseto}_{\lax}\Omega} \&\& {\stdoarrow{\K}} \\
			\& \K
			\arrow["{(-)^*u}"', from=1-1, to=1-3]
			\arrow["{\partial_0}"', two heads, from=1-1, to=2-2]
			\arrow["{\class(-)}"', curve={height=12pt}, dashed, from=1-3, to=1-1]
			\arrow["{\partial_1}", two heads, from=1-3, to=2-2]
		\end{tikzcd}
	\end{equation}
	When a plumbus \(\K\) is equipped with an enhanced 2-classifier, we say it is \textbf{representable}.
\end{definition}

Explicitly, using \cref{lem:marked-lax-defn}, an enhanced 2-classifier is a small discrete opfibration $u : \Omega_{\bullet} \discto \Omega$ such that:
\begin{enumerate}
	\item For any small discrete opfibration $p : E \discto B$, there is a pullback square
	\begin{equation}
		\begin{tikzcd}
			E & {\Omega_{\bullet}} \\
			B & \Omega
			\arrow["{\class_{\bullet} p}", squiggly, from=1-1, to=1-2]
			\arrow["p"', two heads, from=1-1, to=2-1]
			\arrow["\lrcorner"{anchor=center, pos=0.125}, draw=none, from=1-1, to=2-2]
			\arrow["u", from=1-2, to=2-2]
			\arrow["{\class p}"', squiggly, from=2-1, to=2-2]
		\end{tikzcd}
	\end{equation}
	\item For any square as below left, there is are unique 2-cells making the right commute:
	\begin{equation}
		\begin{tikzcd}
			E & {E'} \\
			B & {B'}
			\arrow["g", squiggly, from=1-1, to=1-2]
			\arrow["p"', two heads, from=1-1, to=2-1]
			\arrow["{p'}", two heads, from=1-2, to=2-2]
			\arrow["f", squiggly, from=2-1, to=2-2]
		\end{tikzcd}
		\quad\quad\quad
		\begin{tikzcd}[ampersand replacement=\&]
			\&\&\& {E'} \\
			E \&\&\& {B'} \\
			B \&\&\& {\Omega_\bullet} \\
			\&\&\& \Omega
			\arrow["{p'}"', two heads, from=1-4, to=2-4]
			\arrow["g", squiggly, from=2-1, to=1-4]
			\arrow["p"', two heads, from=2-1, to=3-1]
			\arrow[""{name=2, anchor=center, inner sep=0}, "{\class p'}"{pos=0.6}, curve={height=-12pt}, squiggly, from=2-4, to=4-4]
			\arrow["f", squiggly, from=3-1, to=2-4]
			\arrow[""{name=3, anchor=center, inner sep=0}, "{\class p}"', curve={height=12pt}, squiggly, from=3-1, to=4-4]
			\arrow["u"', two heads, from=3-4, to=4-4]
			\arrow["{\class g}"{description}, shift left=4, between={0.2}{0.8}, Rightarrow, from=3, to=2]
			\arrow[""{name=0, anchor=center, inner sep=0}, "{\class_\bullet p'}", curve={height=-12pt}, squiggly, from=1-4, to=3-4]
			\arrow[""{name=1, anchor=center, inner sep=0}, "{\class_\bullet p}"', curve={height=12pt}, squiggly, from=2-1, to=3-4]
			\arrow["{\exists!}"{description}, between={0.2}{0.8}, Rightarrow, from=1, to=0]
		\end{tikzcd}
	\end{equation}
	\item If the square $(f, g) : p \looseto p'$ above is a pullback, then $\class f$ and $\class_{\bullet} g$ are (unique) isomorphisms.
	\item We have an isomorphism $!_U : \class U \cong \id_U$.
\end{enumerate}

\begin{remark}
\label{mc-0030}
	In a representable plumbus, a map \(f : B \to \Omega \) is a \textbf{formal copresheaf} and \(f^* u\) its \textbf{opfibration of elements}, while \(\class p : B \to \Omega \) is the \textbf{copresheaf of fibers} of an small discrete opfibration \(p:E \to B\).
	When \(p\) is representable by \(b\), then we also say the formal copresheaf \(\class p\) is \textbf{representable} and denote it by \(B(b,-)\).
\end{remark}

\begin{observation}
\label[observation]{mc-000W-observation}
	Given \(u : \Omega _\bullet \discto \Omega \) for which \((-)^*u\) is fully faithful as a functor into tight discrete opfibrations, there is a natural notion of smallness determined by the replete image of \((-)^*u\).
	Indeed, one can see smallness as a derived, rather than primary, notion, as done e.g. in \cite{weberYonedaStructures2toposes2007}.
\end{observation}

The restriction to small discrete opfibrations, rather than the whole class of tight discrete opfibrations, is due to a necessity arising in practice, for instance in the archetypal example of the 2-category of large but locally small categories:

\begin{example}
\label[example]{ex:2-class-cat}
	The archetypal example of 2-classifier is the projection \(\Set_\bullet \overset{u}\discto \Set\) from the category pointed sets to the category of sets in \(\Cat\) (which, for us, is the 2-category of large but locally small categories).
	As the small discrete opfibrations, pick those having small (in the usual sense of `having the size of a set in the first universe') fibers.
	This recovers the well-known discrete opfibration of elements construction:

	\begin{equation}
		\begin{tikzcd}
			{\int  F} & {\Set_\bullet } \\
			B & {\Set}
			\arrow[from=1-1, to=1-2]
			\arrow["{\partial_1}"', two heads, from=1-1, to=2-1]
			\arrow["\lrcorner "{anchor=center, pos=0.125}, draw=none, from=1-1, to=2-2]
			\arrow["u", two heads, from=1-2, to=2-2]
			\arrow["F"', from=2-1, to=2-2]
		\end{tikzcd}
	\end{equation}

	We can choose naturally \(\int \class p \cong p\) and \(\class \int F \cong F\) so as to make this into an equivalence.
\end{example}

Notice in \cref{ex:2-class-cat} above we could not have chosen \emph{all} discrete opfibrations, since that would have forced \(\Omega \) to be the \emph{very large} category of \emph{large sets}, which is not a member of \(\Cat\).

\begin{remark}
	Moreover, let us also remark that both tightness and smallness are necessary notions: tightness refines arbitrary maps, while smallness refines only tight discrete opfibrations.
	For instance, in constructing \(\Alg_{\lax, t\pseudo}(T)\) one uses the notion of tightness from $\K$ to control which maps can become algebras, which is unrelated to the notion of smallness governing the classifiability of discrete opfibrations.
	Moreover, in \(\Alg_{\lax, t\pseudo}(T)\) tightness and smallness will distinguish, respectively, pseudo and strict maps---and this choice is crucial: opfibrations need to be strict to behave well, whereas strictness is too strong of a property for the maps we need to be tight (e.g. it's not a replete property).
\end{remark}

\begin{remark}
	Finally, we remark that smallness is related to \emph{admissibility} from \cite{streetYonedaStructures2categories1978}, the notion of \emph{attribute} from \cite{weberYonedaStructures2toposes2007}, and finally the unnamed `pullback-stable property \(P\)' in \cite{mesiti2classifiersDenseGenerators2025}.
	The \emph{attributes} of Weber are, essentially, what we would call the \emph{small} modules (i.e. two-sided discrete opfibrations), though, unlike Weber, we do not venture in that direction (except briefly in \cref{app:koudenburg}, indeed see \cref{mc-003Q-small-modules}).
	Finally, \emph{admissibility} is the most indirectly related: as shown by Weber, his notion of 2-topos gives rise to a Yoneda structure where the admissible maps are those which represent a small two-sided opfibration.
	However, we cannot easily express this notion in the setting of representable plumbuses since we lack a duality (as in a Weber 2-topos) which would make two-sided discrete opfibrations instances of one-sided ones.
\end{remark}

We can also explore the meaning of codotting in representable plumbuses.

\begin{definition}[Perfectness]\label{mc-002Z}
	A small discrete opfibration in a representable plumbus is \textbf{perfect} when its classifying map is tight.
\end{definition}

We note that $p$ is perfect if and only if it is codotted in $\stdoarrow{\K}$, transporting the codotting of $\K \slice{\looseto} \Omega$ over the equivalence $(-)^{\ast}u$.

\begin{construction}[{Marking and codotting on small discrete opfibrations}]\label[construction]{const:marking.codotting.stdopfib}
	Let $\K$ be a representable plumbus.
	Consider the subcategory $\stdoarrow{\K} \subseteq \loosearrow{\K}$ spanned by the small discrete opfibrations, as defined above.
	It is enhanced, marked and codotted as follows:
	\begin{equation}
		\begin{array}{ccc}
			\begin{tikzcd}[ampersand replacement=\&]
				\cdot \& \cdot \\
				\cdot \& \cdot
				\arrow[from=1-1, to=1-2]
				\arrow[two heads, from=1-1, to=2-1]
				\arrow[two heads, from=1-2, to=2-2]
				\arrow[from=2-1, to=2-2]
			\end{tikzcd}
			\qquad & \qquad
			\begin{tikzcd}[ampersand replacement=\&]
				\cdot \& \cdot \\
				\cdot \& \cdot
				\arrow[squiggly, from=1-1, to=1-2]
				\arrow[two heads, from=1-1, to=2-1]
				\arrow["\lrcorner"{anchor=center, pos=0.125}, draw=none, from=1-1, to=2-2]
				\arrow[two heads, from=1-2, to=2-2]
				\arrow[squiggly, from=2-1, to=2-2]
			\end{tikzcd}
			\qquad & \quad
			\begin{tikzcd}[ampersand replacement=\&]
				E \& {\Omega_\bullet} \\
				B \& \Omega
				\arrow["{\class_\bullet p}", from=1-1, to=1-2]
				\arrow["p"', two heads, from=1-1, to=2-1]
				\arrow["\lrcorner"{anchor=center, pos=0.125}, draw=none, from=1-1, to=2-2]
				\arrow["u", two heads, from=1-2, to=2-2]
				\arrow["{\class p}"', from=2-1, to=2-2]
			\end{tikzcd}
			\\
			\begin{gathered}
				\textbf{tight}\\\text{tight sides}
			\end{gathered}
			\qquad & \qquad
			\begin{gathered}
				\textbf{marked}\\\text{pullbacks}
			\end{gathered}
			\qquad & \quad
			\begin{gathered}
				\textbf{codotted}\\\text{perfect (\cref{mc-002Z})}
			\end{gathered}
		\end{array}
	\end{equation}
	The required closure property for perfect discrete opfibrations is readily verified: indeed, if $p$ is perfect and $q \to p$ is a cartesian square with tight sides, then by pullback pasting we obtain $q$ is classified by a tight too (because tights are replete in a plumbus), and thus perfect.
\end{construction}

\subsubsection{Good 2-classifiers}
\label{sec:good-2-classifiers}
An important observation, due to \cite{weberYonedaStructures2toposes2007}, is that when $\id_1$ is small---and we always assume that in a representable plumbus---then the maps $\class(\id_X) = X \xrightarrow{!} 1 \xrightarrow{\class (\id_1)} \Omega$ are terminal in each category $\K(X,\Omega)$, and thus $\class(\id_1)$ can be considered `internally terminal' in $\Omega$, that is:

\begin{proposition}[{\cite{weberYonedaStructures2toposes2007}, Proposition~8.2}]
\label{id-1-is-terminal}
	The classifying map $\class(\id_1)$ is right adjoint to $!:\Omega \to 1$.
\end{proposition}

Moreover, \cite{shulmanClassifyingDiscreteOpfibration2023} makes the following observation (which we adapt to the enhanced setting):

\begin{proposition}[{\cite{shulmanClassifyingDiscreteOpfibration2023}, Theorem 0.1}]
\label{shulman-comma-th}
	In a representable plumbus where $\id_1$ is perfect (\cref{mc-002Z}), there exists an l-rigged comma (\cref{mc-002L}):
	\begin{equation}
	\label[diagram]{shulman-comma-th-diagram}
		\begin{tikzcd}[ampersand replacement=\&]
			{\Omega_\bullet} \& 1 \\
			\Omega \& \Omega
			\arrow[from=1-1, to=1-2]
			\arrow["u"', two heads, from=1-1, to=2-1]
			\arrow["\lrcorner"{anchor=center, pos=0.125}, draw=none, from=1-1, to=2-2]
			\arrow[between={0.2}{0.8}, Rightarrow, from=1-2, to=2-1]
			\arrow["{\class(\id_1)}", from=1-2, to=2-2]
			\arrow[equals, from=2-1, to=2-2]
		\end{tikzcd}
	\end{equation}
\end{proposition}
\begin{proof}
	We reproduce the argument of Shulman and adapt it to the enhanced setting.
	Suppose given a lax square:
	\begin{equation}
		\begin{tikzcd}[ampersand replacement=\&]
			X \& 1 \\
			\Omega \& \Omega
			\arrow[from=1-1, to=1-2]
			\arrow["f"', squiggly, from=1-1, to=2-1]
			\arrow["\varphi", between={0.2}{0.8}, Rightarrow, from=1-2, to=2-1]
			\arrow["{\class(\id_1)}", from=1-2, to=2-2]
			\arrow[equals, from=2-1, to=2-2]
		\end{tikzcd}
	\end{equation}
	We want to show such data is equivalent to a map $h:X \to \Omega_\bullet$.
	Intuitively, a map like $h$ corresponds to a small discrete opfibration (which is the one classified by $uh$) and a section thereof (which picks out the additional `points' $h$ picks by landing in $\Omega_\bullet$ rather than $\Omega$).

	Indeed, note $\varphi$ is a map of classifying maps $\class (\id_X)  \twoto \class f$, and thus, by fully faithfulness of $(-)^*u$, corresponds to a unique map $\id_X \to f^*u$, i.e. the sought section of $f^*u$.
	Clearly, this correspondence is a bijection so \cref{shulman-comma-th-diagram} is a comma.

	To conclude we must prove the comma is l-rigged.
	First, note that by definition of perfectness, $\class(\id_1)$ is tight.
	Second, $u$ and $!_{\Omega_\bullet}$ are also tight by assumption, and finally they jointly detect tightness since $u$ does by \cref{prop:disc-opfibs-detect-tightness} and postcomposing with $!_{\Omega_\bullet}$ always yields a tight map.
\end{proof}

\begin{definition}[{Good 2-classifier}]\label[definition]{mc-0012}
	Let \(\K\) be a representable plumbus.
	We say its 2-classifier \(u\) is \textbf{good} when $\id_1$ is perfect, and we say \(\K\) is \textbf{well-representable}.
\end{definition}

A good 2-classifier, therefore, substantiates the intuition that $\Omega_\bullet$ is really the object of `pointed internal sets'.
In this setting, we can use \( \tau := \class(\id_1) \) to classify discrete opfibrations by taking commas: for every small discrete opfibration $p$, a straightforward application of the pasting lemma for commas (\cref{mc-000F}) shows that \(p \cong \tau / \class p\):
\begin{equation}
\label[diagram]{diag:good-classifying-commas}
	\begin{tikzcd}
		E & {\tau  / \Omega } & 1 \\
		B & \Omega  & \Omega
		\arrow[squiggly, from=1-1, to=1-2]
		\arrow["p"', two heads, from=1-1, to=2-1]
		\arrow["\lrcorner "{anchor=center, pos=0.125}, draw=none, from=1-1, to=2-2]
		\arrow[from=1-2, to=1-3]
		\arrow["u"', two heads, from=1-2, to=2-2]
		\arrow["\lrcorner "{anchor=center, pos=0.125}, draw=none, from=1-2, to=2-3]
		\arrow[shorten <=4pt, shorten >=4pt, Rightarrow, from=1-3, to=2-2]
		\arrow["\tau ", from=1-3, to=2-3]
		\arrow["{\class p}"', squiggly, from=2-1, to=2-2]
		\arrow[equals, from=2-2, to=2-3]
		\end{tikzcd}
		\quad  = \quad
		\begin{tikzcd}[column sep=scriptsize]
		E && 1 \\
		B && \Omega
		\arrow[from=1-1, to=1-3]
		\arrow["p"', two heads, from=1-1, to=2-1]
		\arrow["\lrcorner "{anchor=center, pos=0.125}, draw=none, from=1-1, to=2-3]
		\arrow[shorten <=10pt, shorten >=10pt, Rightarrow, from=1-3, to=2-1]
		\arrow["\tau ", from=1-3, to=2-3]
		\arrow["{\class p}"', squiggly, from=2-1, to=2-3]
	\end{tikzcd}
\end{equation}
Observe that \cref{mc-000F} here extends to a statement about left-tight pullbacks and l-rigged commas: if the pullback above is left-tight, then the comma square on the right is l-rigged, and \emph{vice versa}.

Thus we have again an equivalence of 2-fibred enhanced 2-categories
\begin{equation}
	\tau / - : \K\slice{\looseto}_{\lax}\Omega \xrightarrow {\sim } \stdoarrow{\K}.
\end{equation}
which makes $\tau$ a \emph{good 2-classifier}, in the sense of Definition 2.15 of \cite{mesiti2classifiersDenseGenerators2025}, for the pullback-stable property $P$ of being small.

\begin{example}
	The archetypal 2-classifier in $\Cat$ (\cref{ex:2-class-cat}) arises from a good 2-classifier, namely the inclusion \(1 \xrightarrow {1} \Set\), which is indeed the classifier of $\id_1$.
\end{example}

\begin{construction}[{Mesiti's construction}]\label[construction]{mc-002P}
	As anticipated in the introduction, in \cite{mesiti2classifiersDenseGenerators2025} it is shown that for any small category \(C\), the (chordate enhanced) 2-category \([C, \Cat]\) admits a good 2-classifier:
	\begin{equation}
		\begin{tikzcd}
		{C} && {\Cat} \\[-3ex]
		c && {[c/C, \Set]} \\
		{c'} && {[c'/C, \Set]}
		\arrow["\Omega ", from=1-1, to=1-3]
		\arrow[""{name=0, anchor=center, inner sep=0}, "f"', from=2-1, to=3-1]
		\arrow[""{name=1, anchor=center, inner sep=0}, "{(f^*)^*}", from=2-3, to=3-3]
		\arrow[shorten <=13pt, shorten >=27pt, maps to, from=0, to=1]
		\end{tikzcd}
	\end{equation}
	with \(\tau  : 1 \to  \Omega \) picking the constant functor \(1 \in  [c/C, \Set]\) for each \(c \in C\).
\end{construction}

\begin{example}[{Good 2-classifier on \(\mathrm{Graph}(\Cat)\)}]\label[example]{mc-002Q}
	Following Mesiti's recipe for \(\mathrm{Graph}(\Cat) = \left[\left\{\begin{tikzcd}[sep=small,cramped]
		e & v
		\arrow["t"', shift right, from=1-1, to=1-2]
		\arrow["s", shift left, from=1-1, to=1-2]
	\end{tikzcd}\right\}, \Cat\right]\), one finds that:
	\begin{equation}
		\Omega (e) = \left [ \left \{
			\begin{tikzcd}[cramped,sep=tiny]
			& {\id_e} \\
			s && t
			\arrow[from=1-2, to=2-1]
			\arrow[from=1-2, to=2-3]
			\end{tikzcd}
			\right \}, \Set
		\right ],
		\qquad
		\Omega (v) = \left [\left \{ \id_v \right \}, \Set \right ]
	\end{equation}
	with \(\Omega (s)\) and \(\Omega (t)\) given by projecting the feet of the spans.
	Thus \(\Omega \) is the (large) graph of spans of sets, and \(\tau : 1 \to \Omega \) picks out the trivial span \(1 = 1 = 1\).
\end{example}

\section{Lifting theorem}
\label[section]{sec:lifting-theorem}

For the rest of the section, fix a representable plumbus \(\K\) (\cref{mc-0045}), having enhanced 2-classifier \(u : \Omega_\bullet \discto \Omega \).
The main goal of this section is to give conditions on an enhanced 2-monad \(T\) on \(\K\) such that \(\LaxAlg_{\lax, t\pseudo}(T)\) (and \(\PsAlg_{\lax, t\pseudo}(T)\)) are again representable (\cref{mc-0003}) or even well-representable (\cref{mc-0003-good}) plumbuses (\cref{prop:lax-alg-plumbus}).

\begin{remark}
	These theorems take care of both lax and pseudo-$T$-algebras, but not of strict $T$-algebras.
	The matter of strictification of our result is, in fact, not trivial. Whatever technique is employed would need to play well with (1) the strictness of small discrete opfibrations and (2) the classification property of $\omega$. We expect that descent arguments on suitable 2-categories (such as presheaf 2-categories) can show that when $T$ is opfibrantly cartesian, its pseudo-algebra coclassifier $T'$ is as well; then $\Omega$ can classify strict discrete opfibrations amongst \emph{strict} $T'$ algebras.
	We leave this line of enquiry open for future work.
\end{remark}

\subsection{Lax algebra structures on marked-lax terminal objects}
\label[section]{sec:formal-2-class}

In this section, we will show that a marked-lax terminal object admits an essentially unique lax algebra structure for any appropriate 2-monad.

The following can be considered a purely formal version of the main result of the paper, \cref{mc-0003} below.

\begin{definition}
\label{def:marked.2-monad}
	An enhanced 2-monad $(T,i,m)$ over a marked enhanced 2-category $\K$ is \textbf{marked} when it preserves marked 1-cells.
	It is \textbf{fully marked} when $i$ and $m$ have marked components.
\end{definition}

\begin{construction}[{Marking and codotting of lax algebras}]
	Suppose $T$ is marked.
	Consider the enhanced 2-category $\LaxAlg_{\lax, t\pseudo}(T)$ of lax algebras defined in \cref{def:lax-alg}.
	We consider it marked at those morphisms whose carriers are marked in $\K$, and codotted at those algebras whose carrier is codotted in $\K$ and whose structure map is marked:
	\begin{equation}
		\begin{array}{ccc}
			\begin{tikzcd}[ampersand replacement=\&]
				TA \& TB \\
				A \& B
				\arrow[""{name=0, anchor=center, inner sep=0}, "Tf", from=1-1, to=1-2]
				\arrow["\alpha"', from=1-1, to=2-1]
				\arrow["\beta", from=1-2, to=2-2]
				\arrow[""{name=1, anchor=center, inner sep=0}, "f"', from=2-1, to=2-2]
				\arrow["\wr", between={0.3}{0.7}, Rightarrow, from=0, to=1]
			\end{tikzcd}
			\qquad & \qquad
			\begin{tikzcd}[ampersand replacement=\&]
				TA \& TB \\
				A \& B
				\arrow[""{name=0, anchor=center, inner sep=0}, "Tf"{inner sep=.8ex}, "/"{marking}, from=1-1, to=1-2]
				\arrow["\alpha"', from=1-1, to=2-1]
				\arrow["\beta", from=1-2, to=2-2]
				\arrow[""{name=1, anchor=center, inner sep=0}, "f"'{inner sep=.8ex}, "/"{marking}, from=2-1, to=2-2]
				\arrow[between={0.3}{0.7}, Rightarrow, from=0, to=1]
			\end{tikzcd}
			\qquad & \qquad
			\begin{tikzcd}[ampersand replacement=\&]
				TA \\
				A
				\arrow["\alpha"'{inner sep=.8ex}, "/"{marking}, from=1-1, to=2-1]
			\end{tikzcd}
			\\
			\begin{gathered}
				\textbf{tight}\\\text{tight pseudo-$T$-maps}\\\phantom{space}
			\end{gathered}
			\qquad & \qquad
			\begin{gathered}
				\textbf{marked}\\\text{carrier marked in $\K$}\\\phantom{space}
			\end{gathered}
			\qquad & \quad
			\begin{gathered}
				\textbf{codotted}\\\text{$A$ codotted in $\K$}\\\text{and $\alpha$ marked}
			\end{gathered}
		\end{array}
	\end{equation}
\end{construction}

\begin{theorem}
\label{th:lift.of.enh.quasi-terminal.to.laxalg}
	Let $\K$ be a marked and codotted enhanced 2-category with an codotted-marked-lax terminal object $U$.
	Suppose that $(T,i,m)$ is a marked enhanced 2-monad on $\K$ such that $TU$ is codotted.
	Then there is a (unique) lax algebra structure on $U$ which is codotted-marked-lax terminal in $\LaxAlg_{\lax, t\pseudo}(T)$.
\end{theorem}

Before concluding with the proof of this theorem, we immediately observe that:

\begin{corollary}
\label{cor:lift.of.enh.quasi-terminal.to.psalg}
	When $T$ is fully marked, then the lax algebra structure on $U$ is pseudo and it is codotted-marked-lax terminal in $\PsAlg_{\lax, t\pseudo}(T)$.
\end{corollary}

Note we cannot strictify further unless we also strictify (3) in the definition of marked-lax terminal object, but that does not hold in the cases we are interested.

\begin{proof}[Proof of~\cref{th:lift.of.enh.quasi-terminal.to.laxalg}]
	We begin by constructing the algebra structure; we will then prove that it is codotted-marked-lax terminal. You may find the definition of a lax algebra at the begining of Section 1 in \cite{lack-2002-codescent}. We will use the characterization of marked-lax terminal object from \cref{lem:marked-lax-defn}.

	The carrier of the algebra structure is ${\class (TU) : TU \to U}$, which we note is tight by the assumption that $TU$ was codotted, and is marked by condition (4). We construct the multiplicator and unitor of $\class U$ from the following diagrams, which are guaranteed by condition (2) of $U$'s universal property:
	\begin{equation}
		\label{eq:U-alg-struct}
		\begin{tikzcd}[ampersand replacement=\&]
			{T^2 U} \&\& TU \&\& U \&\& TU \\
			\\
			TU \&\& U \&\&\&\& U
			\arrow["{T(\class(TU))}"{inner sep=.8ex}, "/"{marking}, from=1-1, to=1-3]
			\arrow["m"', from=1-1, to=3-1]
			\arrow[""{name=0, anchor=center, inner sep=0}, "{\class(T^2U)}"{description}, from=1-1, to=3-3]
			\arrow[""{name=0p, anchor=center, inner sep=0}, phantom, from=1-1, to=3-3, start anchor=center, end anchor=center]
			\arrow[""{name=0p, anchor=center, inner sep=0}, phantom, from=1-1, to=3-3, start anchor=center, end anchor=center]
			\arrow[""{name=1, anchor=center, inner sep=0}, "{\class (TU)}", from=1-3, to=3-3]
			\arrow[""{name=1p, anchor=center, inner sep=0}, phantom, from=1-3, to=3-3, start anchor=center, end anchor=center]
			\arrow["i", from=1-5, to=1-7]
			\arrow[""{name=2, anchor=center, inner sep=0}, curve={height=18pt}, equals, from=1-5, to=3-7]
			\arrow[""{name=2p, anchor=center, inner sep=0}, phantom, from=1-5, to=3-7, start anchor=center, end anchor=center, curve={height=18pt}]
			\arrow[""{name=3, anchor=center, inner sep=0}, "{\class(U)}"{description, pos=0.3}, curve={height=-12pt}, from=1-5, to=3-7]
			\arrow[""{name=3p, anchor=center, inner sep=0}, phantom, from=1-5, to=3-7, start anchor=center, end anchor=center, curve={height=-12pt}]
			\arrow[""{name=3p, anchor=center, inner sep=0}, phantom, from=1-5, to=3-7, start anchor=center, end anchor=center, curve={height=-12pt}]
			\arrow["{\class(TU)}", from=1-7, to=3-7]
			\arrow[""{name=4, anchor=center, inner sep=0}, "{\class (TU)}"'{inner sep=.8ex}, "/"{marking}, from=3-1, to=3-3]
			\arrow[""{name=4p, anchor=center, inner sep=0}, phantom, from=3-1, to=3-3, start anchor=center, end anchor=center]
			\arrow[shift right=4, between={0.2}{0.8}, Rightarrow, from=0p, to=4p]
			\arrow["\sim", shift left=3, between={0.2}{0.8}, Rightarrow, from=0p, to=1p]
			\arrow[between={0.2}{1}, Rightarrow, from=3p, to=1-7]
			\arrow["\sim"', between={0.2}{0.8}, Rightarrow, from=3p, to=2p]
		\end{tikzcd}
	\end{equation}
	By condition (3) and because $T$ preserves markings, the 2-cells labeled with $\sim$ above are isomorphisms.
	They may therefore be inverted, and this gives us the definition of our structure.
	Note that if $i$ and $m$ are marked as well, as is the case for a fully marked 2-monad, then the other 2-cells are also invertible, making the algebra pseudo.

	There are three axioms that this structure is meant to satisfy, which may be found displayed on page 225 of \cite{lack-2002-codescent}. All three are instances of \cref{lem:quasi.enhanced.etc}(B) when viewed from the right perspective.

	The universal property of $U$ as a lax algebra follows similarly. Let $\alpha : TA \to A$ be another lax algebra; to prove \cref{lem:marked-lax-defn}(1), we will extend $\class A : A \to U$ to a lax morphism of lax algebras. We form the laxator as follows:
	\begin{equation}
		\label{eq:laxators.formal}
		\begin{tikzcd}
			TA && TU \\
			\\
			A && U
			\arrow["{T\class A}"{inner sep=.8ex}, "/"{marking}, from=1-1, to=1-3]
			\arrow["\alpha"', from=1-1, to=3-1]
			\arrow[""{name=0, anchor=center, inner sep=0}, "{\class (TA)}"{description}, from=1-1, to=3-3]
			\arrow[""{name=1, anchor=center, inner sep=0}, "{\class (TU)}", from=1-3, to=3-3]
			\arrow[""{name=2, anchor=center, inner sep=0}, "{\class A}"'{inner sep=.8ex}, "/"{marking}, from=3-1, to=3-3]
			\arrow["\sim", shift left=4, between={0.2}{0.8}, Rightarrow, from=0, to=1]
			\arrow[shift right=5, between={0.2}{0.8}, Rightarrow, from=0, to=2]
		\end{tikzcd}
	\end{equation}
	As with the lax algebra structure of $U$, here we use that $T$ preserves marked cells and that $\class A$ is marked; therefore, $T\class A$ is marked, and so the induced 2-cell in the top right is an isomorphism. That this is indeed a lax morphism follows quickly by applying \cref{lem:quasi.enhanced.etc} to both required equations (which may be found on page 226 of \cite{lack-2002-codescent}).

	For condition (2), we need to show that for any lax algebra morphism $f : A \to B$, there is a unique lax algebra transformation $k : \class A \Rightarrow \class B \circ f$. Of course, we have assumed that there is a unique 2-cell of that signature; it therefore remains to show that it must be a lax algebra morphism, which means proving the equality below:
	\begin{equation}
		\begin{tikzcd}[sep=scriptsize]
			TA && TU && TA && TU \\
			& TB && {=} \\
			A && U && A && U \\
			& B &&&& B
			\arrow[""{name=0, anchor=center, inner sep=0}, "{T(\class A)}", from=1-1, to=1-3]
			\arrow[""{name=1, anchor=center, inner sep=0}, from=1-1, to=2-2]
			\arrow["\alpha"', from=1-1, to=3-1]
			\arrow["{\class (TU)}", from=1-3, to=3-3]
			\arrow[""{name=2, anchor=center, inner sep=0}, "{T(\class A)}", from=1-5, to=1-7]
			\arrow["\alpha"', from=1-5, to=3-5]
			\arrow["{\class(TU)}", from=1-7, to=3-7]
			\arrow[""{name=3, anchor=center, inner sep=0}, from=2-2, to=1-3]
			\arrow[from=2-2, to=4-2]
			\arrow[""{name=4, anchor=center, inner sep=0}, "f"', from=3-1, to=4-2]
			\arrow[""{name=5, anchor=center, inner sep=0}, "{\class A}", from=3-5, to=3-7]
			\arrow["f"', from=3-5, to=4-6]
			\arrow[""{name=6, anchor=center, inner sep=0}, "{\class B}"', from=4-2, to=3-3]
			\arrow["{\class B}"', from=4-6, to=3-7]
			\arrow[between={0.2}{1}, Rightarrow, from=0, to=2-2]
			\arrow[between={0.2}{0.8}, Rightarrow, from=1, to=4]
			\arrow[between={0.2}{0.8}, Rightarrow, from=2, to=5]
			\arrow[between={0.2}{0.8}, Rightarrow, from=3, to=6]
			\arrow[between={0.2}{1}, Rightarrow, from=5, to=4-6]
		\end{tikzcd}
	\end{equation}

	This is again an instance of \cref{lem:quasi.enhanced.etc}. We may see both of these diagrams as 2-cells $(\class U \circ T\class A) \twoto \class B \circ (f \circ \alpha)$. Since $\class U \circ T\class A$ and $\class B$ are marked, by (2) of \cref{lem:quasi.enhanced.etc}, the above diagrams must be equal.

	Condition (3) follows immediately from the fact that it holds in $\K$. As for condition (5), if $\alpha : TA \to A$ is a codotted lax algebra, then by assumption $A$ was codotted in $\K$ and $\alpha$ was marked; therefore, $\class A : A \to U$ is tight in $\K$ and pseudo as an algebra morphism.

	Finally, to prove condition (4) it will suffice to show that $!_U : \class U \cong \id_U$ is a 2-cell of lax algebras. This means proving that
	\begin{equation}
		\begin{tikzcd}
			TU & TU && TU && TU \\
			&& {=} \\
			U & U && U && U
			\arrow[""{name=0, anchor=center, inner sep=0}, "{T\class U}"{inner sep=.8ex}, "/"{marking}, curve={height=-12pt}, from=1-1, to=1-2]
			\arrow[""{name=1, anchor=center, inner sep=0}, curve={height=12pt}, equals, from=1-1, to=1-2]
			\arrow["{\class TU}"'{inner sep=.8ex}, "/"{marking}, from=1-1, to=3-1]
			\arrow["{\class TU}", from=1-2, to=3-2]
			\arrow["{T\class U}"{inner sep=.8ex}, "/"{marking}, curve={height=-12pt}, from=1-4, to=1-6]
			\arrow["{\class TU}"'{inner sep=.8ex}, "/"{marking}, from=1-4, to=3-4]
			\arrow[""{name=2, anchor=center, inner sep=0}, "{\class TU}"{description}, from=1-4, to=3-6]
			\arrow["{\class TU}"{inner sep=.8ex}, "/"{marking}, from=1-6, to=3-6]
			\arrow[curve={height=12pt}, equals, from=3-1, to=3-2]
			\arrow[""{name=3, anchor=center, inner sep=0}, curve={height=12pt}, equals, from=3-4, to=3-6]
			\arrow[""{name=4, anchor=center, inner sep=0}, curve={height=-12pt}, from=3-4, to=3-6]
			\arrow["{T!_U}", between={0.2}{0.8}, Rightarrow, from=0, to=1]
			\arrow["{\class T\class U}", between={0.2}{1}, Rightarrow, from=2, to=1-6]
			\arrow["{\class(\class TU)}"', between={0.2}{0.8}, Rightarrow, from=2, to=4]
			\arrow["{!_U}", between={0.2}{0.8}, Rightarrow, from=4, to=3]
		\end{tikzcd}
	\end{equation}
	Since $!_U \circ \class (\class TU)$ is an identity (as in the proof of \cref{lem:marked-lax-defn}), it suffices to show that $\class TU (T!_U) = \class (T\class U)^{-1}$, or that the following square is an identity:
	\begin{equation}
		\begin{tikzcd}
			TU && U \\
			TU & TU & U
			\arrow[""{name=0, anchor=center, inner sep=0}, "{\class TU}", from=1-1, to=1-3]
			\arrow[equals, from=1-1, to=2-1]
			\arrow[equals, from=1-3, to=2-3]
			\arrow[""{name=1, anchor=center, inner sep=0}, curve={height=-12pt}, from=2-1, to=2-2]
			\arrow[""{name=2, anchor=center, inner sep=0}, curve={height=12pt}, equals, from=2-1, to=2-2]
			\arrow["\class TU"', from=2-2, to=2-3]
			\arrow["{\class (T\class U)}", between={0.2}{1}, Rightarrow, from=0, to=2-2]
			\arrow["{T!_U}", between={0.2}{0.8}, Rightarrow, from=1, to=2]
		\end{tikzcd}
	\end{equation}
	But this follows by the uniqueness of condition (2) in $\K$, since it is a map $\class TU \Rightarrow \class TU \circ \id$.
\end{proof}

\subsection{Opfibrantly cartesian 2-monads}
We define here 2-monads we call \emph{(fully) opfibrantly cartesian} (\cref{def:opfib-cart}) so that they lift to $\stdoarrow{\K}$ as $T^{\discdown_s}$.
We then end up in the hypotheses of \cref{th:lift.of.enh.quasi-terminal.to.laxalg}, and we are thus able to lift the enhanced 2-classifier of $\K$, \emph{qua} marked-lax terminal object in $\stdoarrow{\K}$, to $\LaxAlg_{\lax, t\pseudo}(\stdoarrow{T})$.
It then suffices to prove that $\LaxAlg_{\lax, t\pseudo}(\stdoarrow{T}) \cong \stdoarrow{\LaxAlg_{\lax, t\pseudo}(T)}$, that is, that (lax) algebras of $\stdoarrow{T}$ are equivalent to small discrete opfibrations in $\LaxAlg_{\lax, t\pseudo}(T)$, and we do so in \cref{prop:stdos-in-laxalg-are-laxalg-in-stdos} by suitably defining the latter.

This yields \cref{mc-0003}, and we discuss some of its immediate corollaries.
We also readily get \cref{mc-0003-good}, which says that good enhanced 2-classifiers lift under the same hypotheses.

The elementary insight behind the theorem we are about to prove is that any algebra structure $\omega:T\Omega \to \Omega$ must correspond to a small discrete opfibration over $T\Omega$, and the most na\"ive guess we can make for this opfibration is to be given by $Tu$.

We thus assume the following:

\begin{definition}
\label[definition]{def:opfib-cart}
	Let $\K$ be a representable plumbus with enhanced 2-classifier $u:\Omega_\bullet \discto \Omega$.
	An enhanced 2-monad $(T,i,m)$ is \textbf{opfibrantly cartesian} when
	\begin{center}
		\begin{minipage}[c]{.65\textwidth}
			\begin{enumerate}[label=\Alph*.]
				\item $Tu$ is a small discrete opfibration,
				\item $T$ preserves pullbacks of $u$,
				\item $Tu$ is perfect (\cref{mc-002Z}), i.e. $\omega$ is tight,
			\end{enumerate}
			and \textbf{fully opfibrantly cartesian} when furthermore
			\begin{enumerate}
				\item[D.] $i$ and $m$ are cartesian\footnotemark~at $u$,
			\end{enumerate}
		\end{minipage}
		\hfill
		\begin{minipage}[c]{.34\textwidth}
			\begin{equation}
			\label[diagram]{mc-0029}
				\begin{tikzcd}[ampersand replacement=\&]
					{T\Omega_\bullet} \& {\Omega_\bullet} \\
					{T\Omega} \& \Omega
					\arrow["{\omega_\bullet}", from=1-1, to=1-2]
					\arrow["Tu"', two heads, from=1-1, to=2-1]
					\arrow["\lrcorner"{anchor=center, pos=0.125}, draw=none, from=1-1, to=2-2]
					\arrow["u", two heads, from=1-2, to=2-2]
					\arrow["\omega"', dashed, from=2-1, to=2-2]
				\end{tikzcd}
			\end{equation}
		\end{minipage}
	\end{center}
	\footnotetext{That is, their naturality squares at $u$ are pullbacks.}
\end{definition}

\begin{lemma}
\label[lemma]{lemma:opfib-cart}
	$(T,i,m)$ is opfibrantly cartesian if and only if
	\begin{enumerate}[label=\Alph*'.]
		\item \(T\) preserves small discrete opfibrations,
		\item \(T\) preserves pullbacks of small discrete opfibrations,
		\item $T$ preserves perfectness.
	\end{enumerate}
	and fully opfibrantly cartesian if and only if furthermore
	\begin{enumerate}
		\item[D'.] \(i\) and \(m\) are cartesian at all small discrete opfibrations,
	\end{enumerate}
\end{lemma}
\begin{proof}
	We only need to prove that \cref{def:opfib-cart} implies these conditions.
	The arguments are simple but illustrative.
	Thus let \(\omega = \class Tu\) be the map classifying \(Tu\), as above.

	For (A') and (B'), consider a small discrete opfibration \(p:E \discto B\), we have:
	\begin{equation}
	\label[equation]{eq:class-of-Tp}
		\begin{tikzcd}[ampersand replacement=\&]
			TE \& {T\Omega_\bullet} \& {\Omega_\bullet} \\
			TB \& {T\Omega } \& \Omega
			\arrow[squiggly, from=1-1, to=1-2]
			\arrow["Tp"', two heads, from=1-1, to=2-1]
			\arrow[squiggly, from=1-2, to=1-3]
			\arrow["Tu"', two heads, from=1-2, to=2-2]
			\arrow["\lrcorner"{anchor=center, pos=0.125}, draw=none, from=1-2, to=2-3]
			\arrow["u", two heads, from=1-3, to=2-3]
			\arrow["{{T\class p}}"', squiggly, from=2-1, to=2-2]
			\arrow["{\omega }"', from=2-2, to=2-3]
		\end{tikzcd}
		\qquad  = \qquad
		\begin{tikzcd}[ampersand replacement=\&]
			TE \&\& {\Omega_\bullet} \\
			TB \& {T\Omega} \& \Omega
			\arrow[squiggly, from=1-1, to=1-3]
			\arrow["Tp"', two heads, from=1-1, to=2-1]
			\arrow["u", two heads, from=1-3, to=2-3]
			\arrow["{{T\class p}}"', squiggly, from=2-1, to=2-2]
			\arrow["{\omega}"', from=2-2, to=2-3]
		\end{tikzcd}
	\end{equation}
	By pasting for pullbacks, the square above right is a pullback if and only if the left half of the square above left is a pullback, and this holds by assumption.
	Thus \(Tp\) is a small discrete opfibration classified by \(\omega (T\class p)\).

	Finally, for (C'), note that \cref{eq:class-of-Tp} shows $Tp$ is classified by a map isomorphic to $\omega(T \class p)$.
	When $p$ is perfect, $\class p$ is tight and so is $T \class p$.
	By assumption $\omega$ is tight, and thus we get $\class Tp$ is isomorphic to a tight map, and thus tight since $\K$ is repletely enhanced (Terminology-\ref{term:repl-enhanced}).

	As for (D'), this is an application of the pasting lemma for pullbacks, as illustrated below for $i$ ($m$ is analogous):
	\begin{equation}
		\begin{tikzcd}
			E & TE & {T\Omega_\bullet} \\
			B & TB & {T\Omega}
			\arrow["i", from=1-1, to=1-2]
			\arrow["p"', two heads, from=1-1, to=2-1]
			\arrow[squiggly, from=1-2, to=1-3]
			\arrow["Tp"', two heads, from=1-2, to=2-2]
			\arrow["\lrcorner"{anchor=center, pos=0.125}, draw=none, from=1-2, to=2-3]
			\arrow["Tu", two heads, from=1-3, to=2-3]
			\arrow["i"', from=2-1, to=2-2]
			\arrow["{T\class p}"', squiggly, from=2-2, to=2-3]
		\end{tikzcd}
		\quad = \quad
		\begin{tikzcd}
			E & {\Omega_\bullet} & {T\Omega_\bullet} \\
			B & {\Omega } & {T\Omega}
			\arrow[squiggly, from=1-1, to=1-2]
			\arrow["p"', two heads, from=1-1, to=2-1]
			\arrow["\lrcorner"{anchor=center, pos=0.125}, draw=none, from=1-1, to=2-2]
			\arrow["i", from=1-2, to=1-3]
			\arrow["u"', two heads, from=1-2, to=2-2]
			\arrow["\lrcorner"{anchor=center, pos=0.125}, draw=none, from=1-2, to=2-3]
			\arrow["Tu", two heads, from=1-3, to=2-3]
			\arrow["{\class p}"', squiggly, from=2-1, to=2-2]
			\arrow["i"', from=2-2, to=2-3]
		\end{tikzcd}
	\end{equation}
\end{proof}

From the proof, in particular \eqref{eq:class-of-Tp}, we evince:

\begin{corollary}
\label[corollary]{cor:class-of-Tp}
	When $T$ is opfibrantly cartesian, $\class Tp \cong \omega T\class p$.
\end{corollary}

Now note that conditions (A') and (B') above suffice to have any (fully) opfibrantly cartesian 2-monad lift to a (fully) marked (\cref{def:marked.2-monad}) enhanced 2-monad $\stdoarrow{T}$ on $\stdoarrow{\K}$, whose action is simply that of $T$ on arrows:
\begin{equation}
	\begin{tikzcd}[ampersand replacement=\&]
		E \& {E'} \\
		B \& {B'}
		\arrow["g", squiggly, from=1-1, to=1-2]
		\arrow["p"', two heads, from=1-1, to=2-1]
		\arrow["{p'}", two heads, from=1-2, to=2-2]
		\arrow["f"', squiggly, from=2-1, to=2-2]
	\end{tikzcd}
	\quad \mapsto \quad
	\begin{tikzcd}[ampersand replacement=\&]
		TE \& {E'} \\
		TB \& {TB'}
		\arrow["Tg", squiggly, from=1-1, to=1-2]
		\arrow["Tp"', two heads, from=1-1, to=2-1]
		\arrow["{Tp'}", two heads, from=1-2, to=2-2]
		\arrow["Tf"', squiggly, from=2-1, to=2-2]
	\end{tikzcd}
\end{equation}
Indeed, by (A') this is well-defined on objects and by (B') it preserves the marking on $\stdoarrow{\K}$, i.e. the pullback squares (recall \cref{const:marking.codotting.stdopfib}).
The unit and multiplication are given by their naturality squares, and in fact note also that the additional condition (D') makes $\stdoarrow{T}$ fully marked, by definition.
Finally, by (C'), $Tu$ is codotted (codotting in $\stdoarrow{\K}$ is indeed given by perfect small discrete opfibrations).
Thus an opfibrantly cartesian 2-monad satisfies the precise conditions that make \cref{th:lift.of.enh.quasi-terminal.to.laxalg} applicable to $\stdoarrow{T}$ over $\stdoarrow{\K}$, and thus we get:

\begin{proposition}
\label[proposition]{mc-0015}
	When $(T,i,m)$ is opfibrantly cartesian, the enhanced 2-classifier $u:\Omega_\bullet \discto \Omega$ becomes equipped with an essentially unique lax $\stdoarrow{T}$-algebra structure.
	In particular, this makes $u$ a marked-lax terminal object in $\LaxAlg_{\lax, t\pseudo}(\stdoarrow{T})$.
	If, furthermore, $(T, i, m)$ is fully opfibrantly cartesian, then $u$ is marked-lax terminal in $\PsAlg_{\lax, t\pseudo}(\stdoarrow{T})$.
\end{proposition}

We now want to prove that `small discrete opfibrations in $\LaxAlg_{\lax, t\pseudo}(\stdoarrow{T})$ are the same as lax $\stdoarrow{T}$-algebras'.
This involves defining a smallness structure on $\LaxAlg_{\lax, t\pseudo}(\stdoarrow{T})$.
We start by observing the following:

\begin{proposition}
	Every lax $\stdoarrow{T}$-algebra is a tight discrete opfibration in $\LaxAlg_{\lax, t\pseudo}(\tdoarrow{T})$, that is, there is a `forgetful' marked enhanced 2-functor:
	\begin{equation}
		\begin{tikzcd}[ampersand replacement=\&]
			{\LaxAlg_{\lax, t\pseudo}(\stdoarrow{T})} \& {\tdoarrow{\LaxAlg_{\lax, t\pseudo}(T)}.}
			\arrow["U", from=1-1, to=1-2]
		\end{tikzcd}
	\end{equation}
\end{proposition}
\begin{proof}
	The claim is obvious once we unpack what the left hand side looks like.
	A lax algebra for $\stdoarrow{T}$ is given by a small discrete opfibration $p:E \discto B$ in $\K$ equipped with a tight square
	\begin{equation}
		\begin{tikzcd}[ampersand replacement=\&]
			TE \& E \\
			TB \& B
			\arrow["\eta", from=1-1, to=1-2]
			\arrow["Tp"', two heads, from=1-1, to=2-1]
			\arrow["p", two heads, from=1-2, to=2-2]
			\arrow["\beta"', from=2-1, to=2-2]
		\end{tikzcd}
	\end{equation}
	and unitor and multiplicator
	\begin{equation}
		\begin{tikzcd}[ampersand replacement=\&, sep=scriptsize]
			E \&\&\& E \\
			\& TE \\
			B \&\&\& B \\
			\& TB
			\arrow["i"', from=3-1, to=4-2]
			\arrow["\beta"', from=4-2, to=3-4]
			\arrow[""{name=1, anchor=center, inner sep=0}, equals, from=3-1, to=3-4]
			\arrow["{\psunit_B}"{pos=0.3}, between={0.2}{0.8}, Rightarrow, from=1, to=4-2]
			\arrow[""{name=0, anchor=center, inner sep=0}, shift left, equals, from=1-1, to=1-4]
			\arrow["i"', from=1-1, to=2-2]
			\arrow["p"', two heads, from=1-1, to=3-1]
			\arrow["p", two heads, from=1-4, to=3-4]
			\arrow["\eta"', from=2-2, to=1-4]
			\arrow[two heads, from=2-2, to=4-2]
			\arrow["{\psunit_E}"{pos=0.3}, between={0.2}{0.8}, Rightarrow, from=0, to=2-2]
		\end{tikzcd}
		\hspace*{10ex}
		\begin{tikzcd}
			& {T^2E} && \\
			TE &&& TE \\
			& {T^2B} & E \\
			TB &&& TB \\
			&& B
			\arrow[""{name=0, anchor=center, inner sep=0}, "{T\eta}"', from=1-2, to=2-1]
			\arrow["{T^2p}"{description}, two heads, from=1-2, to=3-2]
			\arrow["m", from=1-2, to=2-4]
			\arrow["\eta"{description, pos=0.7}, from=2-1, to=3-3]
			\arrow["Tp"{description}, two heads, from=2-1, to=4-1]
			\arrow[""{name=1, anchor=center, inner sep=0}, "\eta", from=2-4, to=3-3]
			\arrow["Tp", two heads, from=2-4, to=4-4]
			\arrow[""{name=2, anchor=center, inner sep=0}, "{T\eta}"', from=3-2, to=4-1]
			\arrow["m"{description, pos=0.3}, from=3-2, to=4-4]
			\arrow["\beta"', from=4-1, to=5-3]
			\arrow[""{name=3, anchor=center, inner sep=0}, "\beta", from=4-4, to=5-3]
			\arrow["{\psmult_E}", between={0.2}{0.8}, Rightarrow, from=0, to=1]
			\arrow["{\psmult_B}"', between={0.2}{0.8}, Rightarrow, from=2, to=3]
			\arrow["p"{description}, two heads, from=3-3, to=5-3]
		\end{tikzcd}
	\end{equation}
	satisfying the equations from \cref{def:lax-alg}.

	Obviously $p$ is a tight discrete opfibration in $\K$ and a strict $T$-morphism of lax $T$-algebras.
	Then the only non-trivial step remaining in the proof has been done in \cref{prop:tdos-in-alg-lax}---proving that $p$ is in fact a tight discrete opfibration in \(\LaxAlg_{\lax, t\pseudo}(T)\).

	We leave the reader to check that this carries over to 1-cells and 2-cells, with no surprises: every lax/pseudo $\stdoarrow{T}$-morphism structure on a square induces lax/pseudo $T$-morphism structures on its sides.

	Finally, we check the various decorations on the 2-categories correspond.
	A tight map in $\LaxAlg_{\lax, t\pseudo}(\stdoarrow{T})$ is given by a square with tight sides and invertible laxators, which amounts precisely to a tight in $\stdoarrow{\LaxAlg_{\lax, t\pseudo}(T)}$.
	A marked map in $\LaxAlg_{\lax, t\pseudo}(\stdoarrow{T})$ is one whose carrier is a pullback square, and so is in $\LaxAlg_{\lax, t\pseudo}(\stdoarrow{T})$.
	We are done.
\end{proof}

Note that $U$ is locally an isomorphism.
We can therefore define the smallness structure on \(\LaxAlg_{\lax, t\pseudo}(T)\) as the image of $U$, that is, the `$T$-strict small discrete opfibration in $\K$'.

\begin{proposition}
\label[proposition]{prop:lax-alg-plumbus}
	When \(\K\) is a plumbus (\cref{def:plumbus}) and \((T,i,m)\) an enhanced 2-monad, the enhanced 2-category \(\LaxAlg_{\lax, t\pseudo}(T)\) is a plumbus too, with small discrete opfibrations the strict $T$-morphisms whose underlying map is a small discrete opfibration in $\K$.
\end{proposition}
\begin{proof}
\label[proof]{mc-002F}
	The definition of the enhanced 2-category \(\LaxAlg_{\lax, t\pseudo}(T)\) has been given in \cref{def:lax-alg}---note that pseudo-$T$-maps are a replete choice of tights.
	The tight terminal object is inherited trivially, and left-tight pullbacks of tight discrete opfibrations exist by \cref{prop:pbs-of-tdos-in-alg-lax} and are constructed out of pullbacks in $\K$: we may therefore reuse the same choice of pullbacks in $\LaxAlg_{\lax, t\pseudo}(T)$ that we made for $\K$.
\end{proof}

The following then holds by definition:

\begin{proposition}
\label{prop:stdos-in-laxalg-are-laxalg-in-stdos}
	Let $(T,i,m)$ be opfibrantly cartesian on a plumbus $\K$, then
	\begin{equation}
		\LaxAlg_{\lax, t\pseudo}(\stdoarrow{T}) \cong \stdoarrow{\LaxAlg_{\lax, t\pseudo}(T)}
	\end{equation}
	as marked enhanced 2-categories (with respect to the decorations defined in \cref{sec:formal-2-class}).
\end{proposition}

\begin{remark}
	\cref{prop:lax-alg-plumbus} and \cref{prop:stdos-in-laxalg-are-laxalg-in-stdos} evidently work for other combinations of laxity of the algebra and the morphisms.
\end{remark}

Finally, this isomorphism allows to transport the marked-lax terminal object $u:\Omega_\bullet \discto \Omega$ of $\LaxAlg_{\lax, t\pseudo}(\stdoarrow{T})$ to $\stdoarrow{\LaxAlg_{\lax, t\pseudo}(T)}$, therefore (by \cref{mc-0045}) making $\stdoarrow{\LaxAlg_{\lax, t\pseudo}(T)}$ a representable plumbus.
We have thus proved our main theorem.

\begin{theorem}
\label[theorem]{mc-0003}
	Let $(T,i,m)$ be an opfibrantly cartesian 2-monad on a representable plumbus $\K$.
	Then the plumbus \(\LaxAlg_{\lax, t\pseudo}(T)\) is representable too, with the same 2-classifier \(u : \Omega_\bullet \discto \Omega \), where $\Omega$ and $\Omega_\bullet$ are equipped with the lax $T$-algebra structure described in \cref{mc-0015} and $u$ is strict.
\end{theorem}
\begin{corollary}
\label[corollary]{mc-0003-pseudo}
	If $T$ is fully opfibrantly cartesian, \(\PsAlg_{\lax, t\pseudo}(T)\) is a representable plumbus.
\end{corollary}

It is useful to unpack the construction we did at the formal level, specifically in \eqref{eq:U-alg-struct}, to actually see how $u$ performs its job in $\LaxAlg_{\lax, t\pseudo}(T)$.

First, let us mention that the structure map of the algebra structure on $\Omega$ and $\Omega_\bullet$---as anticipated in \cref{def:opfib-cart}---is precisely the pullback square classifying $Tu$.

To give an intuition of how this algebra structure `feels like', we anticipate here an example (many others can be found in \cref{sec:examples}).

\begin{example}[{Monoidal categories}]\label[example]{djm-00KC}
	Monoidal categories are the pseudoalgebras of a 2-monad $T$ on $\Cat$, specifically the `free monoid' 2-monad, i.e. the free strict monoidal category 2-monad (Example 4.1.15, \cite{leinsterHigherOperadsHigher2004}).
	$\Cat$ is chordate, and we consider it with the representable smallness structure described in \cref{ex:2-class-cat}.
	It is straightforward to see that $T$ is fully opfibrantly cartesian---indeed $Tu$ is necessarily perfect since all classifying morphisms are tight, $T$ has rank and thus does not change smallness, and $T$ is cartesian (see \emph{ibid.}).

	Now we have that the following square is a pullback, by definition:
	\begin{equation}
		\begin{tikzcd}
			{T(1  / \Set)} & {1  / \Set} \\
			{T\Set} & \Set
			\arrow["{{1  / \omega }}", from=1-1, to=1-2]
			\arrow[from=1-1, to=2-1]
			\arrow["\lrcorner"{anchor=center, pos=0.125}, draw=none, from=1-1, to=2-2]
			\arrow["u", from=1-2, to=2-2]
			\arrow["{\omega }"', dashed, from=2-1, to=2-2]
		\end{tikzcd}
	\end{equation}
	The fiber of \(Tu\) over an \(X \in \Set^n\) (where $n \in \N$) in \(T\Set\) comprises of the possible families $(X, x) : (1/\Set)^n$, so that \(x = (x_i \in X_i)_{0 \leq i \leq n}\).
	Hence we see that \(\omega(X) = \prod_{i \in I} X_i\) computes the cartesian product of sets.

	In any case, \cref{mc-0003} now confirms to us that the category of sets, together with its monoidal structure of cartesian product, classifies strict discrete opfibrations of monoidal categories \(p : E \discto B\) via lax monoidal functors \(\class p : C \to \Set\).
	We have recovered the discrete version of the \emph{monoidal Grothendieck construction} \cite{moellerMonoidalGrothendieckConstruction2020}.
	We refer the reader to \cref{djm-00KD} for generalizations and expansion of this example.
\end{example}

In general, for $T$ algebraic, we see that the algebra structure on $\Omega$ is always going to compute a sort of `limit', whose shape is given by the operations of $T$.

\begin{remark}[{Decategorification}]\label[remark]{mc-002W}
	\cref{mc-0003} should decategorify to a theorem about 1-topoi.
	For instance, it is easy to see that a monad for which \(Tu\) is a mono (where \(u\) is the subobject classifier) induces a \(T\)-algebra structure on \(\Omega \), namely the one given by
	\begin{equation}
		t(\varphi _1, \ldots , \varphi _n)\ \mapsto\ \bigwedge _i \varphi _i.
	\end{equation}
	A result in this direction, involving classifiers for a restricted class of subobjects of \(T\)-algebras, appears in \citep{aristoteMonotoneWeakDistributive2024}.
\end{remark}

\subsubsection{Well-representability}
Finally, suppose now $\K$ is well-representable.
By the equivalence of classification explained in \cref{diag:good-classifying-commas}, we have a comma square:
\begin{equation}
\label[diagram]{mc-0029-good}
	\begin{tikzcd}[ampersand replacement=\&]
		{T\Omega_\bullet} \& 1 \\
		{T\Omega} \& \Omega
		\arrow[from=1-1, to=1-2]
		\arrow["Tu"', two heads, from=1-1, to=2-1]
		\arrow["\lrcorner"{anchor=center, pos=0.125}, draw=none, from=1-1, to=2-2]
		\arrow[between={0.2}{0.8}, Rightarrow, from=1-2, to=2-1]
		\arrow["\tau", from=1-2, to=2-2]
		\arrow["\omega"', from=2-1, to=2-2]
	\end{tikzcd}
\end{equation}
Speaking informally for the moment, note that the above shows again how the algebra structure \(\omega : T\Omega \to \Omega \) must be given by a sort of limit construction: the `elements' of \(\omega (t(\vec {X}))\) (i.e. maps \(\tau \twoto \omega (t(\vec {X}))\)) for a \(T\)-term \(t(\vec {X})\) are given by terms \(t(\vec {x})\) of `elements' \(x_i : \tau \twoto X_i\).

Moreover, we have:

\begin{proposition}
\label[proposition]{proposition:tau-is-pseudo}
	\(\tau \) is tight in \(\LaxAlg_{\lax, t\pseudo}(T)\) and $(\Omega, \omega)$ is cartesian at $\tau$.
\end{proposition}
\begin{proof}
	We know $\tau = \class(\id_1)$ is tight in $\K$ by assumption, we only need to prove it a pseudo-$T$-morphism.
	Since $\tau$ is internally terminal (\cref{id-1-is-terminal}), there is a unique colax structure $\omega T\tau \twoto \tau!$ and it is automatically coherent.
	Moreover, since $\tau$ is right adjoint to $\Omega \to 1$, and the latter is a pseudo $T$-morphism (in fact, strict), then by doctrinal adjunction \citep{kellyDoctrinalAdjunction1974} $\tau$ must in fact be pseudo as well.
\end{proof}

Thus we conclude that \cref{mc-0003} works as-is for well-representable plumbuses:

\begin{theorem}
\label[theorem]{mc-0003-good}
	Let $(T,i,m)$ be an opfibrantly cartesian 2-monad on a well-representable plumbus $\K$.
	Then the plumbus \(\LaxAlg_{\lax, t\pseudo}(T)\) is well-representable too, with the same good 2-classifier \(\tau : 1 \to \Omega \), where $\Omega$ is equipped with its unique lax $T$-algebra structure and $\tau$ with its unique pseudo-$T$-map structure.
	Furthermore, when $T$ is fully opfibrantly cartesian, the plumbus \(\PsAlg_{\lax, t\pseudo}(T)\) is analogously well-representable.
\end{theorem}

\subsection{When are representables pseudo?}
\label[section]{sec:cartesianness}

Let $p:(E,\eta) \discto (B,\beta)$ be a small discrete opfibration of (lax) $T$-algebras: what does the lax structure on $\class p : B \looseto \Omega$ look like?
The laxator fits in the diagram (note $\overline{\class_\bullet p}$ is the lift of $\overline{\class p}$ along $u$):%
\begin{equation}
	\begin{tikzcd}
		& TE && \\
		E &&& {T\Omega_\bullet} \\
		& TB & {\Omega_\bullet} \\
		B &&& {T\Omega} \\
		&& \Omega
		\arrow["\eta"', from=1-2, to=2-1]
		\arrow[""{name=0, anchor=center, inner sep=0}, "{T\class_\bullet p}", squiggly, from=1-2, to=2-4]
		\arrow["Tp"{description}, two heads, from=1-2, to=3-2]
		\arrow[""{name=1, anchor=center, inner sep=0}, "{\class_\bullet p}"{description, pos=0.7}, squiggly, from=2-1, to=3-3]
		\arrow["p"', two heads, from=2-1, to=4-1]
		\arrow["{\omega_\bullet}", from=2-4, to=3-3]
		\arrow["Tu", two heads, from=2-4, to=4-4]
		\arrow["\beta"', from=3-2, to=4-1]
		\arrow[""{name=2, anchor=center, inner sep=0}, "{T\class p}"{description, pos=0.3}, squiggly, from=3-2, to=4-4]
		\arrow[""{name=3, anchor=center, inner sep=0}, "u"{description}, two heads, from=3-3, to=5-3]
		\arrow[""{name=4, anchor=center, inner sep=0}, "{\class p}"', squiggly, from=4-1, to=5-3]
		\arrow["\omega", from=4-4, to=5-3]
		\arrow["{\overline{\class_\bullet p}}", between={0.2}{0.8}, Rightarrow, from=0, to=1]
		\arrow["\lrcorner"{anchor=center, pos=0.125, rotate=-90}, draw=none, from=2-4, to=3]
		\arrow["{\overline{\class p}}"', between={0.2}{0.8}, Rightarrow, from=2, to=4]
	\end{tikzcd}
\end{equation}
By classification, the bottom map correspond to a map $Tp \to \beta^*p$ over $TB$, namely the following gap map:
\begin{equation}
\label{eq:cart-defect}
	\begin{tikzcd}[ampersand replacement=\&]
		TE \\
		\& {\beta ^*E} \& E \\
		\& TB \& B
		\arrow["{\delta _{p}}"{description}, dashed, from=1-1, to=2-2]
		\arrow["{\eta }", curve={height=-12pt}, from=1-1, to=2-3]
		\arrow["Tp"', curve={height=12pt}, two heads, from=1-1, to=3-2]
		\arrow[from=2-2, to=2-3]
		\arrow[two heads, from=2-2, to=3-2]
		\arrow["{\lrcorner }"{pos=0.125}, draw=none, from=2-2, to=3-3]
		\arrow["p", two heads, from=2-3, to=3-3]
		\arrow["{\beta }"', from=3-2, to=3-3]
	\end{tikzcd}
\end{equation}
Concretely, this means $\class p$ is given by maps $t(\vec{E_b}) \to E_{t(\vec b)}$ which compare a term of fibers with the fiber over a term.
For instance, for monoidal discrete opfibrations, this would be a map $E_a \times E_b \to E_{a \otimes b}$.

\begin{definition}[{Cartesian \(T\)-morphism}]\label[definition]{def:cart-defect}
	Let \(f:A \to B\) be a strict \(T\)-morphism.
	We call \textbf{\(T\)-cartesianness defect of \(f\)} the comparison map dashed in \eqref{eq:cart-defect}.
	We say \(f\) is \textbf{\(T\)-cartesian}\footnotemark~when its defect is a split equivalence (that is, when it admits an inverse $i$ and $\varepsilon : i\delta _{f} \cong \id$ such that $\delta _{f}i = \id$), and \textbf{strictly \(T\)-cartesian} when it is an isomorphism.
	\footnotetext{
		We deem the terminology `$T$-cartesian' pretty fitting for the concept, although sometimes in $\infty$-categorical literature the `cartesian maps' are usually taken to be fibrations of sorts.
		We keep the prefix `$T$-' to avoid any confusion.
	}
\end{definition}

We can now observe the following:

\begin{corollary}[{of \cref{mc-0003}}]
\label[corollary]{mc-001R}
	Let \(p:E \discto B\) be a small discrete opfibration in \(\LaxAlg_{\lax, t\pseudo}(T)\).
	Its classifying map \(\class p : B \looseto \Omega \) is pseudo if and only if \(p\) is \(T\)-cartesian.
\end{corollary}

Clearly these considerations apply to $\PsAlg_{\lax, t\pseudo}(\stdoarrow{T})$ too.

\begin{remark}
	\cref{mc-001R} shows the isomorphism of \cref{prop:stdos-in-laxalg-are-laxalg-in-stdos} respects codotting too.
	Indeed, a codotted object in $\LaxAlg_{\lax, t\pseudo}(\stdoarrow{T})$ is a perfect small discrete opfibration whose structure map, \emph{qua} $\stdoarrow{T}$-algebra, is a pullback, that is, a $T$-cartesian small discrete opfibration in \(\LaxAlg_{\lax, t\pseudo}(T)\).
\end{remark}

\begin{remark}
\label{rmk:strict-classifying-map}
	Note that to have $\class p$ be \emph{strict} instead is much harder: it should be the case that the composite
	\begin{equation}
		\begin{tikzcd}[ampersand replacement=\&,column sep=scriptsize]
			{(T \class p)^* Tu} \& {TE} \& {\beta^*E}
			\arrow["\sim", from=1-1, to=1-2]
			\arrow["{\delta_p}", from=1-2, to=1-3]
		\end{tikzcd}
	\end{equation}
	is the identity, which can only happen for $p=u$.
\end{remark}

Inspired by \cref{mc-001R}, in this section we pause to contemplate the notion of `cartesianness'.
The notion of (strict) \(T\)-cartesian morphism, in the 1-dimensional setting, has been singled out before by Leinster who studied them under the name of \emph{pullback-homomorphisms} in Appendix B of \cite{leinsterNotionsMobiusInversion2012}.
He shows \(\mathrm{fc}\)-cartesian morphisms, where \(\mathrm{fc}\) is the free category monad on graphs of sets, are precisely the functors enjoying the unique lifting of factorizations (ULF) property.

In the following, it will be enough to assume \(T\) is simply a 2-monad (but note for most of the definitions below, strictness is not necessary) on a 2-category \(\K\) admitting the limits we use throughout.

The first thing we want to observe is that when $p$ is representable (\cref{mc-001T}) then we get a pretty interesting specialization of the concept:

\begin{definition}
\label[definition]{mc-0004}
	Let \((A,\alpha )\) be a (lax/pseudo/strict) \(T\)-algebra and \(a:X \rightsquigarrow A\) a colax \(T\)-morphism.
	We say \(\alpha \) is \textbf{(strictly) cartesian at \(a\)} when the representable discrete opfibration at \(a\) (\cref{mc-001T}) is a (strictly) \(T\)-cartesian morphism:
	\begin{equation}
		\begin{tikzcd}
			{T(a/A)} & TA \\
			{a/A} & A
			\arrow["{T\partial_1}", from=1-1, to=1-2]
			\arrow["{\overline {a} / \alpha }"', from=1-1, to=2-1]
			\arrow["\lrcorner "{anchor=center, pos=0.125}, draw=none, from=1-1, to=2-2]
			\arrow["\alpha ", from=1-2, to=2-2]
			\arrow["{\partial_1}", from=2-1, to=2-2]
		\end{tikzcd}
	\end{equation}
	This means (1) $T(a/A)$ is split-equivalent to $\alpha \times_A \partial_1$ and (2) the resulting map $\overline{a}/\alpha$ makes $a/A$ a $T$-algebra of the appropriate kind.
\end{definition}

Thus, as a corollary of \cref{mc-001R}, we have:

\begin{corollary}
\label[corollary]{mc-001U}
	The representable copresheaf \({B(b,-) : B \looseto \Omega} \) (\cref{mc-0030}) associated to a representable small discrete opfibration \(b/B \overset{\partial_1}\discto B\) is always a lax \(T\)-morphism, and it is pseudo precisely when \(B\) is \(T\)-cartesian at \(b\).
\end{corollary}

However, it seems that the notion of $T$-cartesian morphism has a natural place in 2-algebra, besides its relation to the classification of discrete opfibration.
An example we find illuminating is the following:

\begin{example}
\label[example]{mc-001E}
	Let \(T=\mathrm{Fam}_{\mathrm{f}}\) be the free finite coproduct completion 2-monad on \(\mathrm{Cat}\).
	Then a category with coproducts \(A\), i.e. a \(\mathrm{Fam}_{\mathrm{f}}\)-algebra, is cartesian at \(\id_A\) iff it is extensive.

	To prove the claim we made, observe the comma algebra \(A/\coprod  : \mathrm{Fam}_{\mathrm{f}}(A/A) \to  A/A\) is given by
	\begin{equation}
		(I \in  \Finset, (\varphi _i :x_i \to  y_i)_{i \in  I}) \ \mapsto \ \coprod _{i \in  I} \varphi _i : \coprod _{i \in  I} x_i \to  \coprod _{i \in  I} y_i
	\end{equation}
	Consider
	\begin{equation}
		\begin{tikzcd}
		{\mathrm{Fam}_{\mathrm{f}}(A/A)} \\
		& {\mathrm{Fam}_{\mathrm{f}}(A) \times _A A/A} & {\mathrm{Fam}_{\mathrm{f}}(A)} \\
		& {A/A} & A
		\arrow["{\delta _{\partial_1}}"{description}, dashed, from=1-1, to=2-2]
		\arrow[curve={height=-13pt}, from=1-1, to=2-3]
		\arrow[curve={height=13pt}, from=1-1, to=3-2]
		\arrow[from=2-2, to=2-3]
		\arrow["{A/\coprod }"', from=2-2, to=3-2]
		\arrow["\lrcorner "{anchor=center, pos=0.125}, draw=none, from=2-2, to=3-3]
		\arrow["\coprod ", from=2-3, to=3-3]
		\arrow["{\partial_1}"', from=3-2, to=3-3]
		\end{tikzcd}
	\end{equation}
	The \(\mathrm{Fam}_{\mathrm{f}}\)-cartesianness defect of \(\partial_1 : A/A \to A\) sends \((I, (\varphi _i :x_i \to y_i)_i)\) to \(((I,(y_i)_i), \coprod _i \varphi _i : \coprod _i x_i \to \coprod _i y_i)\).
	To say this is a split equivalence means saying that for all \(\psi :X \to \coprod _i y_i\) there exists a family \((\psi _i :x_i \to y_i)_i\) unique up to isomorphism such that \(\coprod _i \psi _i \cong \psi \).
	This is precisely saying that the canonical map \(\prod _i A / y_i \to A/(\coprod _i y_i)\) is an equivalence of categories, i.e. that \(A\) is extensive (\cite{carboniIntroductionExtensiveDistributive1993}).
\end{example}

\begin{example}[{Descent}]\label[example]{mc-003F}
	Generalizing the above, suppose \(T\) is a free \(\Phi \)-cocompletion 2-monad on \(\Cat\), with \(\Phi \) a class of weights.
	Then \(A\) is \(T\)-cartesian at the classifying object, i.e the map \(\partial_1 : A/A \to A\) is \(T\)-cartesian, precisely when \(A\) satisfy descent for \(\Phi \)-colimits, i.e. \(\lim _i A / a_i \simeq A / \operatorname {colim}_i a_i\).
\end{example}

\begin{example}[{Extensive monoidal category}]\label[example]{mc-001Q}
	It's easy to see the argument above for extensive categories applies more generally to \(T =\) the free symmetric monoidal category 2-monad, thus recovering the notion of \emph{extensive monoidal category} of Example 9.2 in \cite{galvez-carrilloDecompositionSpacesIncidence2018}.
	Hence extensive monoidal categories are characterized as those monoidal categories cartesian at the identity.
\end{example}

In light of these examples, we might define an \textbf{extensive \(T\)-algebra} to be one cartesian at the identity.

\begin{example}
\label[example]{mc-002I}
	In the special case \(X = 1\), being cartesian at \(a\) means that for every term \(t(b_1, \ldots , b_n)\), it is the case that
	\begin{equation}
		A(a, t(b_1, \ldots , b_n)) \cong t(A(a, b_1), \ldots , A(a, b_n)).
	\end{equation}
	For \(T\) the free symmetric monoidal category 2-monad on \(\Cat\) is easy to see this corresponds to the characteristic representability of products.%

	In general, the above is quite literally expressing the fact that the corepresentables at \(a\), which normally are just lax \(T\)-morphisms, be actually \emph{pseudo}.
	This is \cref{mc-001U}.
\end{example}

\begin{example}[{Spanish double categories}]\label[example]{mc-002O}
	A \textbf{spanish} double category (`span-like' in Definition 5.3.1.5 of \cite{myersCategoricalSystemsTheory2023}) is one where any 2-cell out of a unit and into a composite, as below left, factors uniquely as below right:

	\begin{equation}
		\begin{tikzcd}
			{\cdot } & {\cdot } \\
			& {\cdot } \\
			{\cdot } & {\cdot }
			\arrow[from=1-1, to=1-2]
			\arrow[""{name=0, anchor=center, inner sep=0}, "\bullet "{marking}, equals, from=1-1, to=3-1]
			\arrow["\bullet "{marking}, from=1-2, to=2-2]
			\arrow["\bullet "{marking}, from=2-2, to=3-2]
			\arrow[from=3-1, to=3-2]
			\arrow[shorten <=7pt, shorten >=5pt, Rightarrow, from=0, to=2-2]
			\end{tikzcd}
			\quad  = \quad
			\begin{tikzcd}
			{\cdot } & {\cdot } \\
			{\cdot } & {\cdot } \\
			{\cdot } & {\cdot }
			\arrow[from=1-1, to=1-2]
			\arrow[""{name=0, anchor=center, inner sep=0}, "\bullet "{marking}, equals, from=1-1, to=2-1]
			\arrow[""{name=1, anchor=center, inner sep=0}, "\bullet "{marking}, from=1-2, to=2-2]
			\arrow[from=2-1, to=2-2]
			\arrow[""{name=2, anchor=center, inner sep=0}, "\bullet "{marking}, equals, from=2-1, to=3-1]
			\arrow[""{name=3, anchor=center, inner sep=0}, "\bullet "{marking}, from=2-2, to=3-2]
			\arrow[from=3-1, to=3-2]
			\arrow[shorten <=10pt, shorten >=10pt, Rightarrow, from=0, to=1]
			\arrow[shorten <=10pt, shorten >=10pt, Rightarrow, from=2, to=3]
		\end{tikzcd}
	\end{equation}
	Consider now the free double category 2-monad \(T\) on \(\mathrm{Graph}(\Cat)\) (as considered in \cref{sec:intro}).
	A strict object \(a:1 \to A\) in a double category \(A\) is simply an object of \(A\), and \(A\) is spanish exactly when it is \(T\)-cartesian at all its objects.
	Indeed, an object of \(T(a/A)\) is a composite of triangles such as above right, while \(\alpha \times _A \partial_1\) is comprised of single triangles as above left.
	The cartesianness defect of \(a\) composes a series of triangles, and asking this map invertible means having unique factorizations as required.

	\cref{mc-001U} instantiated for this example yields the fact that spanish double categories have strong (Paré) representable (§2.1 in \citep{pareYonedaTheoryDouble2011}).
	This observation plays a major role in the compositionality theorem of \cite{myersCategoricalSystemsTheory2023} (Theorem 5.3.3.1 there).
\end{example}

Note the difference between \cref{mc-002I} and \cref{mc-002O}: in the former, objects $a:1 \to A$ are considered \emph{colax}, whereas in the latter they are \emph{strict}.

\section{Examples and applications}
\label[section]{sec:examples}

In this section, we will deploy our main theorem (\cref{mc-0003}) to construct a number of examples of strict discrete opfibration classifiers in 2-categories of algebras and lax morphisms.
We will begin with an example that emphasizes the difference between strict discrete opfibration classifiers amongst lax morphisms and discrete opfibration classifiers amongst strict morphisms.

\begin{example}[{Presheaves with lax transformations}]\label[example]{mc-002U}
	Consider a category \(C\), with set of objects \(i : C_0 \hookrightarrow C\).
	It is a classical observation (see e.g. Example 6.6 \cite{blackwellTwodimensionalMonadTheory1989}) that functoriality of an assignment $C_0 \to \Cat$ is an algebraic structure, for \(T\) the `free indexed category' 2-monad \(T : X \mapsto \sum _{c \in C} C(c,-) \times X(c)\) induced by the 2-adjunction \(\mathrm{lan}_i \dashv i^*\).

	Thus \(\LaxAlg_{\lax,\pseudo}(T)\) is the enhanced 2-category \([C, \Cat]_{\lax,\pseudo}^\lax\) of lax \(C\)-indexed categories, lax natural transformation between them, and modifications, with pseudo transformations as tight maps.

	Notice that \(T\) satisfies the conditions of \cref{def:opfib-cart}, since everything is checked pointwise, and moreover it is finitary over the well-representable plumbus \(\Cat^{C_0}\) (we leave to the reader the easy task of showing that the plumbus structure of \cref{ex:2-class-cat} extends to any discrete power thereof).
	So our \cref{mc-0003-good} equips \([C, \Cat]_{\lax, \pseudo}^\lax\) with a 2-classifier, namely \(1 \xrightarrow {\underline {1}} \underline {\Set}\) which picks a (fixed) singleton \(1 \in \Set\) at each \(c \in C_0\) (we are denoting with \(\underline {X}\) the indexed category constant at \(X \in \Cat\)).

	This is very different from the 2-classifier described by Mesiti in \cite{mesiti2classifiersDenseGenerators2025}, which we described in \cref{mc-002P}!
	There is no contradiction though, since Mesiti's 2-classifier is for the 2-category \([C, \Cat]_{\strict}\) of indexed categories with \emph{strict} natural transformations between them.
	Classification therein is significantly different, since, as we have seen above, classifying maps are seldom strict \(T\)-morphisms.
	Indeed, one can see that a strict \(T\)-morphism \(B \to \underline {\Set}\) necessarily classifies a trivial opfibration, i.e. ones of the form \(\pi _X:\underline {E} \times B \to B\) for a fixed category \(E\).
\end{example}

\begin{example}[{Opfamilial 2-monads}]\label[example]{djm-00KD}
	Let's expand on \cref{djm-00KC}.
	Rather famously, symmetric monoidal categories are also the \emph{strict} algebras for a 2-monad \(T\) on the 2-category \(\Cat\) of categories.
	Explicitly, if \(\Finset^{\cong }\) is the (skeletal) groupoid of (standard) finite sets and bijections and \(i : \Finset^{\cong } \hookrightarrow \Cat\) is the inclusion of the finite discrete categories, then \(T C\) may be defined as the bijective-on-objects/fully-faithful factor of the inclusion \(\mathsf{Tree}_{0,2}(C) \to \Finset^{\cong } /^{\colax} C\) from the free pointed magma of nullary-or-binary trees with leaves labeled in $C$ to the colax slice of over \(C\), sending a tree to its set of leaves.

	This is an instance of \emph{polynomial 2-monad} in the sense of \cite{weberFamilial2functorsParametric2007}, represented by the polynomial
	\begin{equation}
		\begin{tikzcd}
		\ast  & {\Finset^{\cong }_{\ast }} & {\Finset^{\cong }} & \ast
		\arrow[from=1-2, to=1-1]
		\arrow[from=1-2, to=1-3]
		\arrow[from=1-3, to=1-4]
		\end{tikzcd}
	\end{equation}
	Let \(T : \Cat / I \to  \Cat / I\) be a polynomial 2-monad represented by a polynomial
	\begin{equation}
		\begin{tikzcd}
		I & E & B & I
		\arrow["s"', from=1-2, to=1-1]
		\arrow["p", from=1-2, to=1-3]
		\arrow["t", from=1-3, to=1-4]
		\end{tikzcd}
	\end{equation}
	Any polynomial 2-monad is cartesian.
	If furthermore \(I\) is discrete and \(p\) is a split fibration and $t$ is a split opfibration, then \(T\) is opfamilial by Remark 4.4.6 of \cite{weberInternalAlgebraClassifiers2015} and therefore preserves discrete opfibrations by Theorem 6.2 of \cite{weberFamilial2functorsParametric2007}.
	Any such polynomial 2-monad is therefore fully opfibrantly cartesian.

	For any \(I\), \(\Cat / I\) has discrete opfibration classifier given by \(I \times \Set \to I\) (by Theorem 4.6 of \cite{weberYonedaStructures2toposes2007}).
	We may therefore apply \cref{mc-0003}.%
\end{example}

\emph{Familial} 2-monads are, instead, rarely opfibrantly cartesian.
Indeed, this fails already for $\mathrm{Fam}$: it is cartesian but it does not preserve discrete opfibrations (the reader can readily verify this on $\Set_\bullet \to \Set$).
Its dual, $\mathrm{Fam}(-\op)\op$, instead works as it is opfamilial.

\begin{example}[{Orthogonal factorization systems}]\label[example]{djm-00KK}
	The walking arrow \(\Delta [1]\) is, like any category, a comonoid with respect to the cartesian product.
	Therefore, the 2-functor \(X \mapsto X^{\Delta [1]} : \Cat \to \Cat\) is a 2-monad; as observed in \cite{korostenskiFactorizationSystemsEilenbergMoore1993}, its pseudoalgebras\footnote{The authors of \cite{korostenskiFactorizationSystemsEilenbergMoore1993} say `normal' but actually they sneak in a normalization in §2.2.} are orthogonal factorization systems.
	This is a fun example of a 2-monad which is opfibrantly cartesian but not cartesian (despite preserving pullbacks).

	Indeed, as a right adjoint, \((-)^{\Delta [1]}\) preserves discrete opfibrations and pullbacks.
	For a map \(p : E \to B\), a naturality square for \(i\) is a pullback just when \(p\) reflects identities and the naturality square for \(m\) is a pullback when \(p\) has unique lifting of factorizations---discrete opfibrations satisfy both these properties.

	The algebra structure on \(\Set\) classifies the discrete opfibration \(u^{\Delta [1]} : \Set_{\ast }^{\Delta [1]} \to \Set^{\Delta [1]}\) whose fibers are (isomorphically) given by the elements of the domain of an arrow in \(\Set\).
	That is, the algebra structure is \(\mathrm{dom} : \Set^{\Delta [1]} \to \Set\), or in other words the strict factorization system \((\id, \mathrm{all})\).

	Therefore, \((\Set, (\id, \mathrm{all}))\) classifies strict discrete opfibrations of strict factorization systems in the 2-category of strict factorization systems and \emph{lax} morphisms.
\end{example}

\begin{example}[{Bénabou's construction}]\label[example]{djm-00KE}
	We already observed in \cref{sec:intro} how \cref{mc-0003-pseudo} can be used to prove \(\mathbb {S}\mathrm{pan}(\Set)\) classifies tight discrete opfibrations with small fibers in the 2-category of double categories and lax double functors.

	We can use this fact to deduce a generalization of the Grothendieck construction, usually attributed to Bénabou.
	Let $f:A \to B$ be a functor in $\Cat$.
	There is an embedding $L:\Cat \to \catfont{Dbl}_{\lax}$ which makes categories double by putting all their morphisms in the loose direction.
	It is trivial to observe that the double functor $Lf : LA \to LB$ so obtained from $f$ is a small discrete opfibration of double categories.
	Therefore, we get a classification map $\class Lf : LB \to \mathbb {S}\mathrm{pan}(\Set)$, which is equivalent to a normal lax functor $B \to \mathbb{P}\mathrm{rof}$.
\end{example}

We will now exemplify the \emph{iterability} property of \cref{mc-0003}.
Indeed, once established that \(\PsAlg_{\lax,t\pseudo}(T)\) is a representable plumbus\footnotemark, given a new \(S\) opfibrantly cartesian over it---or equivalently, given \(S\) distributing over \(T\) \cite{walkerPresentationsPseudodistributiveLaws2025} and still opfibrantly cartesian \emph{qua} 2-monad over \(\PsAlg_{\lax,t\pseudo}(T)\)---we can conclude \(\PsAlg_{\lax,t\pseudo}(ST)\) is a representable plumbus too.
\footnotetext{The following applies to lax $T$-algebras too, though one must adjust the directions of the structure morphisms of the distributive law.}

Note that, given a distributive law \(\lambda : TS \Rightarrow ST\), where \(T\) and \(S\) are both opfibrantly cartesian 2-monads, then the resulting composite \(ST\) is also opfibrantly cartesian as soon as \(\lambda \) is an opfibrantly cartesian transformation (i.e. its naturality squares are pullbacks at the small discrete opfibrations).
Equivalently, for the lift \(S^T\) to still be opfibrantly cartesian on \(\PsAlg_{\lax,t\pseudo}(T)\) it suffices for \(S\) to be opfibrantly cartesian on \(\K\) and to \emph{preserve tightness for \(T\)-morphisms}, i.e. the property of being pseudo, since all the other conditions depend on pullbacks, which are computed in \(\Alg_{\lax,t\pseudo}(T)\) just as in \(\K\).

\begin{example}[{Monoidal factorization systems}]\label[example]{mc-003O}
	Consider the 2-monad \(S=(-)^{\Delta [1]}\) described in \cref{djm-00KK}, and observe that \(S\) lifts to \(\catfont{Sym}\catfont{Mon}_{\lax}(\Cat)\), which is a representable plumbus by \cref{djm-00KC}.
	Observe \(S\) is still opfibrantly cartesian, since \(S\) preserve the pseudomonoidality of the functors it is applied to.
	Therefore we conclude monoidal factorization systems also admit a 2-classifier given by \(1/\Set \xrightarrow{u} \Set\) equipped with the two, compatible, algebra structures we exhibited in \cref{djm-00KK} and \cref{djm-00KC}.
\end{example}

\begin{example}[{Duoidal categories}]\label[example]{mc-003N}
	A duoidal category is a pseudomonoid in \(\catfont{Mon}_{\lax}(\Cat)\).
	The latter 2-category is given by pseudoalgebras over \(\Cat\) so subject to our theorem (\cref{djm-00KC}).
	It is thus a representable plumbus.
	To conclude duoidal categories and lax duoidal functors form a representable plumbus too, we thus need to prove the free monoidal category 2-monad \(T\) preserves pseudomonoidality of monoidal functors, and this is again immediate to verify.
	The resulting 2-classifier is \(1/\Set \xrightarrow{u} \Set\) where $\Set$ is equipped with the duoidal structure \((\times , \times)\).
\end{example}

\begin{example}[{Symmetric monoidal structures}]\label[example]{djm-00L3}
	Suppose that \(\mathbb {T}\) is an \(\F\)-sketch (Definition 5.8 of \cite{arkorEnhanced2categoricalStructures2024}) which only marks pullback squares. This includes for example the sketch for pseudo-categories and many other structures determined by Segal conditions. In this case, we may post-compose by the free symmetric monoidal category monad \(S : \Cat \to \Cat\) to get a monad on the enhanced 2-category \(\catfont{Mod}_{\lax}(\mathbb {T}, \Cat)\) of \(\mathbb {T}\)-models and lax morphisms, since \(S\) is cartesian.

	If furthermore \(\catfont{Mod}_{\lax}(\mathbb {T}, \Cat)\) is equivalent to \(\PsAlg_{\lax,t\pseudo}(T)\) for an opfibrantly cartesian 2-monad \(T : \catfont{Mod}_{\lax}(\mathbb {T}_t, \Cat) \to \catfont{Mod}_{\lax}(\mathbb {T}_t, \Cat)\) on the tights, then \(S\) gives a \(\F\)-monad on \(\PsAlg_{\lax,t\pseudo}(T)\). This does occur for the sketch for pseudo-categories, and we expect it to be quite common.

	If we are in a position to compute a tight discrete opfibration classifier of \(\catfont{Mod}_{\lax}(\mathbb {T}_t, \Cat) = \catfont{Mod}_{s}(\mathbb{T}_t, \Cat)\)---for example, if it is merely a presheaf 2-category---then this tight discrete opfibration classifier lifts to models of $T$ and lax morphisms. For the case that $T$ is the sketch for pseudo-categories, \(\catfont{Mod}_{s}(\mathbb{T}_t, \Cat)\) is isomorphic to the 2-category of graphs of categories; we've already seen that the discrete opfibration classifier here is the graph of spans of sets.

	All told, double categories provide an example of this.
	In particular, the cartesian monoidal structure on \(\mathbb {S}\mathrm{pan}(\Set)\) classifies strict monoidal, strict double discrete opfibrations via lax monoidal, lax double functors.
\end{example}

\section*{Declarations}
\subsection*{Author Contributions}
Both authors made comparable contributions to the preparation of this article.

\subsection*{Funding}
During the preparation of this manuscript, the first author was supported by the Advanced Research + Invention Agency (ARIA) through project code MSAI-PR01-P01; while the second author was supported by the Advanced Research + Invention Agency (ARIA) through project code MSAI-PR01-P14.

\subsection*{Data availability}
No datasets were generated or analysed during the current study.

\subsection*{Conflict of interest}
The authors declare no conflicts of interest.

\bibliography{./main.bib}

\begin{appendices}

\section[Comparison with Koudenburg's theorem]{Comparison with\linebreak{}Koudenburg's theorem}
\label[appendix]{app:koudenburg}
As anticipated in the introduction, Koudenburg's Theorem 8.1 in \cite{koudenburgDoubledimensionalApproachFormal2022} generalizes \cref{mc-0003} in many ways, chiefly by abstracting away small discrete opfibrations in plumbuses to arbitrary loose arrows in an augmented virtual equipment.

In this section we compare our setting and Koudenburg's, and observe the second agrees with the first, albeit its generality makes it challenging to compare precisely.

\begin{warning}
\label[warning]{mc-003Q-warning1}
	The theory of \emph{augmented} virtual double categories and \emph{Yoneda morphisms} has been developed in \cite{koudenburgDoubledimensionalApproachFormal2022} and later expounded in \cite{koudenburgFormalCategoryTheory2024}.
	We invite the reader interested in the following to familiarize themselves with those notions there, since it would take too much space to re-introduce them here.
	We give pointers to \cite{koudenburgDoubledimensionalApproachFormal2022} where necessary.

	However, knowledge of virtual double categories and equipments (in the sense of \cite{cruttwellUnifiedFrameworkGeneralized2010}) suffices to read most of what follows.
\end{warning}

\begin{warning}
\label[warning]{mc-003Q-warning2}
	In the setting of double categories \emph{et similia} we adopt the tight/loose convention to refer to 1-cells.
	The reader is warned not to confuse this with the terminology for enhanced 2-categories we used in the above---we promise to avoid ambiguity.
\end{warning}

The definition below instantiates Definition 4.2 in \cite{koudenburgFormalCategoryTheory2024} for \(A=1\):

\begin{definition}
\label[definition]{mc-0036}
	In an augmented virtual double category with tight terminal object \(1\), a \textbf{Yoneda morphism} (for said object \(1\)) is a tight morphism \(1 \xrightarrow {\tau } \Omega \) such that for every loose arrow \(p : 1 \profto B\) there exists a unique \(\class p:B \to  \Omega \) making the conullary square below cartesian:
	\begin{equation}
		\begin{tikzcd}
		1 && B \\
		& \Omega
		\arrow[""{name=0, anchor=center, inner sep=0}, "p"{inner sep=.8ex}, "\shortmid "{marking}, no head, from=1-1, to=1-3]
		\arrow["\tau "', from=1-1, to=2-2]
		\arrow["{\exists ! \class p}", dashed, from=1-3, to=2-2]
		\arrow["{\text {cart}}", between={0.3}{0.7}, Rightarrow, from=0, to=2-2]
		\end{tikzcd}
	\end{equation}
	A Yoneda morphism is \textbf{good} when \(\tau \) is companiable.
\end{definition}

Intuitively, the cartesianness of such a square amounts to a hom-wise isomorphism \(p(b) \cong \Omega (\tau , \class p(b))\).

\begin{remark}
\label[remark]{mc-003Q-remark1}
	The conullary `square' appearing in \cref{mc-0036} above is exactly what the \emph{augmented} in \emph{augmented virtual double category} affords.
	Note every virtual double category with units is augmented by taking conullary squares to be counary squares with bottom a unit.
	Augmentation is essentially a way to deal with issues of sizes, alternative to the more common one (employed, for instance, here) of keeping track of a class of `small' (our \emph{small}) arrows and objects.
	An exemplified discussion of this approach can be found in \cite{koudenburgAugmentedVirtualDouble2020}.
\end{remark}

In this setting, the augmented virtual double category of \(w\)-weak algebras and \(w'\)-weak morphisms (for \(w,w' \in \{\)lax, colax, pseudo\(\}\)) of an augmented virtual monad \(T\), denoted \(\dblAlg^{w}_{w'}(T)\), has been thoroughly described by Koudenburg in §6 of \cite{koudenburgDoubledimensionalApproachFormal2022}, along with numerous examples including monoidal and double categories with their respective notion of profunctor.

Given any well-representable plumbus \(\K\), there is an associated augmented virtual equipment \(\Mod \left (\K\right )\) of tight two-sided discrete fibrations---also called \emph{modules}---as defined e.g. in Corollary 8.1.14 in \cite{riehlElements2022}.

However, we need to restrict it to \emph{small} discrete opfibrations somehow.

\begin{definition}
\label[definition]{mc-003Q-small-modules}
	Call \textbf{small} a discrete object \(X \in \K\) for which \(X \to 1\) is a small discrete opfibration.
	Now say \(A \xleftarrow {q} E \xrightarrow {p} B\), a module in \(\K\), is \textbf{small} when for each \(a:X \looseto  A\) with \(X\) small, the leg \(p_a\) in the simultaneous pullback below is a small discrete opfibration:
	\begin{equation}
		\begin{tikzcd}
		X && {E_a} && B \\
		A && E && B
		\arrow["a"', squiggly, from=1-1, to=2-1]
		\arrow["{q_a}"', two heads, from=1-3, to=1-1]
		\arrow["{p_a}", two heads, from=1-3, to=1-5]
		\arrow["\lrcorner "{anchor=center, pos=0.125, rotate=-90}, draw=none, from=1-3, to=2-1]
		\arrow[from=1-3, to=2-3]
		\arrow["\lrcorner "{anchor=center, pos=0.125}, draw=none, from=1-3, to=2-5]
		\arrow[equals, from=1-5, to=2-5]
		\arrow["q", two heads, from=2-3, to=2-1]
		\arrow["p"', two heads, from=2-3, to=2-5]
		\end{tikzcd}
	\end{equation}
\end{definition}

We define thus the \textbf{augmented virtual equipment of small modules} \(\sMod(\K)\) to be the evident restriction of \(\Mod \left (\K\right )\).
In particular, its tights are the morphisms of $\K$.

Clearly, \(\sMod(\K)\) contains the small discrete opfibrations over \(B\) as the loose arrows \(1 \profto B\), and moreover their pullback correspond to restriction (see \cite{riehlElements2022}, Proposition 8.2.1):

\begin{equation}
	\begin{tikzcd}
	{E_b} & E \\
	X & B
	\arrow[from=1-1, to=1-2]
	\arrow["{p_b}"', two heads, from=1-1, to=2-1]
	\arrow["\lrcorner "{anchor=center, pos=0.125}, draw=none, from=1-1, to=2-2]
	\arrow["p", two heads, from=1-2, to=2-2]
	\arrow["b"', from=2-1, to=2-2]
	\end{tikzcd}
	\quad \leftrightsquigarrow \quad
	\begin{tikzcd}
	1 & 1 \\
	X & B
	\arrow[equals, from=1-1, to=1-2]
	\arrow[""{name=0, anchor=center, inner sep=0}, "{p_b}"'{inner sep=.8ex}, "\shortmid "{marking}, from=1-1, to=2-1]
	\arrow[""{name=1, anchor=center, inner sep=0}, "p"{inner sep=.8ex}, "\shortmid "{marking}, from=1-2, to=2-2]
	\arrow["b"', from=2-1, to=2-2]
	\arrow["{\text {cart}}", between={0.3}{0.7}, Rightarrow, from=0, to=1]
	\end{tikzcd}
\end{equation}

Thus, by working in arbitrary augmented virtual double category with suitable restrictions, we generalize this situation.

Note the \emph{2-classifier} is replaced by the loose \(u=\Omega (\tau , 1)\), i.e. the companion of \(\tau \), i.e. the loose point classified by \(\id_\Omega \)---if it exists, that is when \(\tau \) is good.
In turn, \(u\) classifies loose points by cartesian squares:
\begin{equation}
	\begin{tikzcd}
	1 && \Omega  \\
	& \Omega
	\arrow[""{name=0, anchor=center, inner sep=0}, "u","\shortmid "{marking}, dashed, no head, from=1-1, to=1-3]
	\arrow["\tau "', from=1-1, to=2-2]
	\arrow[equals, from=1-3, to=2-2]
	\arrow["{\text {cart}}", between={0.3}{0.7}, Rightarrow, from=0, to=2-2]
	\end{tikzcd}
	\qquad \qquad
	\begin{tikzcd}
	1 && B \\
	1 && \Omega
	\arrow[""{name=0, anchor=center, inner sep=0}, "p"{inner sep=.8ex}, "\shortmid "{marking}, no head, from=1-1, to=1-3]
	\arrow[equals, from=1-1, to=2-1]
	\arrow["{\exists {!} \class p}", dashed, from=1-3, to=2-3]
	\arrow[""{name=1, anchor=center, inner sep=0}, "u"'{inner sep=.8ex}, "\shortmid "{marking}, from=2-1, to=2-3]
	\arrow["{\text {cart}}", between={0.3}{0.7}, Rightarrow, from=0, to=1]
	\end{tikzcd}
\end{equation}

It is hard to find an exact match between the enhancement of \(\K\) and structure on \(\sMod(\K)\).
Part of the reason is that such structure is largely irrelevant in the latter, and, in practice, replaced by companiability.
Here are two exemplar instances of this phenomenon.

First, if we have a monad \(T\) on \(\sMod(\K)\), Lemma 7.4 in \cite{koudenburgDoubledimensionalApproachFormal2022} characterizes the companiable maps in \(\dblAlg^{w}_{\lax}(T)\) as those lax \(T\)-morphisms which are (1) companiable in \(\sMod(\K)\) and (2) are in fact pseudo.
Compare this with the enhancement described in \cref{def:lax-alg}.

Secondly, the companion of a tight morphism \(f:A \to  B\) in  is given by the representable module (see \cite{riehlElements2022}, Example 8.2.3)
\begin{equation}
	\begin{tikzcd}
	& {f/B} \\
	A && B
	\arrow["{\partial_0}"', from=1-2, to=2-1]
	\arrow["{\partial_1}", from=1-2, to=2-3]
	\end{tikzcd}
\end{equation}
If $\K$ only admits l-rigged commas (\cref{mc-002L}) then $f/B$ exists only when $f$ is tight.
Thus, as long as this module is small, tightness of \(f\) implies its companiability.

If we accept the analogy between tightness and companiability, the analogue of opfibrantly cartesian 2-monads (as characterized in \cref{lemma:opfib-cart}) are simply monads \((T,i,m)\) on the augmented virtual double category $\sMod(\K)$ such that
\begin{description}
	\item[C'.] if the classifying map of \(p : 1 \profto B\) is companiable, so is the classifying map \(1 \mathrel {\overset {{!}^*}{\profto}} T1 \mathrel {\overset {Tp}{\profto}} TB\).
\end{description}
Ineed, remarkably, the properties (A') and (B') of an opfibrantly cartesian 2-monad automatically hold in this setting since (1) by definition, an augmented virtual double monad sends loose arrows to loose arrows and (2) restrictions are preserved by any functor of augmented virtual double categories (this is Proposition 6.8 in \cite{shulmanFramedBicategoriesMonoidal2008} and Corollary 8.2 in \cite{koudenburgAugmentedVirtualDouble2020}).
For $T$ to be \emph{fully} opfibrantly cartesian it suffices to add the analogue axiom to (D'):
\begin{description}
	\item[D'.] the naturality squares of its structure maps are cartesian at all loose arrows \(1 \profto  B\).
\end{description}

Finally, let us observe that cartesian \(T\)-morphisms (in the sense of \cref{def:cart-defect}) are those whose conjoint exists in the equipment of algebras and it's given by a cartesian (in the sense of equipments) square (to make sense of this we suggest to refer to the definition of loose arrow of algebras from \cite{koudenburgDoubledimensionalApproachFormal2022}), i.e.
\begin{equation}
	\begin{tikzcd}
	TA & TB \\
	A & B
	\arrow["Tf", from=1-1, to=1-2]
	\arrow["\alpha "', from=1-1, to=2-1]
	\arrow["\lrcorner "{anchor=center, pos=0.125}, draw=none, from=1-1, to=2-2]
	\arrow["\beta ", from=1-2, to=2-2]
	\arrow["f"', from=2-1, to=2-2]
	\end{tikzcd}
	\quad \leftrightsquigarrow \quad
	\begin{tikzcd}
	TB & TA \\
	B & A
	\arrow[""{name=0, anchor=center, inner sep=0}, "{TB(1,f)}"{inner sep=.8ex}, "\shortmid "{marking}, no head, from=1-1, to=1-2]
	\arrow["\beta "', from=1-1, to=2-1]
	\arrow["\alpha ", from=1-2, to=2-2]
	\arrow[""{name=1, anchor=center, inner sep=0}, "{B(1,f)}"'{inner sep=.8ex}, "\shortmid "{marking}, from=2-1, to=2-2]
	\arrow["{\text {cart}}", between={0.3}{0.7}, Rightarrow, from=0, to=1]
	\end{tikzcd}
\end{equation}

Let's now state Koudenburg's theorem (Theorem 8.1 in \cite{koudenburgDoubledimensionalApproachFormal2022}) specialized for the case \(A=1\) and to target pseudo-\(T\)-algebras rather than colax ones.

\begin{theorem}[{Koudenburg's Yoneda Lifting Theorem}]\label[theorem]{mc-003Q-koudenburg-theorem}
	Let \(T = (T, i, m)\) be a monad on an augmented virtual double category \(\mathbb {K}\).
	Consider a good Yoneda morphism \(\tau \colon 1 \to \Omega \) admitting the companion \(u = \tau _*\) in \(\mathbb {K}\).
	Letting \({!}:T1 \to  1\), assume that
	\begin{enumerate}
		\item  the conjoint \({!}^*\) exists,
		\item  the right pointwise composite (Definition 3.17 in \cite{koudenburgDoubledimensionalApproachFormal2022}) \(({!}^* \odot  T u)\) exists,
		\item  the cell \(m_p\) is pointwise left \((\tau {!})\)-exact (Definition 5.15 in \cite{koudenburgDoubledimensionalApproachFormal2022}) for each \(p \colon  1 \profto B\),
		\item  the cell \(m_{T u}\) is pointwise left \((\tau {!} \circ m_1)\)-exact, while \(i_{u}\) is pointwise left \((\tau {!})\)-exact,
		\item \(T\) preserves any unary cartesian cell with \(u\) as (loose) target,
		\item  the \(T\)-image of the cocartesian cell defining \(({!}^* \odot Tu)\) is left \((\tau {!})\)-exact.
	\end{enumerate}

	Then the morphism \(\omega  := ({!}^* \odot  T u)^\lambda  \colon  T \Omega  \to  \Omega \) dashed below extends to a pseudo-\(T\)-algebra structure \((\Omega , \omega )\) on \(\Omega \):

	\begin{equation}
		\begin{tikzcd}[column sep=scriptsize]
		1 && T1 && {T\Omega } \\
		&& {\Omega }
		\arrow["{{!}^*}"{inner sep=.8ex}, "\shortmid "{marking}, from=1-1, to=1-3]
		\arrow["\tau "', from=1-1, to=2-3]
		\arrow["{Tu}"{inner sep=.8ex}, "\shortmid "{marking}, from=1-3, to=1-5]
		\arrow["{\text {cart}}", between={0.2}{0.7}, Rightarrow, from=1-3, to=2-3]
		\arrow["\omega ", dashed, from=1-5, to=2-3]
		\end{tikzcd}
	\end{equation}

	With respect to this structure \(\tau \colon 1 \to \Omega \) admits a pseudo-\(T\)-morphism structure that makes it into a good Yoneda embedding both in \(\dblAlg^{\pseudo}_{\lax}(T)\).
\end{theorem}

Observe that, when instantiated in the augmented virtual equipment \(\sMod(\K)\) of small modules in a plumbus introduced above, the above yields a definition of \(\omega \) which is substantially the same, i.e. the map classifying \(Tu\).

As for the assumptions, finding an exhaustive mapping is challenging, but we believe there is a rough correspondence as follows.\footnote{We are indebted to Seerp Roald Koudenburg for his help in unraveling such correspondence.}

Assumption (1) is met since the augmented virtual equipment of modules admits such conjoints: \(!^*\) is the terminal module:
\begin{equation}
	\begin{tikzcd}[sep=scriptsize]
	& T1 \\
	1 && T1
	\arrow["{!}"', from=1-2, to=2-1]
	\arrow[equals, from=1-2, to=2-3]
	\end{tikzcd}
\end{equation}
In the definition of representable plumbus (\cref{mc-0045}), we ask for \(\id_1\) to be small and since \(T\) preserves smallness, the above is a valid loose arrow in \(\sMod(\K)\).
Likewise, (2) is trivially met since the required composite exists by triviality of \(!^*\),
As for (3) and (4), they assert that \(i\) and \(m\) satisfy certain pasting properties that correspond to the pasting properties afforded by our cruder hypothesis that \(i\) and \(m\) be cartesian at tight discrete opfibrations.
Then, in light of the above observation regarding restrictions in \(\sMod(\K)\), (5) implies \(T\) preserves pullbacks of the 2-classifier \(u\).
Similarly, (6) can be taken to correspond to the requirement \(T\) preserves pullbacks of the 2-classifier, which is condition (B) of \cref{def:opfib-cart}.

\end{appendices}
\end{document}